\documentclass{article} 
\usepackage{macros}
\usepackage{natbib}
\bibpunct{(}{)}{;}{a}{,}{,}
\newcommand{\BlackBox}{\rule{1.5ex}{1.5ex}}  % end of proof
\newenvironment{proof}{\par\noindent{\bf Proof\ }}{\hfill\BlackBox\\[2mm]}
\newtheorem{example}{Example} 
\newtheorem{theorem}{Theorem}
\newtheorem{lemma}[theorem]{Lemma} 
\newtheorem{proposition}[theorem]{Proposition} 
\newtheorem{remark}[theorem]{Remark}
\newtheorem{corollary}[theorem]{Corollary}
\newtheorem{definition}[theorem]{Definition}
\allowdisplaybreaks

\title{A Wasserstein distance approach for concentration of \\ empirical risk estimates}

\author{
	Prashanth L. A.\\
{\normalsize Indian Institute of Technology Madras}\\
{\normalsize \texttt{prashla@cse.iitm.ac.in}}
	\and 
	Sanjay P. Bhat\\
{\normalsize Tata Consultancy Services Limited} \\
{\normalsize \texttt{sanjay.bhat@tcs.com}}
}
\date{}

\begin{document}
	
	\maketitle

	\begin{abstract}
		This paper presents a unified approach based on Wasserstein distance to derive concentration bounds for empirical estimates for two  broad classes of risk measures defined in the paper. The classes of risk measures introduced include as special cases  well known risk measures from the finance literature such as conditional value at risk (CVaR), optimized certainty equivalent risk, spectral risk measures, utility-based shortfall risk, cumulative prospect theory (CPT) value, rank dependent expected utility and distorted risk measures. Two estimation schemes are considered, one for each class of risk measures. One estimation scheme involves applying the risk measure to the empirical distribution function formed from a collection of i.i.d. samples of the random variable (r.v.), while the second scheme involves applying the same procedure to a truncated sample. The bounds provided apply to three popular  classes of distributions, namely sub-Gaussian, sub-exponential and heavy-tailed distributions. The bounds are derived by first relating the estimation error to the Wasserstein distance between the true and empirical distributions, and then using recent concentration bounds for the latter. Previous concentration bounds are available only for specific risk measures such as CVaR and CPT-value. The bounds derived in this paper are shown to either match or improve upon previous bounds in cases where they are available. The usefulness of the bounds is illustrated through an algorithm and the corresponding regret bound for a stochastic bandit problem involving a general risk measure from each of the two classes introduced in the paper. 
	\end{abstract}

	\section{Introduction}
	\label{sec:introduction}
	Concentration of sample averages has received a lot of attention in statistics. Sample averages are usually used for estimating the expectation of a random variable (r.v.), and classic inequalities, such as those of Hoeffding and Bernstein, provide the necessary concentration bounds. However, the expected value has several shortcomings in the context of risk-sensitive optimization \citep{allais53,ellsberg61,kahneman1979prospect,rocka}, and several measures have been proposed in the literature to capture the notion of risk in a practical application. Unlike the case of expected value, concentration bounds are either not available, or not optimal, for the estimation of several risk measures. 

In this paper, we consider the estimation of the following popular risk measures: optimized certainty equivalent (OCE) risk \citep{bental1986oce} which includes Conditional Value-at-Risk (CVaR) \citep{rocka} as a special case, spectral risk measure (SRM) \citep{acerbi2002spectral}, utility-based shortfall risk (UBSR) \citep{follmer2002convex}, cumulative prospect theory (CPT) \citep{tversky1992advances}, rank-dependent expected utility (RDEU) \citep{quiggin2012generalized} and distorted risk measures (DRM) \citep{denneberg1990distorted}.  
CVaR is popular in financial applications, where it is necessary to minimize the worst-case losses, say in a portfolio optimization context. 
CVaR is a special instance of the class of OCE risk measures \citep{lee2020oce} as well as SRMs \citep{acerbi2002spectral}. CVaR is an appealing risk measure because it is coherent \citep{artzner1999coherent}, and spectral risk measures retain this property. 
UBSR belongs to the family of convex risk measures \citep{follmer2002convex}, which generalizes the class of  coherent risk measures. 
CPT value is a risk measure that is useful for modeling human preferences. RDEU is a closely-related measure, since it shares with CPT the idea of employing a weight function to distort underlying probabilities. 

We provide a novel categorization of risk measures based on their continuity properties in the space of distributions equipped with the Wasserstein metric. To elaborate, we define two broad types of risk measures, say \tone\ and \ttwo, and show that, under suitable conditions, (i) CVaR, SRM and UBSR are \tone; and (ii) CPT, RDEU and DRM are \ttwo. Such a categorization provides a uniform framework for estimating the aforementioned risk measures as well as deriving concentration bounds for the estimates, which is the primary focus of this work. 

The need for concentration bounds for estimation of risk measures is motivated by the fact that, while information about the underlying distribution is typically unavailable in practical applications, one can often obtain i.i.d. samples from the distribution. Our aim is to estimate the chosen risk measure using these samples, and derive concentration bounds for these estimates. We consider this problem of estimation in the broader context of \tone\ and \ttwo\ risk measures. In the former case, we examine the  estimator obtained by applying the risk measure to the empirical distribution constructed from an i.i.d. sample. On the other hand, in the case of a \ttwo\ risk measure, we consider an estimator that is obtained using the truncated empirical distribution.  Our estimator for \tone\ risk measure, when specialized to the case of CVaR, coincides with the one that is already available in the literature. On the other hand, the general \tone\ estimator applied to SRM and UBSR leads to novel estimators. Next, in the case of CPT, the estimator in literature does not involve truncation, while our \ttwo\ estimator does. For the case of RDEU, our estimator is novel, to the best of our knowledge. 

We derive concentration bounds for estimators of \tone\ and \ttwo\ risk measures for three general classes of distributions, namely sub-Gaussian, sub-exponential and heavy-tailed distributions that satisfy a higher-moment bound. We achieve this in a novel manner by relating the estimation error to the Wasserstein distance between the empirical and true distributions, and then using  concentration bounds for the latter. We specialize the bounds obtained for the case of a \tone\ risk measure to provide concentration results for empirical versions of OCE (with CVaR as a special case), SRM and UBSR. We perform a similar exercise to obtain concentration results for CPT and RDEU, using the general result for \ttwo\ risk measures. 

We now summarize our results when the underlying distribution is sub-Gaussian, which is a popular class of distributions with possibly unbounded support.
\begin{enumerate}[label=(\arabic*)]
	\item For the case of CVaR, we provide a tail bound of the order $O\left(\exp\left(-c n \epsilon^2\right)\right)$, where $n$ is the number of samples, $\epsilon$ is the accuracy parameter, and $c$ is a universal constant. Our bound matches the rate obtained for distributions with bounded support in \citet{brown2007large}, and features improved dependence on $\epsilon$ as compared to the one derived for sub-Gaussian distributions in \citet{Kolla}. Further, unlike our results, the latter work imposes a minimum growth assumption on the underlying distribution. 
	\item Tail bounds of the order $O\left(\exp\left(-c n \epsilon^2\right)\right)$ are shown to hold for any SRM having a bounded risk spectrum, and any UBSR with a Lipschitz utility function. Unlike CVaR, the estimators as well as the concentration results for SRM and UBSR are novel. A trapezoidal rule-based estimator has been proposed for SRM in \cite{pandey2019estimation}, and the bounds there are under more stringent assumptions as compared to the one we derive for SRM.  
	\item For the case of CPT-value, we obtain an order $O\left(\exp\left(-c n \epsilon^2\right)\right)$ bound for the case of distributions with bounded support, matching the rate in \citet{prashanth2018tac}. For the case of sub-Gaussian distributions, we provide a bound that has an improved dependence on the number of samples $n$, as compared to the corresponding bound derived by \citet{prashanth2018tac}.
	Similar bounds are shown to hold for any RDEU involving a utility function with derivative bounded above and bounded away from 0. 
	\item The results outlined above for various risk measures in the sub-Gaussian case rely on a concentration bound on the Wasserstein distance between empirical and true distributions, which we derive. This bound, with explicit constants, may be of independent interest.
	\item As a minor contribution, our concentration bounds open avenues for bandit applications, and we illustrate this claim by considering a risk-sensitive bandit setting, with any \tone/\ttwo\ risk measure governing the objective. We use the concentration bounds for \tone/\ttwo\ risk measures to derive regret bounds for such a bandit problem in two cases: first, when the underlying arms' distribution is assumed to be sub-Gaussian and the risk measure is \tone, and second, when the arms' distribution has bounded support and the risk measure is \ttwo.   Previous works (cf. \citet{galichet2015thesis,aditya2016weighted}) consider CVaR and CPT optimization in a bandit context, with arms' distributions having bounded support. In contrast, we consider unbounded, albeit sub-Gaussian distributions in the broader \tone\ class of risk measures.
\end{enumerate}

\begin{table}
	\caption{Summary of the bounds on $\Prob{ \left|\rho_{n} - \rho(X)\right| > \epsilon}$, for various choices of risk measure $\rho(X)$. Here $X$ is either a bounded, or sub-Gaussian, or sub-exponential r.v., and $\rho_n$ is an estimate of $\rho(X)$ using $n$ i.i.d. samples. The corollaries accompanied by a $\ast$ symbol feature improved bounds in comparison to those in the literature. }
	\label{tab:summary}
	\centering
	\begin{tabular}{c|c|c|c}
		\toprule
		\textbf{Risk measure} & \textbf{Bounded support} & \textbf{Sub-Gaussian} & \textbf{Sub-exponential}\\\midrule
		Conditional  & \citep{wang2010deviation} & \citep{prashanth2019concentration} & \citep{prashanth2019concentration}\\
		Value-at-Risk&  & Corollaries \ref{prop:oce-t1-conc} and \ref{cor:oce-subgauss} & Corollary \ref{cor:cvar-subexp}\\\midrule
		Spectral risk  & \citep{pandey2019estimation} & \citep{pandey2019estimation} & \citep{pandey2019estimation}\\
		measure &&Corollaries \ref{specconcprop} and \ref{cor:spec-subgauss}\ $\ast$ &Corollary \ref{cor:spec-subexp}\ $\ast$\\\midrule
		Utility-based & Corollary \ref{cor:ubsr-subgauss} & Corollaries \ref{ubsrconcprop} and \ref{cor:ubsr-subgauss} & Corollary \ref{cor:ubsr-subexp}\\
		shortfall risk &&&\\\midrule
		Cumulative prospect   & \citep{prashanth2018tac}&  \citep{prashanth2018tac}& Corollary \ref{cor:cpt-subexp}\\
		theory &Corollary \ref{cor:cpt-bounded}&Corollaries \ref{cor:cpt-subgauss} and \ref{cor:cpt-subgauss-const}\ $\ast$&\\\midrule
		Rank-dependent  & \multicolumn{3}{c}{\multirow{2}{*}{Bounds for CPT apply in light of Lemma \ref{lemma:rdeu-t2}}}\\
		expected utility  \\\bottomrule
	\end{tabular}
\end{table}

In addition, we also derive concentration bounds for the five risk measures mentioned above, for the case when the underlying distribution is either sub-exponential or heavy-tailed, but satisfying a higher moment bound. To the best of our knowledge, barring CVaR and SRM, tail bounds are not available for the aforementioned classes of distributions for the other risk measures.  

Table \ref{tab:summary} provides a risk-measure-wise summary of the bounds presented in this paper for various classes of distributions along with the relevant previous works in parentheses. After characterizing the two types of risk measures in Section \ref{sec:risk-abstract}, and establishing the type for several popular risk measures, we provide a map of our results in Section \ref{sec:results}   for various risk measures under different assumptions on the underlying distribution through Tables \ref{tab:summary-expec}--\ref{tab:summary-subexp}. In particular, these tables list bounds for risk estimation that hold in expectation (under a bounded higher moment condition), and concentration bounds for sub-Gaussian and sub-exponential distributions, respectively.

\textit{Related work.}
Concentration of empirical CVaR has been the topic of many recent works, cf. \cite{brown2007large,wang2010deviation,thomas2019concentration,Kolla,kagrecha2019distribution,prashanth2019concentration}. The first three references address the case when the underlying distribution has bounded support, while the rest consider unbounded distributions which are either sub-Gaussian or sub-exponential or have a bounded higher-moment. SRM estimation has been considered in \cite{pandey2019estimation}, where the authors provide tail bounds for the case of distributions that have either bounded support, or are sub-Gaussian/sub-exponential. In comparison to these works, the bounds that we derive are under less stringent assumptions. For instance, in \cite{prashanth2019concentration}, the authors require a minimum growth condition on the underlying distribution, while in \cite{pandey2019estimation}, the authors require that the underlying density be bounded above zero, or follow the Gaussian/exponential form. 
While estimation of UBSR has been considered in \cite{dunkel2010stochastic} and \cite{hu2018utility}, concentration bounds for UBSR has not been addressed in previous works, to the best of our knowledge. In the first of these two references, the authors present a stochastic approximation based method for UBSR estimation, while in the second reference, the authors propose a sample-average approximation for UBSR. In both works, there are no non-asymptotic bounds. Instead, an asymptotic normality result is provided. Finally, CPT estimation is the topic of \cite{prashanth2016cumulative} and \cite{prashanth2018tac}. Our estimate for \ttwo\ measures resembles the one in the aforementioned references, except that we employ truncation, which aids in arriving at better concentration results for the sub-Gaussian case. Concentration bounds for empirical versions of RDEU are not available in the literature, to the best of our knowledge, and we fill this gap.

Since CVaR and spectral risk measures are weighted averages of the underlying distribution quantiles, a natural alternative to a Wasserstein-distance-based approach is to employ concentration results for quantiles such as in \cite{Kolla}. While such an approach can provide bounds with better constants, the resulting bounds also involve distribution-dependent quantities (see \cite{prashanth2019concentration}, for instance), and require different proofs for sub-Gaussian and sub-exponential r.v.s. In contrast, our approach provides a unified method of proof.  

In \cite{cassel2018general}, the authors consider a general risk measure along with an abstract norm on the space of distributions such that the risk measure satisfies a polynomial growth bound with respect to the norm, and the norm difference between the empirical and true distributions satisfies a given concentration inequality. There are thus parallels between how continuity properties of the risk measure and concentration inequalities for empirical distributions are combined in \cite{cassel2018general} and in this paper. 
Our definition of \tone\ risk measures is also motivated by the stability criteria in the aforementioned reference.  However, there are some key differences. First, instead of an abstract norm, we work with the Wasserstein metric, and impose a \holdernosp-continuity requirement on the risk measure in the space of distributions. Second, the concentration bounds we derive easily specialize to SRM and UBSR — two risk measures not considered there. Third, CPT and RDEU risk measures do not satisfy the stability criteria of \cite{cassel2018general}. We address this gap by providing the \ttwo\ class of risk measures, which
CPT/RDEU belong to. Fourth, the requirement $2$ in Definition 3 of \cite{cassel2018general}, which specifies the stability criteria,  relates to the concentration of empirical distribution, and is usually satisfied for sub-Gaussian distributions. We study concentration of empirical risk estimates for sub-Gaussian distributions as well as for the more general classes of sub-exponential and heavy-tailed distributions. Note that the aforementioned requirement in \cite{cassel2018general} does not hold for sub-exponential and higher-moment bounded distributions under standard metric such as the sup-norm and the Wasserstein distance. 

The rest of the paper is organized as follows:
In Section \ref{sec:prelims}, we cover background material on Wasserstein distance, and state concentration bounds on Wasserstein distance between empirical and true distribution functions under different assumptions on the tail of the underlying distribution. In Section \ref{sec:risk-abstract}, we define two types of risk measures, and establish the type for five popular risk measures. In Section \ref{sec:results}, we provide a map of the results in Sections \ref{sec:expec-bounds}--\ref{sec:heavytailed}.
In Section \ref{sec:expec-bounds}, we provide bounds in expectation for risk estimation.
In Section \ref{sec:subgauss}, we present concentration bounds for the two general types of risk measures for the case when the underlying distribution is sub-Gaussian. In Sections \ref{sec:subexp} and \ref{sec:heavytailed}, we provide concentration bounds for the cases of sub-exponential and heavy-tailed distributions, respectively. In Section \ref{sec:bandits}, we discuss bandit applications. 
Finally, in Section \ref{sec:conclusions}, we provide the concluding remarks.

	\section{Wasserstein Distance}
	\label{sec:prelims}
	In this section, we introduce the notion of Wasserstein distance, a popular metric for measuring the proximity between two distributions. The reader is referred to Chapter 6 of \citet{villani2008optimal} for a detailed introduction. 

Given two cumulative distribution functions (CDFs) $F_{1}$ and $F_{2}$ on $\R$, let $\Gamma(F_{1},F_{2})$ denote the set of all joint distributions on $\R^{2}$ having $F_{1}$ and $F_{2}$ as marginals. 
\begin{definition}
	\label{definition:wasserstein}
	Given two CDFs $F_1$ and $F_2$ on $\R$ and $p\ge 1$, the $p$-Wasserstein distance between them is defined by
	\begin{align}
		W_{p}(F_1,F_2) \triangleq \inf_{F\in\Gamma(F_{1},F_{2})} \left[\int_{\R^2} |x-y|^p dF(x,y)\right]^{\frac{1}{p}}.
	\end{align}
\end{definition}
In this paper, we will mostly restrict ourselves to the case $p=1$. For convenience, we shall refer to the $1$-Wasserstein distance as `the Wasserstein distance'.

Given $L>0$ and $\kappa>0$, a function $f:\R\rightarrow \R$ is $L$-H\"{o}lder of order $\kappa$ if $|f(x)-f(y)|\leq L|x-y|^{\kappa}$ for all $x,y\in\R$. The function $f:\R\rightarrow \R$ is $L$-Lipschitz if it is $L$-H\"{o}lder of order $1$. Finally, if $F$ is a CDF on $\R$, we define the generalized inverse $F^{-1}:[0,1]\rightarrow \R$ of $F$ by $F^{-1}(\beta)=\inf\{x\in\R: F(x)\geq \beta\}$. In the case where $F$ is strictly increasing and continuous, $F^{-1}$ equals the usual inverse of a bijective function. 

The following lemma provides alternative characterizations of the Wasserstein distance between two CDFs.
\begin{lemma}
	\label{lemma:lipschitz-wasserstein}
	Suppose $X$ and $Y$ are r.v.s having CDFs $F_{1}$ and $F_{2}$, respectively. Then,
	\begin{equation}
		\sup \left|\E(f(X) - \E(f(Y))\right|=W_{1}(F_1,F_2)=\int_{-\infty}^{\infty}|F_{1}(s)-F_{2}(s)|\mathrm{d}s=\int_{0}^{1}|F_{1}^{-1}(\beta)-F_{2}^{-1}(\beta)|\mathrm{d}\beta,\label{lipwass2}
	\end{equation}
	where the supremum is over all functions $f:\R\rightarrow\R$ that are $1$-Lipschitz. 
\end{lemma}
\begin{proof}
	See Section \ref{sec:wass-equiv-proof}.
\end{proof}

For deriving concentration bounds for empirical versions of risk measures, we shall use a bound on the Wasserstein distance between the empirical distribution function (EDF) and the underlying CDF. We first define the EDF of a r.v. $X$ before stating the relevant Wasserstein distance bounds.
Given  i.i.d. samples $X_1,\ldots,X_n$ from the distribution $F$ of a r.v. $X$, the EDF $F_n$ is defined by 
\begin{align}
	F_n\left(x\right) = \frac1{n} \sum_{i=1}^n \indic{X_i \le x}, \ \  \textrm{for all }x\in \R.
	\label{eq:edf}
\end{align}
In the above, $\indic{\cdot}$ denotes the indicator function, i.e., for an event $A$, $\indic{A} = 1$ if $A$ happens, and $\indic{A} = 0$ otherwise.

Concentration bounds on the Wasserstein distance between the EDF of an i.i.d. sample and the underlying CDF from which the sample is drawn, have been derived in \cite{lei2020convergence,weed2019sharp,bolley2007quantitative,boissard2011simple,fournier2015rate}. 
In the following sections, we
shall state  Wasserstein distance bounds for two different classes of unbounded r.v.s.
\subsection{Distributions satisfying an exponential moment bound}
In this section, we state a Wasserstein distance bound assuming that the underlying distribution satisfies an exponential moment bound with an exponent greater than one. Sub-Gaussian distributions are a popular class of distributions that satisfy this assumption. 

\begin{assumption}
\label{ass:c1}
There exist $\beta >1$,  $\gamma> 0$ and $\top>0$ such that $\E\left(\exp\left(\gamma|X|^\beta\right)\right) < \top < \infty$.
\end{assumption}

For the special case of distributions satisfying \ref{ass:c1}, we extract the required Wasserstein concentration bound from  \cite{fournier2015rate}. This result will be used to derive concentration bounds for empirical versions of 
various risk measures in Section \ref{sec:subgauss}. 
\begin{lemma}
	\label{lemma:wasserstein-dist-bound}
	Let $X$ be a r.v. with CDF $F$, and  suppose $X$ satisfies \ref{ass:c1}. Then, for every $\epsilon \ge 0$ and $n\geq 1$, we have
	\[ \Prob{ W_{1}(F_n,F) > \epsilon} \le c_1\left( \exp\left(-c_2 n \epsilon^2\right)\indic{\epsilon\le 1} + \exp\left(-c_3 n \epsilon^{\beta}\right)\indic{\epsilon> 1}\right),\]
	where $c_1, c_2,$ and $c_3$ are constants that depend on the parameters $\beta, \gamma$ and $\top$ specified in \ref{ass:c1}.
\end{lemma}
\begin{proof}
See Section \ref{sec:wasserstein-dist-bound-c1-proof}.
\end{proof}

We next define a sub-Gaussian  r.v., which is a popular sub-class of unbounded r.v.s  satisfying assumption \ref{ass:c1}.
\begin{definition}
	A r.v. $X$ is sub-Gaussian if there exists $\sigma >0$ such that
\begin{align} 
	P( X \ge \epsilon) \le \exp\left(-\frac{\epsilon^2}{2\sigma^2}\right), \textrm{ and } P( X \le -\epsilon) \le \exp\left(-\frac{\epsilon^2}{2\sigma^2}\right), 
	\label{eq:subgauss-1}
\end{align}
for every $\epsilon>0$.
\end{definition}
A sub-Gaussian r.v. $X$ satisfies 
	\begin{align} 
	\E\left(\exp\left(\gamma X^2 \right) \right) \le  2,
	\label{eq:subgauss-expomoment}
\end{align} 
where $\gamma$ is a universal constant multiple of the sub-Gaussianity parameter $\sigma$. Thus, a sub-Gaussian random variable satisfies \ref{ass:c1} with $\beta=\top=2$.

\begin{remark}
Note that our definition of sub-Gaussianity above does not require the r.v. to have mean zero. If the r.v. $X$ has mean zero, then $X$ is sub-Gaussian if and only if
there exists  $\sigma >0$ such that
\[ \E\left(\expo{\lambda X}\right) \le \expo{\dfrac{\lambda^2 \sigma^2}{2}} \text{ for every } \lambda \in \R.\]
The reader is referred to Section 2.5 of \cite{vershynin2018high} for a detailed introduction to sub-Gaussian distributions.
\end{remark}

\begin{remark}
	While sub-exponential distributions also satisfy an exponential moment bound, they do not satisfy \ref{ass:c1}, as the exponent $\beta$  in the sub-exponential case equals one. While Theorem 2 in \citet{fournier2015rate} covers the  case $\beta=1$, the tail bound there is weak as it exhibits a polynomial decay for large deviations. In the next section, we handle the sub-exponential case separately by using a Wasserstein distance bound from the recent work of \citet{lei2020convergence}.
	\label{rem6}
\end{remark}

The following result for sub-Gaussian distributions is an immediate corollary of Lemma \ref{lemma:wasserstein-dist-bound}.
\begin{corollary}
	\label{cor:subgauss-wasserstein-dist-bound}
	Let $X$ be a r.v. with CDF $F$. Suppose that $X$ is sub-Gaussian with parameter $\sigma$. Then, for every $\epsilon \ge 0$ and $n\geq 1$, we have
	\[ \Prob{ W_{1}(F_n,F) > \epsilon} \le c_1 \exp\left(-c_2 n \epsilon^2\right),\]
	where $c_1,$ and $c_2$ are constants that depend on the sub-Gaussianity parameter $\sigma$.
\end{corollary}

	The bound in Corollary \ref{cor:subgauss-wasserstein-dist-bound} does not make the constants explicit, since an expression that makes the dependence of the constants $c_1,c_2$ on the underlying parameters explicit is not available in \citet{fournier2015rate}. From the proof of Theorem 2 there, we were unable to extract the explicit dependence of the constants on the underlying parameters.
	The lack of explicit constants may be of concern in bandit applications. To address this shortcoming, we provide Wasserstein distance bounds with explicit constants for sub-Gaussian r.v.s in the result below. 
	In the next section, we provide such bounds for sub-exponential r.v.s. 

\begin{lemma}\textbf{\textit{(Wasserstein distance bound)}}
	\label{lemma:wasserstein-dist-bound-subgauss}
	Let $X$ be a sub-Gaussian r.v. with parameter $\sigma$. Let $F$ denote the CDF of $X$.  Then, for every $n\geq 1$ and  $\epsilon$ such that $\frac{512\sigma}{\sqrt{n}}<\epsilon < \frac{512\sigma}{\sqrt{n}}+16\sigma\sqrt{\myexp}$, we have
	\[ \Prob{ W_{1}(F_n,F) > \epsilon} \le \exp\left(-  \frac{n}{256\sigma^2\myexp}\left(\epsilon-\frac{512\sigma}{\sqrt{n}}\right)^2\right),\]
	where $\myexp$ is Euler's number.
\end{lemma}
\begin{proof}
	See Section \ref{sec:wasserstein-dist-bound-subgauss}.
\end{proof}

Notice that, unlike Lemma \ref{lemma:wasserstein-dist-bound}, the constants are made explicit in the bound above. 

\begin{remark}
	The bound in the lemma above has a term of the form $\left(\epsilon-\frac{512\sigma}{\sqrt{n}}\right)^2$ inside the exponential function because we first derive a tail bound on the centered error $W_{1}(F_n,F) - \E [W_{1}(F_n,F)]$, and then use $\E [W_{1}(F_n,F)] \le \frac{512\sigma}{\sqrt{n}}$.
	The tail bound uses arguments similar to those employed in establishing the well-known McDiarmid inequality (cf. Theorem 5.1 of \cite{lei2020convergence}), and is of the form
	\[ \Prob{ W_{1}(F_n,F) - \E [W_{1}(F_n,F)] > \tilde\epsilon} \le \exp\left(-\frac{n\tilde{\epsilon}^2}{128  \sigma^2 \myexp}\right), \textrm{ for any } \tilde\epsilon \in (0,16\sigma\sqrt{\myexp}).\]
	The main bound in Lemma \ref{lemma:wasserstein-dist-bound-subgauss} follows by setting $\tilde\epsilon =\left(\epsilon-\frac{512\sigma}{\sqrt{n}}\right)$ in the above inequality, and using $\E [W_{1}(F_n,F)] \le \frac{512\sigma}{\sqrt{n}}$.  The reader is referred to Section \ref{sec:wasserstein-dist-bound-subgauss} for the detailed proof.
\end{remark}

\subsection{Sub-exponential distributions}
The second class of r.v.s that we treat in this paper are sub-exponential. As mentioned in Remark \ref{rem6},   sub-exponential r.v.s satisfy an exponential moment bound with exponent equal to one, and hence these r.v.s do not satisfy \ref{ass:c1}. The reason for treating sub-exponential distributions will be made apparent after presenting the Wasserstein concentration bound below for the sub-exponential case. 
Before presenting this bound, we specify below the condition that characterizes a sub-exponential r.v..

\begin{assumption}
    \label{ass:subexp}
	There exists a $c>0$ such that 
	\begin{equation}	
	\Prob{X > \epsilon} \le \exp(-c\epsilon), \textrm{ and } \Prob{X \le -\epsilon} \le \exp(-c\epsilon),\ \ \forall \epsilon \ge 0.
	\label{subexptail}
\end{equation}
\end{assumption}

A sub-exponential r.v. $X$ satisfies
	\begin{equation}	
		\E\left(\exp\left(c'|X|\right)\right) \le 2, \label{eq:subexpdef}
	\end{equation}
where $c'$ is an universal constant multiple of the constant $c$ from \ref{ass:subexp}.
Further,  if the r.v. $X$ is mean zero, then an equivalent characterization of sub-exponential r.v.s is the following (see \citet[Sec. 2.7]{vershynin2018high}):
there exist  positive constants $\sigma$ and $b$ such that
\begin{equation} \E(\expo{\lambda X}) \le \expo{\dfrac{\lambda^2 \sigma^2}{2}} \text{ for every } |\lambda| < \frac{1}{b}.
\label{bernstein}
\end{equation}

Since \ref{ass:c1} requires $\beta>1$, sub-exponential r.v.s fall outside the class of r.v.s satisfying \ref{ass:c1}. The reason for separately handling sub-exponential r.v.s is that the bound in \cite{fournier2015rate} is not satisfactory for such r.v.s. In particular, the tail bound there exhibits a power law decay for   large values of $\epsilon$. To obtain an exponential decay in the tail bound, we rely on a bound from \cite{lei2020convergence}, which is presented in Lemma \ref{lemma:wasserstein-dist-bound-subexp} below.  In this bound, which is an analogue of Lemma \ref{lemma:wasserstein-dist-bound} for the case of sub-exponential distributions, we make all the constants explicit by tracing through the proofs of Theorem 3.1 and Corollary 5.2 of \cite{lei2020convergence}.

\begin{lemma}\textbf{\textit{(Wasserstein distance bound: sub-exponential case)}}
	\label{lemma:wasserstein-dist-bound-subexp}
	Let $X$ be a r.v. satisfying \ref{ass:subexp} with parameter $c$. Let $F$ denote the CDF of $X$.  Then, for every $n\geq 1$ and  $\epsilon$ satisfying $\epsilon > \frac{384}{c\sqrt{n}}$, we have
	\[ \Prob{ W_{1}(F_n,F) > \epsilon} \le \exp\left(-  \frac{n\left(\epsilon-\frac{384}{c\sqrt{n}}\right)^2}{\frac{32}{c^2} + \frac{4}{c}\left(\epsilon-\frac{384}{c\sqrt{n}}\right)}\right).\]
\end{lemma}
\begin{proof}
	See Section \ref{sec:wasserstein-dist-bound-subexp-proof}.
\end{proof}

%%%%%%%%%%%%%%%%%%%%%%%%%%%%%%%%%%%%%%%%%%%%%%%%%%%%%%%%%%%%%%%%%%%%%%%%%%%%%
%%%%%%%%%%%%%%%%%%%%%%%%%%%%%%%%%%%%%%%%%%%%%%%%%%%%%%%%%%%%%%%%%%%%%%%%%%%%%%%%%%%%%%%%%%%%%%%%%%%%%%%%%%%%%%%%%%%%%%%%%%%%%%%%%%%%%%%%%%%%%%%%%%%%%%%%%%
\subsection{Heavy-tailed distributions}
The third class of distributions that we consider in this paper are those that satisfy a higher-moment bound, as specified by  the following condition:
\begin{assumption}
	\label{ass:heavy}
	There exists $\beta >2$ such that $\E\left(|X|^\beta\right) < \top < \infty$.
\end{assumption}

Distributions satisfying \ref{ass:heavy} fall under the broad class of heavy-tailed distributions (see \cite{nair2012scheduling}), which includes  sub-Gaussian and  sub-exponential distributions, as well as distributions with infinite variance. 

Next, we provide an analogue of Lemma \ref{lemma:wasserstein-dist-bound} for the case of distributions satisfying \ref{ass:heavy}.
\begin{lemma}\textbf{\textit{(Wasserstein distance bound)}}
	\label{lemma:wasserstein-dist-bound-heavy}
	Let $X$ be a r.v. with CDF $F$. Suppose that $X$ satisfies \ref{ass:heavy}. Then, for every $n\geq 1$, $\epsilon \ge 0$ and $\eta\in(0,\beta)$, we have
	\[ \Prob{ W_{1}(F_n,F) > \epsilon} \le c_1\left( \exp\left(-c_2 n \epsilon^2\right)\indic{\epsilon\le 1} 
	+ n\left(n\epsilon\right)^{-(\beta-\eta)}\indic{\epsilon> 1}\right),\]
	where
	$c_1, c_2$ are constants that depend on $\eta$ and the parameters $\beta$ and $\top$ specified in \ref{ass:heavy}. 
\end{lemma}
\begin{proof}
	See Section \ref{sec:wasserstein-dist-bound-heavy-proof}.
\end{proof}

The constants are not available explicitly in the tail bound presented above, as the corresponding result in \cite{fournier2015rate} does not specify constants explicitly, and we could not obtain an explicit expression for the constants as a function of $\beta, \eta$ and $\top$ from the proofs given therein.

	\section{Risk measures and estimators}
	\label{sec:risk-abstract}
	Let $\L$ denote the space of CDFs on $\R$. A {\em risk measure} is simply a map $\rho:\L\rightarrow \R$. With a slight abuse of notation, we will find it convenient to also think of a risk measure as a real-valued function on the set of real-valued r.v.s. More precisely, given a risk measure $\rho$ and a real-valued r.v. $X$ with CDF $F$, we will write $\rho(X)$ to mean $\rho(F)$.  

\subsection{Empirical risk estimators}

Given a risk measure $\rho$ and  CDF $F$, we consider two intuitive empirical estimates of $\rho(F)$ in this paper. The first estimate involves applying $\rho$ to the EDF $F_n$ formed by drawing $n$ samples from $F$. 
To make this precise, we define the estimate $\rho_{n}$ of $\rho(F)$ by
\begin{align}
	\rho_{n}=\rho(F_{n}),
	\label{eq:rho-est} 
\end{align}
where $F_n$ is as defined in \eqref{eq:edf}.

Our second estimate of $\rho(F)$ involves applying $\rho$ to a truncated sample drawn from $F$.  
We introduce some notation to make this precise. Given $\tau>0$, we denote $F|_{\tau}=F\indic{-\tau\leq x<\tau}+\indic{x\geq \tau}$. If $F$ is the CDF of a r.v. $X$, then $F|_{\tau}$ is the CDF of the r.v. $X\indic{-\tau\leq X <\tau}$. 

As before, let $F_n$ denote the EDF of r.v. $X$, formed using i.i.d. samples $X_1,\ldots,X_n$. Our second estimate of the risk measure $\rho(F)$ is then given by  
\begin{align}
	\rho_{n,\tau} \triangleq \rho(F_n|_{\tau}).
	\label{eq:rho-est-trunc} 
\end{align}

We consider two types of risk measures in this paper. The first type satisfies a \holdernosp-continuity requirement in the metric space of distributions under the Wasserstein distance, which is made precise below. The second type, which is presented later, handles some classes of risk measures that may not satisfy a \holder condition.

\subsection{Type-1 \tone\ risk measures }
\label{sec:t1}
\begin{risktype}
	\label{type:one}
Let $(\L,W_1)$ denote the metric space of distributions, with Wasserstein distance as the metric. 
The risk measure $\rho(\cdot)$ is  \holder-continuous on $(\L,W)$, i.e.,  there exists $\kappa\in(0,1]$ and $L>0$ such that, for any two distributions $F, G \in \L$, the following holds:
\begin{align}
	\left| \rho(F) - \rho(G)\right| \le L \left(W_1(F,G)\right)^\kappa. \label{tonedef}
\end{align} 
\end{risktype}

Optimized certainty equivalent (OCE) risk \citep{bental1986oce,bental2007oce}, spectral risk measure \citep{acerbi2002spectral}, and utility-based shortfall risk \citep{follmer2002convex} are three popular families of risk measures that are of type \ref{type:one}. We introduce these risk measures in the following sections and specialize the estimate given in \eqref{eq:rho-est} to the cases of the aforementioned risk measures.

\subsubsection{OCE risk}\label{ocesec}
We adapt the definition of an OCE risk given in \citet{lee2020oce}.  Given a  nondecreasing, convex disutility function $\phi:\R\rightarrow \R$ and a r.v. $X$ such that $\phi(X)$ is integrable, the OCE risk of $X$ determined by $\phi$ is 
    \begin{align}
        \oce(X)\triangleq\inf_{\xi}\left\lbrace \xi + \E\left[ \phi\left(X -\xi\right)\right] \right\rbrace.
        \label{ocedef} 
    \end{align}
   
The following lemma shows that an OCE is a risk measure of type \ref{type:one}. 
\begin{lemma}
	\label{lemma:oce-t1}
	Let $X$ and $Y$ be r.v.s with CDFs $F_{X}$ and $F_{Y}$, respectively, and let $\phi$ be a disutility function as in \eqref{ocedef}. Assume that $\phi$ is $L$-Lipschitz for some $L>0$. Then
	\begin{equation}
		|\oce(X)-\oce(Y)|\leq L W_{1}(F_{X},F_{Y}).
		\label{ocelip}
	\end{equation}
	\label{ocet1lem}
\end{lemma}
\begin{proof}
	See Section \ref{sec:oce-t1-proof}.
\end{proof}

Next, we discuss estimation of an OCE risk from  i.i.d. samples $X_1,\ldots,X_n$, which are drawn from the distribution of $X$. We estimate $\oce(X)$ from such a sample by
\begin{equation}
	\oce_{n} = \inf_{\xi} \left\lbrace \xi + \frac{1}{n}\sum_{i=1}^n\phi( X_i -\xi) \right\rbrace. \label{ocest}
\end{equation}
It is straightforward to see that $\oce_{n}= \oce(Z_n)$, where $Z_n$ is a r.v. with distribution $F_n$. The estimator \eqref{ocest} is thus a special case of the general empirical risk estimate given in \eqref{eq:rho-est}.

Next, we define CVaR, a risk measure that is popular in financial applications.	The CVaR at level $\alpha \in (0,1)$ for an integrable r.v $X$ is defined by 
	\begin{equation}
	\cvar_{\alpha}(X) \triangleq \inf_{\xi} \left\lbrace \xi + \frac{1}{(1-\alpha)}\E\left( X -\xi\right)^+ \right\rbrace, \textrm{ where } (y)^+ = \max\{y,0\}.
	\label{cvardef}
	\end{equation}

It is well known (see \citet{rocka}) that the infimum in the definition of CVaR above is achieved for $\xi=\var(X)$, where $\var(X)=F^{-1}(\alpha)$ is the value-at-risk of the r.v. $X$ at confidence level $\alpha$. Thus CVaR may also be written alternatively as given, for instance, in \citet{Kolla}. In the special case where $X$ has a continuous distribution, $\cvar_{\alpha}(X)$ equals the expectation of $X$ conditioned on the event that $X$ exceeds $\var(X)$. 

A comparison of \eqref{cvardef} and \eqref{ocedef} shows that CVaR at level $\alpha\in(0,1)$ is an OCE risk with the disutility function $\phi(x)=(1-\alpha)^{-1}(x)^{+}$, which is nondecreasing, convex, and $L$-Lipschitz for $L=(1-\alpha)^{-1}$. Applying Lemma \ref{ocet1lem} lets us conclude that CVaR is a risk measure of type \ref{type:one} satisfying 
\begin{equation}
		|\cvar_{\alpha}(X)-\cvar_{\alpha}(Y)|\leq (1-\alpha)^{-1}W_{1}(F_{X},F_{Y})
		\label{cvrlip}
\end{equation}
for any two integrable random variables $X$ and $Y$. 	

We may now apply \eqref{ocest} to estimate CVaR from  i.i.d. samples $X_1,\ldots,X_n$, which are drawn from the distribution of $X$. The resulting estimate of CVaR appeared earlier in  \cite{brown2007large}, and is given by
\begin{equation}
	c_{n,\alpha} = \inf_{\xi} \left\lbrace \xi + \frac{1}{n(1-\alpha)}\sum_{i=1}^n\left( X_i -\xi\right)^+ \right\rbrace. \label{cvarest}
\end{equation}
It is straightforward to see that $c_{n,\alpha}= \cvar_\alpha(Z_n)$, where $Z_n$ is a r.v. with distribution $F_n$. 

\subsubsection{Spectral risk measures}
Spectral risk measures are an alternative generalization of CVaR. Given a weighting function $\phi:[0,1]\rightarrow [0,\infty)$, the spectral risk measure $M_{\phi}$ associated with $\phi$ is defined  by 
\begin{align}
	M_{\phi}(X) = \int_{0}^{1}\phi(\beta)F_X^{-1}(\beta)\mathrm{d}\beta, \label{specdef}
\end{align}
where $X$ is a r.v. with CDF $F_X$. 
If the weighting function, also known as the {\em risk spectrum}, is increasing and integrates to 1, then $M_{\phi}$ is a coherent risk measure like CVaR. In fact, CVaR is itself a special case of \eqref{specdef}, with $\cvar_{\alpha}(X)=M_{\phi}$ for the risk spectrum $\phi=(1-\alpha)^{-1}\indic{\beta\geq \alpha}$
(see \cite{acerbi2002spectral} and \cite{Dowd} for details).
Assuming the r.v. $X$ models losses in a financial application, it is apparent that CVaR treats all losses above a certain threshold equally, by assigning the same weight for $\beta\ge \alpha$. 
One could think of a spectral risk measure with a weight function chosen such that higher losses are given more weight. An example of such a weight function, proposed by \cite{cotter2006extreme}, is $\phi(\beta) = \frac{\varkappa\,\myexp^{-\varkappa(1-\beta)}}{1-\myexp^{-\varkappa}},\ \beta \in [0,1]$, where $\varkappa$ is a constant that controls risk-aversion.
The dual-power risk measure \cite{wang1996premium} employs the weighting function given by $\phi(\beta)= \varkappa (1-\beta)^{\varkappa-1}$, with $\varkappa\ge 1$. 

Our next lemma shows that a spectral risk measure having a bounded weighting function is a risk measure of type \ref{type:one}. The proof is omitted as the lemma follows very easily from the definition \eqref{specdef}, and the last characterization of the Wasserstein distance given in \eqref{lipwass2}. 

\begin{lemma}
	\label{lemma:srm-t1}
	Let $X$ and $Y$ be r.v.s with CDFs $F_{X}$ and $F_{Y}$, respectively, and let $\phi:[0,1]\rightarrow [0,\infty)$ be a weighting function as in \eqref{specdef}. Assume that $\phi(u)\leq K$ for all $u\in[0,1]$. Then 
	\begin{equation}
		|M_{\phi}(X)-M_{\phi}(Y)|\leq KW_{1}(F_{X},F_{Y}).
		\label{specrsklip}
	\end{equation}
	\label{spectt1lem}
\end{lemma}
\begin{remark}
	\label{remark:srm-ph}
	The two examples of the weighting function given above satisfy the boundedness requirement in the lemma above. However, there are spectral risk measures with weighting function that are not bounded. For example, the weighting function of the  proportional hazard transform risk measure \citep{wang1995insurance} is unbounded, and  given by  $\phi(\beta) = \frac{1}{\varkappa} \beta^{\frac{1}{\varkappa}-1}$, with $\varkappa\ge 1$. Such a risk measure may  not fall under the class of \ref{type:one} risk measures. We shall handle  spectral risk measures with unbounded weight functions through the \ref{type:two} class of risk measures introduced in subsection \ref{sec:t2} below. 
\end{remark}	

We now discuss the  estimation of a spectral risk measure from an i.i.d. sample $X_{1},\ldots,X_{n}$, which is drawn from the CDF $F$ of a r.v. $X$. Using the EDF, a natural empirical estimate of the spectral risk measure $M_{\phi}(X)$ of $X$ is 
\begin{equation}
	m_{n,\phi}=\int_{0}^{1}\phi(\beta)F_{n}^{-1}(\beta)\mathrm{d}\beta. \label{empspecdef}
\end{equation}
As in the case of an OCE risk, the estimate defined above is also a special case of  the general estimate given in \eqref{eq:rho-est}.

\subsubsection{Utility-based shortfall risk (UBSR)}
VaR as a risk measure is not popular owing to the fact that it is not sub-additive. CVaR overcomes this limitation, and is a coherent risk measure. Convex risk measures  \citep{follmer2002convex} are  a more general class of measures than coherent risk measures, because sub-additivity and homogeneity imply convexity. UBSR form a popular class of convex risk measures. 

For a r.v. $X$, the utility-based shortfall risk $S_\alpha(X)$ is defined as 
\begin{align}
	S_\alpha(X) = \inf\left\{\xi \in \R \mid \E\left( l(X-\xi)\right)\le \alpha\right\},
	\label{eq:ubsr-def}
\end{align}
where $l:\R\rightarrow \R$ is a utility function. $S_\alpha(X)$ can be seen to be the value  as well as the minimizer of the following constrained minimization problem:
\begin{align} 
	\min_{\xi\in \R} \xi  \qquad\textrm{ subject to }\qquad \E\left( l(X-\xi)\right)\le \alpha.\label{eq:ubsr-equiv-form}
\end{align}

UBSR is a risk measure that generalizes VaR, since one recovers VaR from \eqref{eq:ubsr-def} by employing an indicator function for $l$. While VaR is not a convex risk measure, a suitable choice for the utility function $l$ in \eqref{eq:ubsr-def} can yield a convex risk measure, cf. Section 4.9 of \cite{follmer2016stochastic}.  Furthermore, in contrast to CVaR, UBSR is invariant under randomization \citep{dunkel2010stochastic}. Invariance under randomization  formally means that given two r.v.s $X_1, X_2$ satisfying $S_\alpha(X_i) \le 0, i=1,2$, the  compound r.v. $Z$ that chooses between $X_1$ and $X_2$ using an independent Bernoulli r.v. also satisfies  $S_\alpha(Z)\leq 0$.

For the purpose of showing that UBSR is of type \ref{type:one}, we require that $l$ be Lipschitz. While our results below require no additional assumptions on $l$, we point out that a convexity assumption on $l$ leads to a useful dual representation \citep{follmer2002convex}. Likewise, if the utility is increasing, then the estimation of  UBSR from i.i.d. samples can be performed in a computationally efficient manner (see \cite{hu2018utility}).

The following lemma shows that UBSR is a risk measure of type \ref{type:one}. 
\begin{lemma}
	\label{lemma:ubsr-t1}
	 Suppose the utility function $l$ is non-decreasing, and there exist $K,k>0$ such that $l$ is $K$-Lipschitz and satisfies $l(x_{2})\geq l(x_{1})+k(x_{2}-x_{1})$ for every $x_{1},x_{2}\in\R$ satisfying  $x_{2}\geq x_{1}$. Let $X$ and $Y$ be  r.v.s with CDFs $F_{X}$ and $F_{Y}$, respectively. Then
	\begin{align}
		|S_{\alpha}(X)-S_{\alpha}(Y)|\leq \frac{K}{k} W_{1}(F_{X},F_{Y}).
		\label{ubrlip}
	\end{align}
	\label{ubsrt1lem}
\end{lemma}
\begin{proof}
	See Section \ref{sec:ubsr-t1-proof}.
\end{proof}

Next, we discuss the estimation of UBSR $S_\alpha(X)$. Given $n$ i.i.d. samples $X_{1},\ldots,X_{n}$, UBSR is estimated as the solution of the following constrained optimization problem \citep{hu2018utility}:  
\begin{align}
	\min_{\xi\in \R} \xi  \qquad\textrm{ subject to }\qquad \frac{1}{n}\sum_{i=1}^n  l(X_i-\xi)\le \alpha.
	\label{eq:sr-est}
\end{align}
The problem above can be seen as a sample-average approximation to \eqref{eq:ubsr-equiv-form}. In other words, the solution of  \eqref{eq:sr-est} is the UBSR value of a r.v. distributed according to the EDF  $F_n$ of $X$. Thus, as in the case of OCE risk  and spectral risk measures, the estimation scheme provided above is a special case of the general estimate given in \eqref{eq:rho-est}.

In \cite{hu2018utility}, the authors show that the problem \eqref{eq:sr-est} can be solved efficiently using a bisection method if the utility function $l$ is increasing and an interval containing  $S_\alpha(X)$ is known. The reader is referred to \cite{hu2018utility} for further details.

\subsubsection{Two Examples} \label{examples}
The \ref{type:one} measures seen so far, namely, OCE risk, spectral risk measures, and UBSR all satisfy \eqref{tonedef} with $\kappa=1$. To illustrate the point that the additional generality provided in \eqref{tonedef} is not vacuous, we next present an example of a risk measure that satisfies the definition of a \ref{type:one} measure for some $\kappa<1$, but not for $\kappa=1$. 

\begin{example} \label{ctrex1}
Consider the risk measure $\rho(F) = \int_0^1 (1-F(x))^{\frac{1}{3}}\mathrm{d}x$. To show that $\rho$ is \ref{type:one}, let $F$ and $G$ be two distributions. We have 
\begin{eqnarray*}
|\rho(F)-\rho(G)|&\leq & \int_{0}^{1}|(1-F(x))^{\frac{1}{3}}-(1-G(x))^{\frac{1}{3}}|\mathrm{d}x \\
&\leq & 2^{\frac{2}{3}}  \int_{0}^{1}|F(x)-G(x)|^{\frac{1}{3}}\mathrm{d}x \\
&\leq &  2^{\frac{2}{3}}\left[\int_{0}^{1}|F(x)-G(x)|\mathrm{d}x\right]^{\frac{1}{3}}\leq 2^{\frac{2}{3}}[W_{1}(F,G)]^{\frac{1}{3}},
\end{eqnarray*}
where the last inequality above follows from Lemma \ref{lemma:lipschitz-wasserstein}, the first and third inequalities above follow from Jensen's inequality, and the second inequality follows from the fact that $|a^{\alpha}-b^{\alpha}|\leq 2^{1-\alpha}|a-b|^{\alpha}$ for every $a,b\in\R$ and $\alpha\in(0,1]$ such that $\alpha$ equals the ratio of two odd positive integers (see Fact 2.2.77 in \citet{dsb}). This shows that $\rho$ satisfies \eqref{tonedef} for all $F$ and $G$ with $\kappa=\frac{1}{3}$. 

We claim that there exists no $L>0$ such that $\rho$ satisfies \eqref{tonedef}  with $\kappa=1$ for all distributions. To see this, let $L>0$, and pick $a\in(0,L^{-\frac{3}{2}})$. Let $F$ (respectively, $G$) be the CDF of a random variable that takes the value $a$ (respectively, $0$) almost surely. It is a simple matter to calculate $\rho(F)=a^{\frac{1}{3}}$, $\rho(G)=0$, and $W_{1}(F,G)=a$. Consequently, we have  $\frac{|\rho(F)-\rho(G)|}{W_{1}(F,G)}=a^{-\frac{2}{3}}>L$. Since $L>0$ was chosen arbitrarily, our claim follows. 
\end{example}

To provide motivation for the class of risk measures that we introduce next, we provide an example of a risk measure that is not \ref{type:one}. 

\begin{example} \label{ctrex2}
	Consider the risk measure $\rho(F) = \int_0^\infty \sqrt{1-F(x)}\mathrm{d}x$. To show that $\rho$ is not \ref{type:one}, let $L>0$ and $\kappa\in(0,1]$ be given. Choose $p\in(0,L^{-2})$. Additionally, choose $a=1$ or $a>[Lp^{\kappa-\frac{1}{2}}]^{\frac{1}{1-\kappa}}$ accordingly as $\kappa=1$ or $\kappa<1$, respectively. Define CDFs $G$ and $F$ by $G(x)\triangleq \indic{x\geq 0}$ and $F(x)\triangleq (1-p)\indic{x\geq 0}+p\indic{x\geq a}$ for all $x\in\R$. Note that $G$ is the CDF of a r.v. that equals $0$ a.s., while $F$ is the CDF of a r.v. that takes the value $0$ with probability $(1-p)$ and the value $a$ with probability $p$. It is a simple matter to evaluate $\rho(G)=0$ and $\rho(F)=a\sqrt{p} $.  Likewise, Lemma \ref{lemma:lipschitz-wasserstein} can be used to compute $W_{1}(F,G)=ap.$ One can now easily check that, in the case where $\kappa=1$, $\frac{|\rho(F)-\rho(G)|}{[W_{1}(F,G)]^{\kappa}}=p^{-\frac{1}{2}}>L$, while in the case where $\kappa<1$, $\frac{|\rho(F)-\rho(G)|}{[W_{1}(F,G)]^\kappa}=\frac{a^{1-\kappa}}{p^{\kappa-\frac{1}{2}}}>L$. We have thus shown that, for every $L>0$ and $\kappa\in(0,1]$, there exist CDFs $F$ and $G$ such that \eqref{tonedef} fails to hold. It follows that the risk measure $\rho$ is not \ref{type:one}. 
	\end{example}

	\subsection{Type-2 \ttwo\ risk measures }
\label{sec:t2}

In this subsection, we consider risk measures that satisfy a weaker version of the \holdernosp-continuity requirement specified in \ref{type:one}\ risk measures, and the definition below makes this continuity requirement precise.
\begin{risktype}
	\label{type:two}
Let $(\L,W_1)$ denote the metric space of distributions, with Wasserstein distance as the metric. 
The risk measure $\rho(\cdot)$ satisfies {\em a   truncated \holdernosp-continuity condition with a tail bound} on $(\L,W_{1})$ if there exist positive constants $\alpha_{1}$ ,$\alpha_{2}$, $\alpha_{3}$ $L_{1},L_{2},L_{3},K_{1},K_{2}$, and $\gamma$ such that $\alpha_{1},K_{1},K_{2}\leq 1$ and,  for every choice of two distributions $F, G \in \L$ and every $\tau >0$, the following holds:
\begin{align}
	\left| \rho(F) - \rho(G|_{\tau})\right| &\leq L_{1} \left(W_1(F,G)\right)^{\alpha_{1}} \tau^{\gamma}+L_{2}\int_{K_{1}\tau}^{\infty}[1-F(z)]^{\alpha_{2}}\mathrm{d}z \nonumber \\
	& {}+ L_{3}\int^{-K_{2}\tau}_{-\infty}[F(z)]^{\alpha_{3}}\mathrm{d}z.
	\label{ttwodef}
\end{align} 
\end{risktype}

To see the significance of the inequality \eqref{ttwodef}, let $\rho$ be a type \ref{type:two}\ risk measure satisfying \eqref{ttwodef}, and consider two distributions $F$ and $G$ supported on the interval $[-K_{2}\tau,K_{1}\tau]$ for some $\tau>0$. Then $G|_{\tau}=G$, and $1-F(z)=0=F(y)$  for all $z>K_{1}\tau$ and $y<-K_{2}\tau$. The inequality \eqref{ttwodef} then yields
\begin{equation}
    \left| \rho(F) - \rho(G)\right| \leq L_{1} \left(W_1(F,G)\right)^{\alpha_{1}} \tau^{\gamma}.
\end{equation}
The equation above indicates that, when restricted to distributions having bounded support, a type  \ref{type:two}\ risk measure is H\"{o}lder continuous with respect to  the Wasserstein distance with a H\"{o}lder constant that grows in an unbounded manner as the  length of the support increases.  

The reader will notice that the condition \eqref{ttwodef} is not symmetric in the distributions $F$ and $G$, and it may seem that a more natural-looking condition might be to require that 
\begin{align}
	\left| \rho(F) - \rho(G)\right| &\leq L_{1} \left(W_1(F,G)\right)^{\alpha_{1}} \tau^{\gamma}+L_{2}\int_{\tau}^{\infty}|F(z)-G(z)|^{\alpha_{2}}\mathrm{d}z \nonumber \\
	& {}+ L_{3}\int^{-\tau}_{-\infty}|F(z)-G(z)|^{\alpha_{3}}\mathrm{d}z
	\label{ttwoalt}
\end{align} 
hold for all $\tau\geq 1$. 
Note that the condition \eqref{ttwoalt} implies \eqref{ttwodef}, and hence \eqref{ttwodef} is the weaker of the two. Also, as we show below, \eqref{ttwodef} suffices for the bounds that we wish to derive and apply to known risk measures. Hence we choose to go with the weaker, albeit odd-looking, condition \eqref{ttwodef}. 

In the following subsection, we introduce cumulative prospect theory (CPT), which is a prominent risk measure in human-centered decision making systems, and an example of a \ref{type:two}\ risk measure. Subsequently, we describe rank-dependent expected utility (RDEU), which we show to be a special instance of the CPT-value that we define below. 
The risk measure that was seen to be not  \ref{type:one} in example  \ref{ctrex2} is a special case of the CPT-value that we introduce next.

\subsubsection{Cumulative prospect theory (CPT)}
For any r.v. $X$, the CPT-value is defined as 
\begin{align}
	C(X) &= \int_{0}^{\infty} w^+\left(\Prob{u^+(X)>z}\right) \mathrm{d}z - \int_{0}^{\infty} w^-\left(\Prob{u^-(X)>z}\right) \mathrm{d}z. \label{eq:cpt-val}
\end{align}
Let us deconstruct the above definition. First, the functions
$u^+,u^-:\R\rightarrow \R_+$ are utility functions which are assumed to be continuous,
with $u^+(x)=0$ when $x\le 0$ and increasing otherwise, and with $u^-(x)=0$ when $x\ge 0$ and decreasing otherwise.
The utility functions capture the human inclination to play safe with gains and take risks with losses  -- see Fig \ref{fig:u}.
Second, 
$w^+,w^-:[0,1] \rightarrow [0,1]$ are weight functions, which are assumed to be continuous, non-decreasing and satisfy $w^+(0)=w^-(0)=0$ and $w^+(1)=w^-(1)=1$.
The weight functions $w^+, w^-$ capture the human inclination to view probabilities in a non-linear fashion. \cite{tversky1992advances,Barberis:2012vs}  (see Fig \ref{fig:w} from \cite{tversky1992advances}) recommend the following choices for $w^+$ and $w^-$, based on inference from experiments involving human subjects:
\begin{align}
	w^+(p) = \frac{p^{0.61}}{\left(p^{0.61}+(1-p)^{0.61}\right)^{\frac1{0.61}}}, \textrm{ and } w^-(p) = \frac{p^{0.69}}{\left(p^{0.69}+(1-p)^{0.69}\right)^{\frac1{0.69}}}.\label{eq:cpt-weights}
\end{align}

\begin{minipage}{.5\textwidth}
	\centering
	\tabl{c}{
		\scalebox{0.65}{\begin{tikzpicture}
				\begin{axis}[width=10cm,height=6.5cm,legend pos=south east,         
					axis lines=middle,
					xmin=-5,     
					xmax=5,    
					ymin=-4,    
					ymax=4,  
					ylabel={\large\bf Utility},
					xlabel={\large\bf Gains},
					x label style={at={(axis cs:4.7,-0.8)}},
					y label style={at={(axis cs:-0.8,4.8)}},
					xticklabels=\empty,
					yticklabels=\empty,
					legend style={at={(0.5,-0.1)},anchor=north, legend columns=-1,column sep=1ex}
					]
					\addplot[name path=cptplus,domain=0:5, green!35!black, very thick,smooth] 
					{pow(abs(x),0.8)}; 
					\addlegendentry{\large $\bm{u^+}$};
					\addplot[name path=cptminus,domain=-5:0, red!35!black,very thick,smooth] 
					{-2*pow(abs(x),0.7)}; 
					\addlegendentry{\large $\bm{-u^-}$}
					\addplot[domain=-5:5, blue, thick]           {x}; 
					
					\path[name path=diagplus] (axis cs:0,0) -- (axis cs:5,5);
					\path[name path=diagminus] (axis cs:-5,-5) -- (axis cs:0,0);
					\path[name path=xaxisplus] (axis cs:0,0) -- (axis cs:5,0);
					\path[name path=xaxisminus] (axis cs:-5,0) -- (axis cs:0,0);
					\path[name path=yaxisplus] (axis cs:0,0) -- (axis cs:0,5);
					\path[name path=yaxisminus] (axis cs:0,-5) -- (axis cs:0,0);
					
					\node at (axis cs:  -3.7,-0.45) {\large\bf Losses};
				\end{axis}
		\end{tikzpicture}}\\[1ex]
	}
	
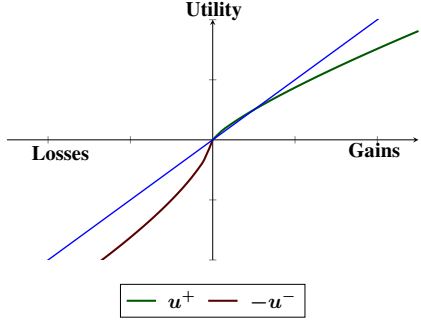
\captionof{figure}{Utility function}
	\label{fig:u}
\end{minipage}%
\begin{minipage}{.5\textwidth}
	\centering
	\tabl{c}{
		\scalebox{0.65}{\begin{tikzpicture}
				\begin{axis}[width=10cm,height=6.5cm,legend pos=south east,
					grid = major,
					grid style={dashed, gray!30},
					xmin=0,     
					xmax=1,    
					ymin=0,    
					ymax=1,  
					axis background/.style={fill=white},
					ylabel={\large Weight $\bm{w(p)}$},
					xlabel={\large Probability $\bm{p}$},
					legend style={at={(0.5,-0.25)},anchor=north, legend columns=-1,column sep=1ex}
					]
					\addplot[domain=0:1, magenta, thick,smooth,samples=1500] 
					{pow(x,0.6)/pow((pow(x,0.61) + pow(1-x,0.61)),1.64)}; 
					\addlegendentry{\large $\bm{w(p)=\frac{p^{0.61}}{(p^{0.61}+ (1-p)^{0.61})^{1/0.61}}}$};           
					\addplot[domain=0:1, upmaroon, thick]           {x};                      
					\addlegendentry{\large $\bm{w(p)=p}$};           
				\end{axis}
		\end{tikzpicture}}\\[1ex]
	}
	
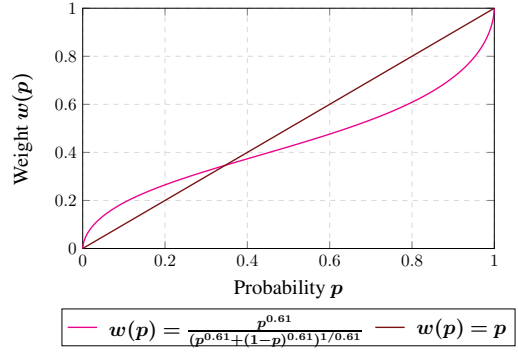
\captionof{figure}{Weight function}
	\label{fig:w}
\end{minipage}  

The lemma below shows that CPT-value is a risk measure of type \ref{type:two} under certain assumptions. 
\begin{lemma}
	\label{lemma:cpt-t2}
	Suppose that the utility functions $u^+,u^-:\R\rightarrow \R_+$ are differentiable, and their derivatives are bounded above and below by $K^{+}>0$ and $k^{+}>0$, and $K^{-}$ and $k^{-}>0$, respectively,  in absolute value.
	Further assume that the weight functions $w^{+}$ and $w^{-}$ are H\"{o}lder continuous with exponent $\alpha\in(0,1]$ and H\"{o}lder constant $L>0$,  and let $F,G \in\L$ be CDFs. Then, for every $\tau>0$, we have
	\begin{align}
		|C(F)-C(G|_{\tau})|&\leq L(K^{+}+K^{-})[W_{1}(F,G)]^{\alpha}\tau^{(1-\alpha)}+LK^{+}\int_{\frac{k^{+}}{K^{+}}\tau}^{\infty}[1-F(z)]^{\alpha}\mathrm{d}z \nonumber\\
		& {}+LK^{-}\int^{-\frac{k^{-}}{K^{-}}\tau}_{-\infty}[F(z)]^{\alpha}\mathrm{d}z.
		\label{cptwkhld}
	\end{align}
	\label{cpttt2lem} 
\end{lemma}
\begin{proof}
	See Section \ref{sec:cpt-t2-proof}.
\end{proof}

\begin{remark}
\label{rem:ctrex}
It is not difficult to see that the risk measure considered in Example \ref{ctrex2} can be written as a right-tailed version of CPT value by letting $u^{-}$ be identically zero, $u^{+}$ to be the identity map, and $w^{+}$ to be the square root function. Although these choices do not satisfy the assumptions of Lemma  \ref{cpttt2lem}, one can follow the steps in the proof of Lemma \ref{cpttt2lem} to show that the risk measure in the example is \ref{type:two}. Example \ref{ctrex2} thus provides a risk measure that is \ref{type:two} but not \ref{type:one}. 
\end{remark}

%%%%%%%%%%%%%%%%%%%%%%%%%%%%%%%%%%%%%%%%%%%%%%%%%%%%%%%%%%%%%%%%%%%%% 
We now describe the CPT-value estimation scheme, which is a variant of the one proposed in \citet{prashanth2016cumulative}. In particular, unlike the aforementioned reference, we employ a truncated EDF to obtain the CPT-value estimate.
Let $X_i$, $i=1,\ldots,n,$ denote $n$ independent samples from the distribution of $X$. 
For any given $\tau_n$, and  real-valued functions $u^+$ and $u^-$, let $\left(F_n|_{\tau_n}\right)^+$ and $\left( F_n|_{\tau_n}\right)^-$ denote the truncated EDFs formed from the samples $\{u^+(X_i), i=1,\ldots,n\}$ and $\{u^-(X_i), i=1,\ldots,n\}$, respectively.
Using the truncated EDFs, the CPT-value is estimated as follows:
\begin{align}
	C_n =& \int_{0}^{\tau_n} w^+(1-\left(F_n|_{\tau_n}\right)^+(x))  \mathrm{d}x - \int_{0}^{\tau_n} w^-(1-\left(F_n|_{\tau_n}\right)^-(x))  \mathrm{d}x.\label{eq:cpt-est}
\end{align}
Notice that we have substituted the complementary (truncated) EDFs $\left(1-\left(F_n|_{\tau_n}\right)^+(x)\right)$ and \\$\left(1-\left(F_n|_{\tau_n}\right)^-(x)\right)$  for $\Prob{u^+(X)>x}$ and $\Prob{u^-(X)>x}$, respectively,  in \eqref{eq:cpt-val}, and then performed an integration of the weight function composed with the complementary EDF. It is apparent that the CPT-value estimator in \eqref{eq:cpt-est} equals the CPT value of the truncated EDF $F_{n}|_{\tau_{n}}$, and thus is a special case of the estimator \eqref{eq:rho-est-trunc} for a general risk measure.

Let $\tilde X_i = X_i \indic{X_i \le \tau_n}$, for $i=1,\ldots,n$. 
Using arguments similar to that in Section III of \citet{prashanth2016cumulative}, the first and second integral, say $C_n^+$ and $C_n^-$, in \eqref{eq:cpt-est} can be easily computed using the order statistics $\{\tilde X_{(1)}, \ldots, \tilde X_{(n)}\}$ of the truncated samples $\{\tilde X_i: i=1,\ldots,n\}$ as follows:
\begin{align*}
	C_n^+&=\sum_{i=1}^{n} u^+(\tilde X_{[i]}) \left[w^+\left(\frac{n+1-i}{n}\right)- w^+\!\left(\frac{n-i}{n}\right) \right],\\
	C_n^-&=\sum_{i=1}^{n} u^-(\tilde X_{[i]}) \left[w^-\left(\frac{i}{n}\right)- w^-\left(\frac{i-1}{n}\right) \right].  
\end{align*}

%%%%%%%%%%%%%%%%%%%%%%%%%%%%%%%%%%%%%%%%%%%%%%%%%%%%%%%%%%%%%%%%%%%%% 
%%%%%%%%%%%%%%%%%%%%%%%%%%%%%%%%%%%%%%%%%%%%%%%%%%%%%%%%%%%%%%%%%%%%% 
\subsubsection{Rank-dependent expected utility (RDEU)}
Let $w:[0,1]\rightarrow [0,1]$ be an increasing %, H\"{o}lder continuous 
weight function 
such that $w(0)=0$ and $w(1)=1$. Let $u:\R\rightarrow\R$ be a continuous, increasing  function satisfying  $u(0)=0$. Then, following \cite{quiggin2012generalized}, the RDEU-value $V(F)$ is defined by
\begin{align}
	V(F) = \int_{-\infty}^{\infty} u(x) \mathrm{d}(w \circ F)(x). \label{eq:rdeu-def}
\end{align}

The result below shows that RDEU-value is a special case of CPT-value as defined in \eqref{eq:cpt-val}
To elaborate, CPT allows one the freedom of choosing two different weight functions $w^+$ and $w^-$. By suitably defining these, we show that RDEU is a special case of CPT. 
\begin{lemma}
	\label{lemma:rdeu-t2} Let $w:[0,1]\rightarrow [0,1]$ and $u:\R\rightarrow \R$ be a weight function and a utility function as above. Assume that $u$ is unbounded above and below. 
	Define $u^+(x) =u(x\indic{x\ge 0})$, $u^-(x) =-u(x\indic{x< 0})$, for $x\in \R$. Let $w^-(p) =w(p)$, and $w^+(p)=1-w(1-p)$, for $p\in [0,1]$. Then, for any r.v. $X$ with CDF $F$, we have
	\[V(F) = \int_{0}^{\infty} w^+\left(\Prob{u^+(X)>z}\right) \mathrm{d}z - \int_{0}^{\infty} w^-\left(\Prob{u^-(X)>z}\right) \mathrm{d}z. \]
\end{lemma}
\begin{proof}
	See Section \ref{sec:rdeu-t2-proof}.
\end{proof}

Thus, the RDEU-value $V$ is a special case of CPT-value, with $u^{\pm}$ and $w^{\pm}$ chosen as in the statement of the lemma above. Note that for the choice of $w^+$ and $w^-$ given in Lemma \ref{lemma:cpt-t2}, we have 
\[w^+(p) + w^-(1-p)=1.\] CPT-value, as defined in \eqref{eq:cpt-val}, is more general, as the weight functions $w^+$ and $w^-$ are not required to satisfy the aforementioned equality. Owing to the fact that RDEU is a special case of CPT, the estimation scheme presented in \eqref{eq:cpt-est} applies to RDEU as well.

%%%%%%%%%%%%%%%%%%%%%%%%%%%%%%%%%%%%%%%%%%%%%%%%%%%%%%%%%%%%%%%%%%%%% 
%%%%%%%%%%%%%%%%%%%%%%%%%%%%%%%%%%%%%%%%%%%%%%%%%%%%%%%%%%%%%%%%%%%%% 
%%%%%%%%%%%%%%%%%%%%%%%%%%%%%%%%%%%%%%%%%%%%%%%%%%%%%%%%%%%%%%%%%%%%% 
%%%%%%%%%%%%%%%%%%%%%%%%%%%%%%%%%%%%%%%%%%%%%%%%%%%%%%%%%%%%%%%%%%%%% 
\subsubsection{Distorted risk measure (DRM)}
For a r.v. $X$ with CDF $F$, the DRM $D(X)$ is defined as
\begin{align}
	D(X) = \int_{-\infty}^{0} \left[w\left(1-F(z)\right) -1\right] \mathrm{d}z +\int_{0}^{\infty} w\left(1-F(z)\right) dz,
\end{align}
where $w:[0,1]\rightarrow [0,1]$ is a weight function that satisfies $w(0)=0$ and $w(1)=1$. 
If $w$ is concave, then DRM is a coherent risk measure. Further, DRMs are equivalent to spectral risk measures if the weight function is increasing and differentiable. In this case, the risk spectrum for the equivalent spectral risk measure is given by $\phi(\beta)=w'(1-\beta),$ for $\beta \in [0,1]$.

It is easy to see that DRMs are of type \ref{type:two}, as they can be treated as a special case of CPT-value. Thus, the bounds we derive for CPT-value estimation can be easily specialized to handle the case of DRMs. While DRMs are equivalent to spectral risk measures, the bounds for spectral risk measures derived earlier require that the risk spectrum be bounded. As noted in Remark \ref{remark:srm-ph}, there are spectral risk measures (or DRMs) that do not satisfy this boundedness requirement, and for such DRMs, one could take the \ref{type:two} route to arrive at estimation bounds. An example of such a risk measure is the proportional hazard transform, which has the following risk spectrum: $\phi(\beta) = \frac{1}{\varkappa} \beta^{\frac{1}{\varkappa}-1}$, with $\varkappa\ge 1$.

	%%%%%%%%%%%%%%%%%%%%%%%%%%%%%%%%%%%%%%%%%%%%%%%%%%%%%%%%%%%%%%%%%%%%% 
	\section{Map of the results}
	\label{sec:results}
	%%%%%%%%%%%%%% EXPECTATION BOUNDS %%%%%%%%%%%%%%%%%%%%%%%%%%%%%%%%%%%%%%%%%%%%%%%%%%%%%%%%%%%%%%%
	In the next three sections, we provide bounds in expectation  as well as concentration bounds for estimates of the risk measures described in the previous section under the assumptions (\ref{ass:c1})-(\ref{ass:heavy}) (as well as \ref{ass:heavyvariant} appearing in the next section) on the underlying  distribution. Since the number of combinations of risk measures and distributional assumptions is rather large, we provide here a tabular summary of all the bounds that will be presented in the succeeding sections. 
	
	\begin{table}[h]
		\caption{Summary of the bounds in expectation of the form $\EE{|\rho_{n}-\rho(X)|}$ for various choices of risk measure $\rho(X)$. Here $X$ is a r.v. satisfying $\E\left(|X|^\beta\right) < \top < \infty$ for some $\beta>1$, and $\rho_n$ is an estimate of $\rho(X)$ using $n$ i.i.d. samples. }
		\label{tab:summary-expec}
		\centering
		\begin{tabular}{c|c|c}
			\toprule
			\textbf{Risk measure} & \textbf{Bound in expectation} & \textbf{Reference}\\\midrule
			OCE  (includes CVaR) & $O\left(\frac{1}{n^{\min\{\frac{1}{2},1-\frac{1}{\beta}\}}}\right)$ & Corollary \ref{cor:oce-expec}\\\midrule
			Spectral risk measure  & $O\left(\frac{1}{n^{\min\{\frac{1}{2},1-\frac{1}{\beta}\}}}\right)$& Corollary \ref{cor:srm-expec} \\\midrule
			Utility-based shortfall risk &    $O\left(\frac{1}{n^{\min\{\frac{1}{2},1-\frac{1}{\beta}\}}}\right)$& Corollary \ref{cor:ubsr-expec} \\\midrule
			Cumulative prospect theory    & $O\left(\frac{1}{\left(n^{\min\{\frac{1}{2},1-\frac{1}{\beta}\}}\right)^{\frac{\beta\alpha-1}{\beta-1}}}\right)$& Corollary \ref{cor:cpt-expec-bd} \\\bottomrule
		\end{tabular}
	\end{table}
	%%%%%%%%%%%%%% SUBGAUSS %%%%%%%%%%%%%%%%%%%%%%%%%%%%%%%%%%%%%%%%%%%%%%%%%%%%%%%%%%%%%%%
	\begin{table}[h]
		\caption{Summary of the concentration bounds of the form $\Prob{|\rho_{n}-\rho(X)|>\epsilon}$, for various choices of risk measure $\rho(X)$. Here $X$ is a sub-Gaussian r.v. with parameter $\sigma$. Here the constants $c_1,c_2$ are functions of the sub-Gaussianity parameter $\sigma$, but an explicit expression is not available. For bounds with explicit constants, see Table \ref{tab:summary-subgauss-const}}
		\label{tab:summary-subgauss}
		\centering
		\begin{tabular}{c|c|c}
			\toprule
			\textbf{Risk measure} & \textbf{Concentration bound} & \textbf{Reference}\\\midrule
			OCE  & $2c_1\exp\left(- \frac{c_2 n \epsilon^2}{L^2}\right)$ & Corollary \ref{prop:oce-t1-conc} \\
			&&with $\beta=2$\\\midrule
			Spectral risk measure  & $2c_1\exp\left(- \frac{c_2 n \epsilon^2}{K^2}\right)$& Corollary \ref{specconcprop} \\
			&&with $\beta=2$\\\midrule
			Utility-based shortfall risk &    $ 2c_1\exp\left(- \frac{c_2k^{2} n \epsilon^2}{K^2}\right)$& Corollary \ref{ubsrconcprop} \\
			&&with $\beta=2$\\\midrule
			Cumulative prospect theory    & $ c_1 \exp\left( - c_2 n \left(\frac{\epsilon- c_3(n)}{L(K^++K^-)\tau_{n}^{1-\alpha}}\right)^{\frac{2}{\alpha}}  \right),$& Corollary \ref{cor:cpt-subgauss} \\
			& $c_3(n) = \frac{\mbox{const} }{\sqrt{\log n} n^{\alpha(1-\alpha)}}$ $\tau_n= \mbox{const} \sqrt{\log n}$.&with $\beta=2$\\\bottomrule
		\end{tabular}
	\end{table}
	%\clearpage \newpage
	
	%%%%%%%%%%%%%% SUBGAUSS - with constants %%%%%%%%%%%%%%%%%%%%%%%%%%%%%%%%%%%%%%%%%%%%%%%%%%%%%%%%%%%%%%%
	\begin{table}[h]
		\caption{Summary of the concentration bounds of the form $\Prob{|\rho_{n}-\rho(X)|>\epsilon}$, for various choices of risk measure $\rho(X)$. Here $X$ is a sub-Gaussian r.v. with parameter $\sigma$. Unlike Table \ref{tab:summary-subgauss}, the bounds here feature explicit constants. }
		\label{tab:summary-subgauss-const}
		\centering
		\begin{tabular}{c|c|c}
			\toprule
			\textbf{Risk measure} & \textbf{Concentration bound} & \textbf{Reference}\\\midrule
			OCE  & $\exp\left(- \frac{n}{256\sigma^2\myexp} \left(\frac{\epsilon}{L}-\frac{512\sigma}{\sqrt{n}}\right)^2\right)$ & Corollary \ref{cor:oce-subgauss}\\\midrule
			Spectral risk measure  & $\exp\left(- \frac{n}{256\sigma^2\myexp} \left(\frac{\epsilon}{K}-\frac{512\sigma}{\sqrt{n}}\right)^2\right)$& Corollary \ref{cor:spec-subgauss} \\\midrule
			Utility-based shortfall risk &    $ \exp\left(- \frac{n}{256\sigma^2\myexp} \left(\frac{k\epsilon}{K}-\frac{512\sigma}{\sqrt{n}}\right)^2\right)$& Corollary \ref{cor:ubsr-subgauss} \\\midrule
			Cumulative prospect theory    & $ \exp\left( - \frac{n}{256\sigma^2\myexp} \right. \times$& Corollary \ref{cor:cpt-subgauss-const} \\
			&$\left.\left(\left(\frac{\epsilon- c_3(n)}{L(K^++K^-)\left[\max\left\{\frac{K^{+}}{k^{+}},\frac{K^{-}}{k^{-}}\right\}\left(\log n\right)^{\frac{1}{2}}\right]^{1-\alpha}}\right)^{\frac{1}{\alpha}} -\frac{512 \sigma}{\sqrt{n}}\right)^2  \right)$ &\\\bottomrule
		\end{tabular}
	\end{table}
	\clearpage\newpage
	%%%%%%%%%%%%%% SUBEXP %%%%%%%%%%%%%%%%%%%%%%%%%%%%%%%%%%%%%%%%%%%%%%%%%%%%%%%%%%%%%%%
	\begin{table}[h]
		\caption{Summary of the concentration bounds of the form $\Prob{|\rho_{n}-\rho(X)|>\epsilon}$, for various choices of risk measure $\rho(X)$. Here $X$ is a sub-exponential r.v. with parameter $c$. }
		\label{tab:summary-subexp}
		\centering
		\begin{tabular}{c|c|c}
			\toprule
			\textbf{Risk measure} & \textbf{Concentration bound} & \textbf{Reference}\\\midrule
			OCE  & $\exp\left(- \frac{n}{\frac{32}{c^2} + \frac{4}{c}\left(\frac{\epsilon}{L}-\frac{384}{c\sqrt{n}}\right)} \left(\frac{\epsilon}{L}-\frac{384}{c\sqrt{n}}\right)^2\right)$ & Corollary \ref{cor:cvar-subexp}\\\midrule
			Spectral risk measure  & $\exp\left(- \frac{n}{\frac{32}{c^2} + \frac{4}{c}\left(\frac{\epsilon}{K}-\frac{384}{c\sqrt{n}}\right)} \left(\frac{\epsilon}{K}-\frac{384}{c\sqrt{n}}\right)^2\right)$& Corollary \ref{cor:spec-subexp} \\\midrule
			Utility-based shortfall risk &    $ \exp\left(- \frac{n}{\frac{32}{c^2} + \frac{4}{c}\left(\frac{k\epsilon}{K}-\frac{384}{c\sqrt{n}}\right)} \left(\frac{k\epsilon}{K}-\frac{384}{c\sqrt{n}}\right)^2\right)$& Corollary \ref{cor:ubsr-subexp} \\\midrule
			Cumulative prospect theory    & $ \exp\Bigg( -  \frac{n}{\frac{32}{c^2} + \frac{4}{c}\left(\left(\frac{\epsilon- \frac{(K^{+}+K^{-})L}{c\alpha n^{\alpha}}}{(K^{+}+K^{-})L \tau_n^{1-\alpha}}\right)^{\frac{1}{\alpha}} - \frac{384}{c\sqrt{n}}\right)}$& Corollary \ref{cor:cpt-subexp}\\
			& $\qquad\times\left(\left(\frac{\epsilon-\frac{(K^{+}+K^{-})L}{c\alpha n^{\alpha}}}{(K^{+}+K^{-})L \tau_n^{1-\alpha}}\right)^{\frac{1}{\alpha}} -\frac{384}{c\sqrt{n}}\right)^2  \Bigg),$&\\
			& $\tau_n= \mbox{const} \sqrt{\log n}$ &\\\bottomrule
		\end{tabular}
	\end{table}
	%%%%%%%%%%%%%%%%%%%%%%%%%%%%%%%%%%%%%%%%%%%%%%%%%%%%%%%%%%%%%%%%%%%%% 
	%%%%%%%%%%%%%%%%%%%%%%%%%%%%%%%%%%%%%%%%%%%%%%%%%%%%%%%%%%%%%%%%%%%%% 
	\section{Bounds in expectation for risk estimation}
	\label{sec:expec-bounds}
	In this section, we provide non-asymptotic bounds, which hold in expectation, for estimation of risk measures. For these bounds, we assume that the underlying distribution has a finite higher-moment bound. More precisely, we assume that the following condition holds.  
\begin{assumption}
	\label{ass:heavyvariant}
	There exists $\beta >1$ such that $\E\left(|X|^\beta\right) < \top < \infty$.
\end{assumption}

\subsection{Bounds for risk measures of type \ref{type:one}}
\begin{theorem}
	\label{thm:t1-expec-bd}
	Suppose $X$ is a r.v. satisfying \ref{ass:heavyvariant}, $\rho:\L\rightarrow \R$ is a risk measure of type \ref{type:one} with parameters $L>0$ and $\kappa>0$ as in \eqref{tonedef}, and $\rho_{n}$ is given by \eqref{eq:rho-est}. Then, for every $n\geq 1$, we have 
	\begin{equation}
		\EE{|\rho_{n}-\rho(X)|}\leq 
		L \left(\frac{2^{\beta+3} \top}{n^{\min\{\frac{1}{2},1-\frac{1}{\beta}\}}}\right)^\kappa.\label{eq:t1-expec-bd}
	\end{equation}
\end{theorem}
\begin{proof}
	See Section \ref{sec:proof-expec-t1}.
\end{proof}

The following corollaries provide bounds in expectation for empirical OCE, SRM and UBSR. The proofs of these corollaries follow in a straightforward fashion using the result in Theorem \ref{thm:t1-expec-bd}, and we omit the proof details.
\begin{corollary}[Bound in expectation: OCE]
	\label{cor:oce-expec}
	Suppose $X$ is a r.v. satisfying \ref{ass:heavyvariant}, $\phi:\R\rightarrow \R $ is a $L$-Lipschitz disutility function, $\oce(X)$ is the OCE risk of $X$ as defined in \eqref{ocedef},  and $\oce_{n}$ is its empirical estimate as in \eqref{ocest}. Then, for every $n\geq 1$, we have 
	\begin{align*}
		&\EE{ \left| \oce_{n} - \oce(X) \right|} 
		\le L \left(\frac{2^{\beta+3} \top}{n^{\min\{\frac{1}{2},1-\frac{1}{\beta}\}}}\right).
	\end{align*}
\end{corollary}

In light of the fact that CVaR is an OCE risk satisfying \eqref{cvrlip}, it is clear that the bound above holds for empirical CVaR with $L=(1-\alpha)^{-1}$.

\begin{corollary}[Bound in expectation: SRM]
	\label{cor:srm-expec}
Suppose $X$ is a r.v. satisfying \ref{ass:heavyvariant}, $\phi:[0,1]\rightarrow [0,\infty) $ is a weighting function uniformly bounded above by $K>0$, $M_{\phi}(X)$ is the OCE risk of $X$ as defined in \eqref{specdef},  and $m_{n,\phi}$ is its empirical estimate as in \eqref{empspecdef}. Then, for every $n\geq 1$, we have 
		\begin{align*}
		&\EE{ \left| m_{n,\phi} - M_{\phi}(X) \right| } 
		\le K \left(\frac{2^{\beta+3} \top}{n^{\min\{\frac{1}{2},1-\frac{1}{\beta}\}}}\right).
	\end{align*}
\end{corollary}
\begin{corollary}[Bound in expectation: UBSR]
	\label{cor:ubsr-expec}
	Suppose $X$ is a r.v. satisfying \ref{ass:heavyvariant}, $l:\R\rightarrow \R $ is a utility function satisfying the assumptions in Lemma \ref{lemma:ubsr-t1}, and $S_{\alpha}(X)$ is the UBSR of $X$ as defined in \eqref{eq:ubsr-def}.  
	Let $\xi_{n,\alpha}$ denote the solution to the constrained problem in \eqref{eq:sr-est}.
	 Then, for every $n\geq 1$, we have
	\begin{align*}
		&\EE{ \left|\xi_{n,\alpha} - S_{\alpha}(X)\right|} 
		\le \frac{K}{k} \left(\frac{2^{\beta+3} \top}{n^{\min\{\frac{1}{2},1-\frac{1}{\beta}\}}}\right),
	\end{align*}
	where $K,k>0$ are as in Lemma \ref{lemma:ubsr-t1}.
\end{corollary}

%%%%%%%%%%%%%%%%%%%%%%%%%%%%%%%%%%%%%%%%%%%%%%%%%%%%%%%%%%%%%%
%%%%%%%%%%%%%%%%%%%%%%%%%%%%%%%%%%%%%%%%%%%%%%%%%%%%%%%%%%%%%%
%%%%%%%%%%%%%%%%%%%%%%%%%%%%%%%%%%%%%%%%%%%%%%%%%%%%%%%%%%%%%%
\subsection{Bounds for risk measures of type \ref{type:two}}
In this subsection, we provide bounds on the expected estimation error for a risk measure of type \ref{type:two}. 

\begin{theorem}
	\label{thm:t2-expec-bd}
	Suppose $X$ is a r.v. satisfying \ref{ass:heavyvariant}, and $\rho:\L\rightarrow \R$ is a risk measure of type \ref{type:two}, with  parameters $\alpha_1, \alpha_2, \alpha_3, L_1, L_2, L_3, K_1, K_2$ as defined in \eqref{ttwodef}. Further, assume that  $\min\{\beta \alpha_2, \beta \alpha_3\} > 1$. Fix $\tau>0$.  Then, for every $n\geq 1$, we have 
	\begin{equation}
		\EE{| \rho_{n,\tau} - \rho(X)|}\leq 
		L _1 \tau^\gamma\left(\frac{2^{\beta+3} \top}{n^{\min\{\frac{1}{2},1-\frac{1}{\beta}\}}}\right)^{\alpha_1}
		+ \frac{L_2 \top^{\alpha_2}}{(\beta\alpha_2-1) (K_1\tau)^{\beta\alpha_2-1}}
		+ \frac{L_3 \top^{\alpha_3}}{(\beta\alpha_3-1) (K_2\tau)^{\beta\alpha_3-1}},\label{eq:t2-expec-bd}
	\end{equation}
	where $\rho_{n,\tau} = \rho(F_n|_{\tau})$.
\end{theorem}
\begin{proof}
	See Section \ref{sec:proof-expec-t2}.
\end{proof}

We now specialize the result in Theorem \ref{thm:t2-expec-bd} for the case of CPT-value. We consider the CPT value estimator based on truncation, defined in \eqref{eq:cpt-est}. 

\begin{corollary}[Bound in expectation: CPT value]
	\label{cor:cpt-expec-bd}
	Assume that the conditions in Lemma \ref{lemma:cpt-t2} hold. Suppose  that $X$ satisfies \ref{ass:heavyvariant} for some $\beta>1$ such that $\beta\alpha>1$, where $\alpha\in(0,1]$ is as in Lemma \ref{lemma:cpt-t2}.
	For each $n\geq 1$, set $\tau_n = \left(n^{\min\{\frac{1}{2},1-\frac{1}{\beta}\}}\right)^{\frac{1}{(\beta-1)}}$, and form the CPT-value estimate $C_n$ using \eqref{eq:cpt-est}. Then  we have
	\begin{align*}
&		\EE{ \left|{C}_n - C(X)\right|}\leq 
\frac{L \top^{\alpha}}{\left(n^{\min\{\frac{1}{2},1-\frac{1}{\beta}\}}\right)^{\frac{\beta\alpha-1}{\beta-1}}} \\
&\qquad\qquad\qquad \times \left(	2^{(\beta+3)\alpha }(K^+ + K^-)
	+ \frac{K^+ }{(\beta\alpha-1) \left(\frac{k^+}{K^+}\right)^{\beta\alpha-1}}
	+ \frac{K^- }{(\beta\alpha-1) \left(\frac{k^-}{K^-}\right)^{\beta\alpha-1}}\right).
	\end{align*}
\end{corollary}
\begin{proof}
	See Section \ref{sec:proof-expec-cpt}. .
\end{proof}

As discussed in subsection \ref{sec:t2}, RDEU and DRM can be viewed as special cases of CPT value, and are hence covered by Corollary \ref{cor:cpt-expec-bd}. 

	%%%%%%%%%%%%%%%%%%%%%%%%%%%%%%%%%%%%%%%%%%%%%%%%%%%%%%%%%%%%%%%%%%%%% 
	%%%%%%%%%%%%%%%%%%%%%%%%%%%%%%%%%%%%%%%%%%%%%%%%%%%%%%%%%%%%%%%%%%%%% 
	%%%%%%%%%%%%%%%%%%%%%%%%%%%%%%%%%%%%%%%%%%%%%%%%%%%%%%%%%%%%%%%%%%%%% 
	\section{Concentration bounds for distributions satisfying \ref{ass:c1}}
	\label{sec:subgauss}
	In this section, we provide concentration bounds for the risk estimates introduced in section \ref{sec:risk-abstract} using the Wasserstein distance bound in Lemma \ref{lemma:wasserstein-dist-bound}. 
The concentration bounds assume that the underlying distribution satisfies condition \ref{ass:c1}, that is, an exponential moment bound with an exponent greater than one. Sub-Gaussian distributions are a popular class of distributions that satisfy this assumption. 
As mentioned before, sub-exponential distributions are treated separately in the next section. 

In the following sections, we relate the estimation errors for both risk measure types to the Wasserstein distance between the EDF and the underlying CDF. The resulting concentration bounds apply to unbounded random variables as long as they belong to one of the three classes mentioned above. 
An alternative approach involves an application of the Dvoretzky-Kiefer-Wolfowitz (DKW) theorem \cite[Chapter 2]{wasserman2006}.
Such an approach has been used to obtain concentration bounds for CVaR/CPT (cf. \cite{thomas2019concentration,prashanth2018tac}, but the sup norm in the DKW inequality will make the resulting bounds applicable only to bounded r.v.s. In contrast, employing the Wasserstein metric allows us to derive concentration bounds for a broader class of  r.v.s. 

\subsection{Bounds for Risk Measures of Type \ref{type:one}}
Given a risk measure $\rho$, a r.v. $X$ with CDF $F$, and $n>0$, we form an empirical estimate $\rho_{n}$ of $\rho(X)$ using \eqref{eq:rho-est}. Recall that the latter estimate is obtained  by applying $ \rho$ to the EDF formed by drawing $n $ samples from $F$. 

\begin{theorem}
	\label{thm:t1-conc-bd}
	Suppose $X$ is a r.v. satisfying \ref{ass:c1}, and $\rho:\L\rightarrow \R$ is a risk measure of type \ref{type:one} with parameters $L>0$ and $\kappa>0$ as in \ref{tonedef}. Then, for every $\epsilon>0$ and $n\geq 1$, we have 
	\begin{equation}
		\Prob{|\rho_{n}-\rho(X)|>\epsilon}\leq c_1\left( \exp\left(-c_2 n \left(\frac{\epsilon}{L}\right)^{\frac{2}{\kappa}}\right)\indic{\epsilon\le L} + \exp\left(-c_3 n \left(\frac{\epsilon}{L}\right)^{\frac{\beta}{\kappa}}\right)\indic{\epsilon> L}\right),\label{c1t1bd}
	\end{equation}
	where the constants $c_1,c_2,$ and $c_3$ are as in Lemma \ref{lemma:wasserstein-dist-bound}.
	\label{c1t1thm}
\end{theorem}
\begin{proof}
	See Section \ref{sec:t1-conc-bd-proof}.
\end{proof}
The result above can be easily specialized for the case of sub-Gaussian distributions by using $\beta=2$ in the bound above.

\begin{remark}
It is easy to see from the first equality in  \eqref{lipwass2} that $X\mapsto \E(X)$ is a \ref{type:one}\  risk measure with parameters $L=\kappa=1$. Theorem \ref{thm:t1-conc-bd} applies with $\beta=2$, and yields the well-known Hoeffding's inequality for the concentration of sample mean in the case of sub-Gaussian r.v.s. 
\end{remark}

Next, we present a concentration bound for the sub-Gaussian case, but with explicit constants. 
\begin{theorem}
	\label{thm:t1-conc-bd-subgauss-const}
	Let $X$ be a sub-Gaussian r.v. with parameter $\sigma>0$. Suppose $\rho:\L\rightarrow \R$ is a risk measure of type \ref{type:one}, with parameters $L$ and $\kappa$. Then, for every $n\geq 1$ and $\epsilon$ such that $\epsilon$ such that $  \frac{512\sigma}{\sqrt{n}}<\left(\frac{\epsilon}{L}\right)^{\frac{1}{\kappa}} < \frac{512\sigma}{\sqrt{n}}+16\sigma\sqrt{\myexp}$, we have 
	\begin{equation}
		\Prob{|\rho_{n}-\rho(X)|>\epsilon}\leq \exp\left(- \frac{n}{256\sigma^2\myexp} \left(\left(\frac{\epsilon}{L}\right)^{\frac{1}{\kappa}}-\frac{512\sigma}{\sqrt{n}}\right)^2\right).
	\end{equation}
\end{theorem}
\begin{proof}
    See Section \ref{sec:t1-conc-bd-subgauss-const-proof}.
\end{proof}

\subsubsection{Bounds for empirical OCE}
We now provide a concentration bound for the empirical OCE estimate \eqref{ocest}, by relating the estimation error $\left| \oce_{n} - \oce(X) \right|$ to the Wasserstein distance between the true and empirical distribution functions, and subsequently invoking Lemma \ref{lemma:wasserstein-dist-bound} that bounds the Wasserstein distance between these two distributions. 
\begin{corollary}[OCE concentration]
	\label{prop:oce-t1-conc}
	Suppose $X$ satisfies \ref{ass:c1} for some $\beta>1$. 
	Let  $\phi:\R\rightarrow \R $ be a $L$-Lipschitz disutility function. Let  $\oce(X)$ be the OCE risk of $X$ as defined in \eqref{ocedef},  and let $\oce_{n}$ be its empirical estimate as in \eqref{ocest}. 	Then, for every $\epsilon >0$ and $n\geq 1$, we have
	\begin{align*}
		\Prob{ \left| \oce_{n}- \oce(X) \right| > \epsilon} &
		\le c_1\left[ \exp\left[-c_2 n (\epsilon/L)^2\right]\indic{\epsilon\le L} \right.\\
		&\quad\left.+ \exp\left[-c_3 n (\epsilon/L)^\beta\right]\indic{\epsilon> L}\right]\!, 
	\end{align*}
	where  the constants $c_1, c_2$ and $c_3$ are as in Lemma \ref{lemma:wasserstein-dist-bound}.
	\label{cvarconcprop}
\end{corollary}
\begin{proof}
	See Section \ref{sec:cvarconcprop-proof}.
\end{proof}	
One can specialize the result above to handle the case of sub-Gaussian random r.v.s. by using $\beta=2$ in the bound above.  However, the constants $c_1, c_2$ are unknown functions of the parameter $\sigma$. 
The following result provides an alternative OCE concentration bound with explicit constants.

\begin{corollary}[OCE concentration with explicit constants]
	\label{cor:oce-subgauss}
	Suppose $X$ is a sub-Gaussian r.v.  with parameter $\sigma$, and let $\phi$ be as in Corollary \ref{cvarconcprop}. Then, for every $n\geq 1$ and  $\epsilon$ such that\\ $  \frac{512\sigma}{\sqrt{n}}<\frac{\epsilon}{L} < \frac{512\sigma}{\sqrt{n}}+16\sigma\sqrt{\myexp}$, we have
	\begin{align*}
		&\Prob{ \left| \oce_{n} - \oce(X) \right| > \epsilon} \le \exp\left(- \frac{n}{256\sigma^2\myexp} \left(\frac{\epsilon}{L}-\frac{512\sigma}{\sqrt{n}}\right)^2\right).
	\end{align*}
\end{corollary}
\begin{proof}
	See Section \ref{sec:oce-subgauss-proof}.
\end{proof}	

As discussed in subsection \ref{ocesec}, CVaR is a special case of an OCE risk, and satisfies the assumptions of corollaries \ref{cvarconcprop} and  \ref{cor:oce-subgauss} with $L=(1-\alpha)^{-1}$. Hence concentration bounds for empirical CVaR follow from the preceding two results. 

The case considered in Corollary \ref{cor:oce-subgauss}, as applied to empirical CVaR, was also treated by \citet{Kolla,prashanth2019concentration}, while the special case of bounded r.v.s has been treated in \citet{brown2007large,wang2010deviation}. 
In terms of dependence on $n$ and $\epsilon$, the tail bound in Corollary \ref{cor:oce-subgauss} is better than the one-sided concentration bound in \citet{Kolla}. In fact, the dependence on $n$ and $\epsilon$ given in Corollary \ref{cor:oce-subgauss} matches that in the case of bounded distributions  \citep{brown2007large,wang2010deviation}. 
More recently,  \cite{prashanth2019concentration} have derived a two-sided concentration result for CVaR estimation. Their bound requires the knowledge of the density in a neighborhood of the true VaR,  while the constants in our bounds depend only on the parameter $\sigma$ of the underlying sub-Gaussian distribution. On the other hand, though our bounds do not depend on density-related information, they probably involve conservative constants. Finally, unlike \cite{prashanth2019concentration},  our bounds allow a multi-armed bandit application in the spirit of the classic UCB algorithm \citep{auer2002finite}, as we require only the knowledge of the sub-Gaussianity parameter in arriving at a confidence term. The reader is referred to Section \ref{sec:bandits} for more details.

\subsubsection{Bounds for empirical spectral risk measure}
In this section, we restrict ourselves to a spectral risk measure $M_{\phi}$ whose associated risk spectrum $\phi$ is bounded. Specifically, we assume that $|\phi(\beta)|\leq K$ for all $\beta\in[0,1]$ for some $K>0$. It immediately follows from Lemma \ref{lemma:srm-t1} that, if $X$ and $Y$ are r.v.s with CDFs $F_{1}$ and $F_{2}$, then 
\begin{equation}
	|M_{\phi}(X)-M_{\phi}(Y)|\leq KW_{1}(F_{1},F_{2}).\label{specriskwass}
\end{equation}
On noting from \eqref{empspecdef} that the empirical estimate $m_{n,\phi}$ of $M_{\phi}(X)$ is simply the spectral risk measure $M_{\phi}$ applied to a r.v. whose CDF is $F_{n}$, we conclude from \eqref{specriskwass} that
\begin{equation}
	|M_{\phi}(X)-m_{n,\phi}|\leq KW_{1}(F,F_{n}).\label{specriskwass2}
\end{equation}
Equation \eqref{specriskwass2} relates the estimation error $|M_{\phi}(X)-m_{n,\phi}|$ to the Wasserstein distance between the true and empirical CDFs of $X$. As in the case of CVaR, invoking Lemma \ref{lemma:wasserstein-dist-bound} provides concentration bounds for the empirical spectral risk measure estimate \eqref{empspecdef}. 
\begin{corollary}[SRM concentration]	\label{specconcprop} 
	Suppose $X$ satisfies \ref{ass:c1} for some $\beta>1$. Let $K>0$ and let $\phi:[0,1]\rightarrow [0,K]$ be a risk spectrum. 
	Then, for every $\epsilon >0$ and $n\geq 1$, we have
	\begin{align*}
		&\Prob{ \left| m_{n,\phi} - M_{\phi}(X) \right| > \epsilon} 
		\le c_1\left[ \exp\left[-c_2 n \left[\frac{\epsilon}{K}\right]^2\right]\indic{\epsilon\le K} + \exp\left[-c_3 n \left[\frac{\epsilon}{K}\right]^{\beta}\right]\indic{\epsilon> K}\right],
	\end{align*}
	where the constants $c_1, c_2$ and $c_3$ are as in Lemma \ref{lemma:wasserstein-dist-bound}.
\end{corollary}
\begin{proof}
	See Section \ref{sec:specconcprop-proof}.
\end{proof}	
As before, setting $\beta=2$ in the bound above handles the special case of sub-Gaussian random r.v.s.

Next, we present an SRM concentration bound with explicit constants.
\begin{corollary}[SRM concentration with explicit constants]
	\label{cor:spec-subgauss}
	Let $X$ be a sub-Gaussian r.v. with parameter $\sigma$, and let $\phi$ be a risk spectrum as in Corollary \ref{specconcprop}. Then, 
	for every $n\geq 1$ and  $\epsilon$ such that $  \frac{512\sigma}{\sqrt{n}}<\frac{\epsilon}{K} < \frac{512\sigma}{\sqrt{n}}+16\sigma\sqrt{\myexp}$, we have
\begin{align*}
	&\Prob{ \left| m_{n,\phi} - M_{\phi}(X) \right| > \epsilon} \le \exp\left(- \frac{n}{256\sigma^2\myexp} \left(\frac{\epsilon}{K}-\frac{512\sigma}{\sqrt{n}}\right)^2\right).
\end{align*}
\end{corollary}
\begin{proof}
	See Section \ref{sec:spec-subgauss-proof}.
\end{proof}	
The bound above is an improvement over the SRM concentration bound derived in \cite{pandey2019estimation} for two reasons. First, our bound is for the un-truncated estimator, while their bound involves a truncated estimate. Second, our bound applies for all $n\ge 1$, while their bound applies only for a sufficiently large number of samples. 

\subsubsection{Bounds for empirical utility-based shortfall risk}
As in the case of OCE and spectral risk measure, using the fact that UBSR is \ref{type:one}\ leads to the following concentration bound.
\begin{proposition}	\label{ubsrconcprop}
	Suppose $X$ satisfies \ref{ass:c1} for some $\beta>1$. Let the utility function $l$ in the definition \eqref{eq:ubsr-def} of $S_\alpha(X)$ satisfy the assumptions in Lemma \ref{lemma:ubsr-t1}, and let $\xi_{n,\alpha}$ denote the solution to the constrained problem in \eqref{eq:sr-est}.
	Then, for every $\epsilon >0$ and $n\geq 1$, we have
	\begin{align*}
		&\Prob{ \left| \xi_{n,\alpha} - S_{\alpha}(X) \right| > \epsilon} 
		\le c_1\left[ \exp\left[-c_2 n \left[\frac{k\epsilon}{K}\right]^2\right]\indic{\epsilon\le \frac{K}{k}} + \exp\left[-c_3 n \left[\frac{k\epsilon}{K}\right]^{\beta}\right]\indic{\epsilon> \frac{K}{k}}\right], 
	\end{align*}
	where the constants $c_1, c_2$ and $c_3$ are as in Lemma \ref{lemma:wasserstein-dist-bound}, while $K,k>0$ as are as in Lemma \ref{lemma:ubsr-t1}.
\end{proposition}
\begin{proof}
	See Section \ref{sec:ubsrconcprop-proof}.
\end{proof}
The specialization to sub-Gaussian r.v.s is immediate. Next, we present an UBSR concentration bound with explicit constants.
\begin{corollary}
	\label{cor:ubsr-subgauss}
	Let $X$ be a sub-Gaussian r.v.  with parameter $\sigma$, and let $l$ be a utility function as in Proposition \ref{ubsrconcprop}. Then, 
	for every $n\geq 1$ and  $\epsilon$ such that $  \frac{512\sigma}{\sqrt{n}}<\frac{\epsilon k}{K} < \frac{512\sigma}{\sqrt{n}}+16\sigma\sqrt{\myexp}$, we have 
	\begin{align*}
		&\Prob{ \left| \xi_{n,\alpha} - S_{\alpha}(X) \right| > \epsilon} \le \exp\left(- \frac{n}{256\sigma^2\myexp} \left(\frac{k\epsilon}{K}-\frac{512\sigma}{\sqrt{n}}\right)^2\right).
	\end{align*}
\end{corollary}
\begin{proof}
	See Section \ref{sec:ubsr-subgauss-proof}.
\end{proof}	

\subsection{Bounds for Risk Measures of Type \ref{type:two}}
We consider two cases in this subsection. The first case is when the distribution $F$ has bounded support, while the second case is that of sub-Gaussian distributions. Although the first case is subsumed under the case of sub-Gaussian distributions, we still consider it separately as the concentration bound that we obtain is stronger than the one in the sub-Gaussian case.

\subsubsection{Distributions with bounded support}
Our first result concerns the risk estimator \eqref{eq:rho-est} for a Type \ref{type:two} risk measure applied to a distribution with bounded support. 
\begin{theorem}
	\label{thm:t2conc-bdd}
	Suppose $X$ takes values in $[-B_2,B_1]$ a.s. for some $B_{1},B_{2}\geq 0$ such that at least one of $B_{1},B_{2}$ is positive. Let  $\rho:\L\rightarrow \R$ be a risk measure of type \ref{type:two}. For each $n$, let 
	$\rho_n = \rho(F_n)$, where $F$ is the CDF of $X$.
	Then, for every $\epsilon>0$ and $n\geq 1$,  we have 
	\begin{align*}
		&\Prob{\left|  \rho_n - \rho(X)\right| > \epsilon}\le c_1\exp\left(-c_2 n \left[\frac{\epsilon}{L_1\tau^\gamma}\right]^{\frac{2}{\alpha_1}}\right), % \\
	\end{align*} 
where $\tau = \max\left\{\frac{B_1}{K_1},\frac{B_2}{K_2}\right\}$,  $K_{1},K_{2},L_1,\gamma,\alpha_1$ are as in the definition \eqref{ttwodef} of a Type \ref{type:two}\ risk measure,    and $c_1, c_2$ are constants that depend on $B_1, B_2$.
\end{theorem}
\begin{proof}
	See Section \ref{sec:t2conc-bdd-proof}.
\end{proof}
As mentioned earlier, the constants $c_1, c_2$ are not explicitly known, and this lack of knowledge hinders bandit applications. To handle such an application, we next present a concentration bound with explicit constants.
\begin{theorem}
	\label{thm:t2conc-bdd-const}
	Assume that the  conditions of Theorem \ref{thm:t2conc-bdd} hold. Then, for every $n\geq 1$
	and $\epsilon$ such that $  \frac{256(B_1+B_2)}{\sqrt{n}}<\left(\frac{\epsilon}{L_1\tau^{\gamma}}\right)^{\frac{1}{\alpha_1}} < \frac{256(B_1+B_2)}{\sqrt{n}}+8(B_1+B_2)\sqrt{\myexp}$,  we have 
	\begin{align*}
		&\Prob{\left|  \rho_n - \rho(X)\right| > \epsilon}\le \exp\left(-  \frac{n}{64 \myexp(B_1 + B_2)^2}\left( \left(\frac{\epsilon}{L_1\tau^{\gamma}}\right)^{\frac{1}{\alpha_1}}-\frac{256(B_1+B_2)}{\sqrt{n}}\right)^2\right), % \\
	\end{align*} 
	where $L_1,\gamma$ and $ \tau$ are as specified in Theorem \ref{thm:t2conc-bdd}.
\end{theorem}
\begin{proof}
	See Section \ref{sec:t2conc-bdd-const-proof}.
\end{proof}

We now turn our attention to  deriving a concentration bound for the CPT estimator in \eqref{eq:cpt-est} for the case of a bounded r.v. 
To put things in context, in \citet{prashanth2018tac}, the authors derive a concentration bound for the same estimator assuming that the underlying distribution has bounded support, and for this purpose, they employ the DKW theorem. 
Interestingly, we are able to provide a matching bound for the case of distributions with bounded support, using a proof technique that relates the the estimation error $\left| C_n -C(X)\right|$ to the Wasserstein distance between the empirical and true CDF, and this is the content of the proposition below.
\begin{corollary}\textbf{\textit{(CPT concentration for bounded r.v.s)}}
	\label{cor:cpt-bounded}
	Suppose $X$ is a r.v. that assumes values in $[-B_{2},B_{1}]$ a.s., where $B_{1},B_{2}\geq 0$, and at least one of $B_{1},B_{2}$ is positive. 
	Further, assume that the conditions of Lemma \ref{cpttt2lem} on the CPT value defined by \eqref{eq:cpt-val} hold.
	Then, for every $\epsilon>0$ and $n\geq 1$, we have
	\[\Prob{\left| C_n - C(X)\right|>\epsilon}  \le 
	c_1\exp\left(-c_2 n \left[\frac{\epsilon}{L\tau^{1-\alpha}}\right]^{\frac{2}{\alpha}}\right), \]
	where $\tau=\max\left\{B_{1}\frac{K^{+}}{k^{+}},B_{2}\frac{K^{-}}{k^{-}}\right\}$, the constants $L, K^+, k^+, K^-, k^-$ and $\alpha$ are as in Lemma \ref{cpttt2lem}, and $c_1,$ and  $c_2$ are as in Theorem \ref{thm:t2conc-bdd}.
\end{corollary}
\begin{proof}
	See Section \ref{sec:cpt-bounded-proof}.
\end{proof}

It is apparent from the bound above that, given $\delta \in (0,1)$, if the number of samples $n$ is of the order $O\left(\frac{1}{\epsilon^{2/\alpha}}\log\left(\frac{1}{\delta}\right)\right)$, then $\left| C_n - C(X)\right|<\epsilon$ with probability $1-\delta$.

Next, we provide a CPT concentration bound with explicit constants.
\begin{corollary}\textbf{\textit{(CPT concentration for bounded r.v.s with explicit constants)}}
	\label{cor:cpt-bounded-const}
	Under the  conditions of Corollary \ref{cor:cpt-bounded}, for every $n\geq 1$ and $\epsilon$ such that $  \frac{256(B_1+B_2)}{\sqrt{n}}<\left(\frac{\epsilon}{L\tau^{1-\alpha}}\right)^{\frac{1}{\alpha}} < \frac{256(B_1+B_2)}{\sqrt{n}}+8(B_1+B_2)\sqrt{\myexp}$, we have
	\[\Prob{\left| C_n - C(X)\right|>\epsilon}  \le 
	\exp\left(-  \frac{n}{64\myexp(B_1 + B_2)^2}\left( \left(\frac{\epsilon}{L\tau^{1-\alpha}}\right)^{\frac{1}{\alpha}}-\frac{256(B_1+B_2)}{\sqrt{n}}\right)^2\right), \]
	where $\tau,\alpha$ and $L$ are as in Corollary \ref{cor:cpt-bounded}.
\end{corollary}
\begin{proof}
	See Section \ref{sec:cpt-bounded-const-proof}.
\end{proof}
\subsubsection{Distributions satisfying \ref{ass:c1}}

Next, we provide a concentration bound for type \ref{type:two}\ risk measures in the case of r.v.s that satisfy \ref{ass:c1}. 

\begin{theorem}
	\label{thm:t2conc-subGauss}
	Let $X$ be a r.v. with a distribution $F$ that satisfies \ref{ass:c1} for some $\beta>1$, and suppose $\rho:\L\rightarrow \R$ is a risk measure of type \ref{type:two}\ with parameters $\alpha_1, \alpha_2, \alpha_3, L_1, L_2, L_3, K_1, K_2$ as defined in \eqref{ttwodef}. 
	Fix $\tau>0$ and let $\rho_{n,\tau} = \rho(F_n|_{\tau})$.
	Fix $\epsilon>0$ such that 
	\begin{align}\epsilon' = \epsilon - \frac{L_2}{\left(K_1\tau\right)^{\beta-1} \alpha_2\gamma (\beta-1)} \exp\left(-\alpha_2 \gamma \left(K_1\tau\right)^{\beta}\right) - \frac{L_3}{\left(K_2\tau\right)^{\beta-1} \alpha_3\gamma (\beta-1)} \exp\left(-\alpha_3 \gamma \left(K_2\tau\right)^{\beta}\right)
	\label{eprimedef}
	\end{align} is positive.
	Then, for every $n\geq 1$,  we have 
	\begin{align}
		&\Prob{\left|  \rho_{n,\tau} - \rho(X)\right| > \epsilon}\le c_1\exp\left(-c_2 n\left(\frac{\epsilon'}{L_1\tau^\gamma} \right)^{\frac{2}{\alpha_1}}\right),
		\label{cptbd1}% \\
	\end{align} 
	where $c_1$ and $c_2$ are constants that depend on the parameters $\beta,\gamma$ and $\top$ specified in \ref{ass:c1}.
\end{theorem}
\begin{proof}
	See Section \ref{sec:t2conc-subGauss-proof}.
\end{proof}	

Next, we apply Theorem \ref{thm:t2conc-subGauss} to obtain a CPT concentration result for a r.v. satisfying \ref{ass:c1}. For this result, we consider the CPT value estimator based on truncation defined in \eqref{eq:cpt-est}. The cut-off value for truncation is chosen as a function of the sample size to get an exponential decay in the tail bound. 

\begin{corollary}\textbf{\textit{(CPT concentration)}}
	\label{cor:cpt-subgauss}
	Assume that the conditions in Lemma \ref{lemma:cpt-t2} hold, and suppose  that $X$ satisfies \ref{ass:c1} for some $\beta>1$.
	For each $n\geq 1$, set 
	\[\tau_n = \left((\log n)^{\frac{1}{\beta}} + 1\right)\max\left\{\frac{K^{+}}{k^{+}},\frac{K^{-}}{k^{-}}\right\},\] 
	$c_3(n) = \frac{L(K^+ + K^-) }{\left(\log n\right)^{\frac{\beta-1}{\beta}} \alpha(1-\alpha)(\beta-1)n^{\alpha(1-\alpha)}}$, and form the CPT-value estimate $C_n$ using \eqref{eq:cpt-est}, where $L,K^+,k^+,K^-,k^-$ and $\alpha$ are as in Lemma \ref{lemma:cpt-t2}.  Then,  for every $n\geq 1$ and  $\epsilon > c_3(n)$, we have 
% 	\begin{align*}
% 		\Prob{\left|{C}_n - C(X)\right|> \epsilon}
% 		\le c_1 \exp\left( - c_2 n \left(\frac{\epsilon- c_3(n)}{L(K^++K^-)\max\left\{\frac{K^{+}}{k^{+}},\frac{K^{-}}{k^{-}}\right\}\left(\log n\right)^{\frac{1}{\beta}}}\right)^{\frac{2}{\alpha}}  \right),
% 	\end{align*}
	\begin{align*}
		\Prob{\left|{C}_n - C(X)\right|> \epsilon}
		\le c_1 \exp\left( - c_2 n \left(\frac{\epsilon- c_3(n)}{L(K^++K^-)\tau_{n}^{1-\alpha}}\right)^{\frac{2}{\alpha}}  \right),
	\end{align*}
	where $c_1, c_2$ are constants that depend on the parameters $\beta,\gamma$ and $\top$ specified in \ref{ass:c1}. 
\end{corollary}
\begin{proof}
	See Section \ref{sec:cpt-subgauss-proof}.
\end{proof}

In Proposition 3 of \cite{prashanth2018tac}, the authors provide a $\left[2n \myexp^{-n^\frac{\alpha}{2+\alpha}} 
+ 2 \myexp^{-n^{\frac{\alpha}{2+\alpha}}\left(\frac{\epsilon}{2H}\right)^{\frac{2}{\alpha}}}\right]$ bound for CPT-estimation in the case where the underlying distribution is sub-Gaussian. It is apparent that the bound we obtain with $\beta=2$ in the theorem above is significantly improved in comparison to the bound of \cite{prashanth2018tac}. 

In \cite{bhat2019concentration}, a tail bound for CPT-value estimation is presented for the sub-Gaussian case. In comparison, the bound we have in Corollary \ref{cor:cpt-subgauss} applies to the more general class of distributions satisfying \ref{ass:c1}. More importantly, our tail bound is applicable for all $n\ge 1$, while the bound in \cite{bhat2019concentration} applies only when the number of samples $n$ is sufficiently large.

Our next result is a  variation on Corollary \ref{cor:cpt-subgauss}, where we specify the constants, but under the slighly more restrictive assumption of sub-Gaussianity.

\begin{proposition}
\textbf{\textit{(CPT concentration bound with explicit constants for sub-Gaussian r.v.s)}}
	\label{cor:cpt-subgauss-const}
	Assume that the conditions in Lemma \ref{lemma:cpt-t2} hold, and suppose  that $X$ is sub-Gaussian with parameter $\sigma$. For each $n\geq 1$, choose $\tau_n$ and $c_{3}(n)$ as given in Corollary \ref{cor:cpt-subgauss} by setting $\beta=2$, and form the CPT-value estimate $C_n$ using \eqref{eq:cpt-est}, where $L,K^+,k^+,K^-,k^-$ and $\alpha$ are as in Lemma \ref{lemma:cpt-t2}. Then, 
	for every  $n\geq 1$ and $\epsilon$ such that $\frac{512\sigma}{\sqrt{n}}<\left(\frac{\epsilon - c_3(n)}{c_4(n)}\right)^{\frac{1}{\alpha}}< \frac{512\sigma}{\sqrt{n}}+16\sigma\sqrt{\myexp}$, we have 
	\begin{align*}
		&\Prob{\left|{C}_n - C(X)\right|> \epsilon}
		\le\exp\left( - \frac{n}{256\sigma^2\myexp} \left(\left(\frac{\epsilon- c_3(n)}{c_4(n)}\right)^{\frac{1}{\alpha}} -\frac{512\sigma}{\sqrt{n}}\right)^2  \right),
	\end{align*}
	where 
	$c_4(n)=L(K^++K^-)\tau_{n}^{1-\alpha}$.
\end{proposition}
\begin{proof}
	See Section \ref{sec:cpt-subgauss-const-proof}.
\end{proof}

	%%%%%%%%%%%%%%%%%%%%%%%%%%%%%%%%%%%%%%%%%%%%%%%%%%%%%%%%%%%%%%%%%%%%% 
	%%%%%%%%%%%%%%%%%%%%%%%%%%%%%%%%%%%%%%%%%%%%%%%%%%%%%%%%%%%%%%%%%%%%% 
	%%%%%%%%%%%%%%%%%%%%%%%%%%%%%%%%%%%%%%%%%%%%%%%%%%%%%%%%%%%%%%%%%%%%% 
	\section{Concentration bounds for distributions satisfying \ref{ass:subexp}}
	\label{sec:subexp}
	In this section, we present concentration bounds for estimators of the risk measures considered so far under the assumption that the underlying distribution is sub-exponential. First, we consider risk measures of type \ref{type:one}. 

\subsection{Bounds for Risk Measures of Type \ref{type:one}}
Using Lemma \ref{lemma:wasserstein-dist-bound-subexp}, we provide an analogue of Theorem \ref{thm:t1-conc-bd} for the case of sub-exponential r.v.s below.
\begin{theorem}
	\label{thm:t1-conc-bd-subexp}
	Let $X$ be a r.v. satisfying \ref{ass:subexp} with parameter $c$. Suppose $\rho:\L\rightarrow \R$ is a risk measure of type \ref{type:one}, with parameters $L$ and $\kappa$. Then, for every $n\geq 1$ and $\epsilon$ satisfying $\left(\frac{\epsilon}{L}\right)^{\frac{1}{\kappa}} > \frac{384}{c\sqrt{n}}$, we have 
	\begin{equation}
		\Prob{|\rho_{n}-\rho(X)|>\epsilon}\leq \exp\left(- \frac{n}{\frac{32}{c^2} + \frac{4}{c}\left(\left(\frac{\epsilon}{L}\right)^{\frac{1}{\kappa}}-\frac{384}{c\sqrt{n}}\right)} \left(\left(\frac{\epsilon}{L}\right)^{\frac{1}{\kappa}}-\frac{384}{c\sqrt{n}}\right)^2\right).
	\end{equation}
\end{theorem}
\begin{proof}
	See Section \ref{sec:t1-conc-bd-subexp-proof}.
\end{proof}

The following corollaries regarding concentration of empirical OCE, SRM and UBSR follow in a straightforward fashion using the result above. We omit the proof details.
\begin{corollary}[OCE concentration]
	\label{cor:cvar-subexp}
	Let $X$ be a r.v. satisfying \ref{ass:subexp} with parameter $c$, and let  $\phi:\R\rightarrow \R $ be a $L$-Lipschitz disutility function. Let  $\oce(X)$ be the OCE risk of $X$ as defined in \eqref{ocedef},  and let $\oce_{n}$ be its empirical estimate as in \eqref{ocest}. 
	Then, for every $n\ge 1$ and every $\epsilon$ satisfying $\frac{\epsilon}{L} > \frac{384}{c\sqrt{n}}$, we have
	\begin{align*}
		&\Prob{ \left| \oce_{n} - \oce(X) \right| > \epsilon} 
		\le \exp\left(- \frac{n}{\frac{32}{c^2} + \frac{4}{c}\left(\frac{\epsilon}{L}-\frac{384}{c\sqrt{n}}\right)} \left(\frac{\epsilon}{L}-\frac{384}{c\sqrt{n}}\right)^2\right).
	\end{align*}
\end{corollary}

Recall that CVaR is an OCE risk satisfying \eqref{cvrlip}. It is therefore clear that the bound above holds for empirical CVaR with $L=(1-\alpha)^{-1}$.

\begin{corollary}[SRM concentration]
	\label{cor:spec-subexp}
Let $X$ be a r.v. satisfying \ref{ass:subexp} with parameter $c$. 
	Let $K>0$ and let $\phi:[0,1]\rightarrow [0,K]$ be a risk spectrum.
	Then, for all $n\ge 1$ and every $\epsilon$ satisfying $\frac{\epsilon}{K} > \frac{384}{c\sqrt{n}}$, we have
	\begin{align*}
		&\Prob{ \left| m_{n,\phi} - M_{\phi}(X) \right| > \epsilon} 
		\le \exp\left(- \frac{n}{\frac{32}{c^2} + \frac{4}{c}\left(\frac{\epsilon}{K}-\frac{384}{c\sqrt{n}}\right)} \left(\frac{\epsilon}{K}-\frac{384}{c\sqrt{n}}\right)^2\right).
	\end{align*}
\end{corollary}
\begin{corollary}[UBSR concentration]
	\label{cor:ubsr-subexp}
Let $X$ be a r.v. satisfying \ref{ass:subexp} with parameter $c$.
	Let the utility function $l$ in the definition \eqref{eq:ubsr-def}  of $S_\alpha(X)$ satisfy the assumptions of Lemma \ref{lemma:ubsr-t1}. 
	For each $n\geq 1$, let $\xi_{n,\alpha}$ denote the solution to the constrained problem in \eqref{eq:sr-est}.
	Then, for all $n\ge 1$ and every $\epsilon$ satisfying $\frac{k\epsilon}{K} > \frac{384}{c\sqrt{n}}$, we have
	\begin{align*}
		&\Prob{ \left|\xi_{n,\alpha} - S_{\alpha}(X)\right| > \epsilon} 
		\le \exp\left(- \frac{n}{\frac{32}{c^2} + \frac{4}{c}\left(\frac{k\epsilon}{K}-\frac{384}{c\sqrt{n}}\right)} \left(\frac{k\epsilon}{K}-\frac{384}{c\sqrt{n}}\right)^2\right),
	\end{align*}
	where the constants $K,k>0$ are as in Lemma \ref{lemma:ubsr-t1}.
\end{corollary}

\subsection{Bounds for Risk Measures of Type \ref{type:two}}
We now provide an analogue of Theorem \ref{thm:t2conc-subGauss} for the case of sub-exponential r.v.s below.
\begin{theorem}
	\label{thm:t2conc-subExp}
	Let $X$ be a r.v. having CDF $F$ and  satisfying \ref{ass:subexp} with parameter $c$. Suppose $\rho:\L\rightarrow \R$ is a risk measure of type \ref{type:two}\ with parameters $\alpha_1, \alpha_2, \alpha_3, L_1, L_2, L_3, K_1, K_2$ as defined in \eqref{ttwodef}. 
	Fix $\tau>0$, and let $\rho_n = \rho(F_n|_{\tau})$.
	Fix $\epsilon>0$ such that  $\epsilon' \triangleq \epsilon - \frac{L_2}{c \alpha_2} \exp\left(-\alpha_2 c K_1\tau\right) - \frac{L_3}{c \alpha_3} \exp\left(-\alpha_3 c K_2\tau\right)>0$ and
	$\left(\frac{\epsilon'}{L_1\tau^{\gamma}}\right)^{\frac{1}{\alpha_1}} > \frac{384}{c\sqrt{n}}$.
	Then, for every $n\geq 1$,  we have 
	\begin{align}
		&\Prob{\left|  \rho_n - \rho(X)\right| > \epsilon}\le \exp\left(-\frac{n}{\frac{32}{c^2} + \frac{4}{c}\left(\left(\frac{\epsilon'}{L_1\tau^{\gamma}}\right)^{\frac{1}{\alpha_1}}-\frac{384}{c\sqrt{n}}\right)}\left(\left(\frac{\epsilon'}{L_1\tau^\gamma} \right)^{\frac{1}{\alpha_1}} -\frac{384}{c\sqrt{n}}\right)^2 \right).
		\label{thm45ineq}
		% \\
	\end{align} 
\end{theorem}
\begin{proof}
	See Section \ref{sec:t2conc-subExp-proof}.
\end{proof}
Finally, we provide a concentration bound for CPT estimation, when the underlying distribution is sub-exponential. 
\begin{corollary}\textbf{\textit{(CPT concentration for sub-exponential r.v.s)}}
	\label{cor:cpt-subexp}
	Assume that the conditions of Lemma \ref{lemma:cpt-t2} hold. Let $X$ be a r.v. satisfying \ref{ass:subexp} with parameter $c$. For each $n\geq 1$, 
	set 
	\[\tau_n = \left(\frac{\log n}{c} + 1\right)\max\left\{\frac{K^{+}}{k^{+}},\frac{K^{-}}{k^{-}}\right\},\] 
	and form the CPT-value estimate $C_n$ using \eqref{eq:cpt-est}. Then, 
	\begin{align*}
		&\Prob{\left|{C}_n - C(X)\right|> \epsilon}
		\le\\&  \exp\left( -  \frac{n}{\frac{32}{c^2} + \frac{4}{c}\left(\left(\frac{\epsilon- \frac{(K^{+}+K^{-})L}{c\alpha n^{\alpha}}}{(K^{+}+K^{-})L \tau_n^{1-\alpha}}\right)^{\frac{1}{\alpha}} - \frac{384}{c\sqrt{n}}\right)} \left(\left(\frac{\epsilon-\frac{(K^{+}+K^{-})L}{c\alpha n^{\alpha}}}{(K^{+}+K^{-})L \tau_n^{1-\alpha}}\right)^{\frac{1}{\alpha}} -\frac{384}{c\sqrt{n}}\right)^2  \right)
	\end{align*}
	holds  for every $\epsilon>\frac{(K^{+}+K^{-})L}{c\alpha n^{\alpha}}$ satisfying
	\[\left(\frac{\epsilon- \frac{(K^{+}+K^{-})L}{c\alpha n^{\alpha}}}{(K^{+}+K^{-})L \tau_n^{1-\alpha}}\right)^{\frac{1}{\alpha}} > \frac{384}{c\sqrt{n}}.\] 
\end{corollary}
\begin{proof}
	See Section \ref{sec:cpt-subExp-proof}.
\end{proof}

	%%%%%%%%%%%%%%%%%%%%%%%%%%%%%%%%%%%%%%%%%%%%%%%%%%%%%%%%%%%%%%%%%%%%% 
	%%%%%%%%%%%%%%%%%%%%%%%%%%%%%%%%%%%%%%%%%%%%%%%%%%%%%%%%%%%%%%%%%%%%% 
	%%%%%%%%%%%%%%%%%%%%%%%%%%%%%%%%%%%%%%%%%%%%%%%%%%%%%%%%%%%%%%%%%%%%% 
	\section{Concentration bounds for distributions satisfying \ref{ass:heavy}}
	\label{sec:heavytailed}
	In this section, we consider heavy-tailed distributions that satisfy \ref{ass:heavy}, i.e.,  a higher moment bound $\E\left(|X|^\beta\right) < \top < \infty$ for some $\beta >2$.

We consider risk measures of type \ref{type:one}, and derive concentration bounds for the empirical estimate of such a risk measure. 
Using Lemma \ref{lemma:wasserstein-dist-bound-heavy}, we provide below an analogue of Theorem \ref{thm:t1-conc-bd} for the case of  r.v.s satisfying \ref{ass:heavy}.
\begin{theorem}
	\label{thm:t1-conc-bd-heavy}
	Suppose $X$ is a r.v. that satisfies \ref{ass:heavy} with parameter $\beta$, and $\rho:\L\rightarrow \R$ is a risk measure of type \ref{type:one} with parameters $L$ and $\kappa$. Then, for every $\epsilon>0$, $n\geq 1$ and $\eta\in(0,\beta)$, we have 
	\begin{equation}
		\Prob{|\rho_{n}-\rho(X)|>\epsilon}\leq c_1\left( \exp\left(-c_2 n \left(\frac{\epsilon}{L}\right)^{\frac{2}{\kappa}}\right)\indic{\epsilon\le L} 
		+ n\left(n\left(\frac{\epsilon}{L}\right)^\frac{1}{\kappa}\right)^{-(\beta-\eta)}\indic{\epsilon> L}\right),
	\end{equation}
	where the constants $c_1,c_2$ are as in Lemma \ref{lemma:wasserstein-dist-bound-heavy}.
	%, while the constants $L>0$ and $\kappa>0$ are as in (\ref{tonedef}).
\end{theorem}
\begin{proof}
	See Section \ref{sec:t1-conc-bd-heavy-proof}. 
\end{proof}

The following corollaries on concentration of empirical CVaR, SRM and UBSR follow in a straightforward fashion from the result above, and we omit the proof details.

\begin{corollary}[OCE concentration]
	Suppose $X$ is a r.v. that satisfies \ref{ass:heavy} with parameter $\beta$. Let  $\phi:\R\rightarrow \R $ be a $L$-Lipschitz disutility function. Let  $\oce(X)$ be the OCE risk of $X$ as defined in \eqref{ocedef},  and let $\oce_{n}$ be its empirical estimate as in \eqref{ocest}.
	Then, for every $n\geq 1$, $\epsilon >0$ and $\eta\in(0,\beta)$, we have
	\begin{align*}
		&\Prob{ \left| \oce_{n} - \oce(X) \right| > \epsilon} 
		\le c_1\left[ \exp\left[-c_2 n (\epsilon/L)^2\right]\indic{\epsilon\le L} \right.\\
		&\left.\qquad\qquad\qquad\qquad\qquad\qquad	+ n\left(n\epsilon/L\right)^{-(\beta-\eta)}\indic{\epsilon> L}\right],
	\end{align*}
	where the constants $c_1$ and $c_2$ are as in Lemma \ref{lemma:wasserstein-dist-bound-heavy}.
\end{corollary}

In light of the fact that CVaR is an OCE risk satisfying \eqref{cvrlip}, it is clear that the bound above holds for empirical CVaR with $L=(1-\alpha)^{-1}$.

\begin{corollary}[SRM concentration]
	Suppose $X$ is a r.v. that satisfies \ref{ass:heavy} with parameter $\beta$.
	Let $K>0$ and let $\phi:[0,1]\rightarrow [0,K]$ be a risk spectrum. 
	Then, for every $n\geq 1$, $\epsilon >0$ and $\eta\in(0,\beta)$, we have
	\begin{align*}
		&\Prob{ \left| m_{n,\phi} - M_{\phi}(X) \right| > \epsilon} 
		\le c_1\left[ \exp\left(-c_2 n \left[\frac{\epsilon}{K}\right]^2\right)\indic{\epsilon\le K} 
		+ n\left(n\left[\frac{\epsilon}{K}\right]\right)^{-(\beta-\eta)}\indic{\epsilon> K}\right],
	\end{align*}
	where the constants $c_1$ and $c_2$ are as in Lemma \ref{lemma:wasserstein-dist-bound-heavy}.
\end{corollary}
\begin{corollary}[UBSR concentration]
	Suppose $X$ is a r.v. that satisfies \ref{ass:heavy} with parameter $\beta$.
	Let the utility function $l$ in the definition \eqref{eq:ubsr-def} of $S_\alpha(X)$ satisfy the assumptions in Lemma \ref{lemma:ubsr-t1}.
	For every $n\geq 1$, let $\xi_{n,\alpha}$ denote the solution to the constrained problem in \eqref{eq:sr-est}.
	Then, for every $n\geq 1$, $\epsilon >0$ and $\eta\in(0,\beta)$, we have
	\begin{align*}
		&\Prob{ \left|\xi_{n,\alpha} - S_{\alpha}(X)\right| > \epsilon} 
		\le c_1\left[ \exp\left(-c_2 n \left[\frac{k\epsilon}{K}\right]^2\right)\indic{\epsilon\le \frac{K}{k}} 
		+ n\left(n\left[\frac{k\epsilon}{K}\right]\right)^{-(\beta-\eta)}\indic{\epsilon> \frac{K}{k}}\right],
	\end{align*}
	where the constants $c_1$ and $c_2$ are as in Lemma \ref{lemma:wasserstein-dist-bound-heavy} and $K,k>0$ are as in Lemma \ref{lemma:ubsr-t1}.
\end{corollary}

For small deviations, i.e., $\epsilon \le 1$, the bounds presented above are satisfactory, as the tail decay matches that of a Gaussian r.v. with constant variance. On the other hand, for large $\epsilon$, the second term exhibits polynomial decay. The latter polynomial term is not an artifact of our analysis. Instead, it relates to the rate obtained in  Lemma \ref{lemma:wasserstein-dist-bound-heavy}. It would be an interesting research direction to investigate if the rate improves in the large $\epsilon$ case by employing the truncated estimator \eqref{eq:rho-est-trunc}.   

We have not presented bounds for empirical risk estimates of  type \ref{type:two}\  risk measures, when the underlying distribution satisfies \ref{ass:heavy}. It would be an interesting direction of future research to fill this gap.
%%%%%%%%%%%%%%%%%%%%%%%%%%%%%%%%%%%%%%%%%%%%%%%%%%%%%%%%%%%%%%%%%%%%% 
%%%%%%%%%%%%%%%%%%%%%%%%%%%%%%%%%%%%%%%%%%%%%%%%%%%%%%%%%%%%%%%%%%%%% 

	%%%%%%%%%%%%%%%%%%%%%%%%%%%%%%%%%%%%%%%%%%%%%%%%%%%%%%%%%%%%%%%%%%%%% 
	%%%%%%%%%%%%%%%%%%%%%%%%%%%%%%%%%%%%%%%%%%%%%%%%%%%%%%%%%%%%%%%%%%%%% 
	%%%%%%%%%%%%%%%%%%%%%%%%%%%%%%%%%%%%%%%%%%%%%%%%%%%%%%%%%%%%%%%%%%%%% 
	\section{Application: Risk-sensitive bandits}
	\label{sec:bandits}
	The concentration bounds for \ref{type:one}\ and \ref{type:two}\ risk measures in previous sections open avenues for bandit applications. We illustrate this claim by using the regret minimization framework in a stochastic $K$-armed bandit problem, with an objective based on an abstract risk measure. Our algorithm can be used as a template for risk-sensitive bandits, where the notion of risk could be any of the five risk measures discussed earlier namely, OCE (including CVaR as a special case), spectral risk measure, UBSR, CPT (including DRM as a special case) and RDEU.  

\subsection{Risk-sensitive bandit problem}
We are given $K$ arms with unknown distributions $P_i, i=1,\ldots,K$.
The interaction of the bandit algorithm with the environment proceeds, over $n$ rounds, as follows: (i) At round $t$, select an arm $I_t \in \{ 1,\ldots,K\}$;
(ii) Observe a sample cost from the distribution $P_{I_t}$ corresponding  to the arm $I_t$.

Let $\rho$ be a risk measure, and let $\rho(i)=\rho(P_{i})$ denote the risk associated with arm $i$, for $i=1,\ldots,K$. 
Let $\rho_{*} = \min_{i=1,\ldots,K} \rho(i)$ denote the lowest risk among the $K$ distributions, and $\Delta_i = (\rho(i)  - \rho_*)$ denote the gap in risk values of arm $i$ and that of the best arm.

The classic objective in a bandit problem is to find the arm with the lowest expected value. We consider an alternative formulation, where the goal is to find the arm with the lowest risk. Using the notion of regret, this objective is formalized as follows:
\[R_n = \sum_{i=1}^K \rho(i) T_i(n)  - n \rho_* = \sum_{i=1}^K T_i(n) \Delta_i,\]
where $T_i(n)= \sum_{t=1}^n \indic{I_t=i}$ is the number of pulls of arm $i$ up to time instant $n$. 
The regret definition above is in the spirit of those for CVaR and CPT-sensitive bandits in \cite{galichet2015thesis} and \cite{aditya2016weighted}, respectively. 

\subsection{\ref{type:one}\ risk measures with sub-Gaussian arms}
\label{sec:t1-bandits-subgauss}
We present a straightforward adaptation Risk-LCB of the well-known UCB algorithm \citep{auer2002finite} to handle an objective based on the abstract risk measure $\rho$. The algorithm caters to \ref{type:one}\ risk measures and arms' distributions that are sub-Gaussian with common parameter $\sigma$. 
The relevant concentration bound for the former case is in Theorem \ref{thm:t1-conc-bd-subgauss-const}, which we recall below.
\begin{equation}
		\Prob{|\rho_{m}-\rho(X)|>\epsilon}\leq \exp\left(- \frac{m}{256\sigma^2\myexp} \left(\left(\frac{\epsilon}{L}\right)^{\frac{1}{\kappa}}-\frac{512\sigma}{\sqrt{m}}\right)^2\right),\label{eq:t1-tailbd}
	\end{equation}
where $\rho_m$, which is formed using \eqref{eq:rho-est}, is an $m$-sample estimate of the risk measure $\rho(X)$, and $L$ and $\kappa$ come from the defining inequality (\ref{tonedef}) of a \ref{type:one} risk measure.
The tail bound above holds for $\epsilon$ satisfying the constraint given by 
 \begin{align}  
 \frac{512\sigma}{\sqrt{m}}<\left(\frac{\epsilon}{L}\right)^{\frac{1}{\kappa}} < \frac{512\sigma}{\sqrt{m}}+16\sigma\sqrt{\myexp}.
 \label{eq:eps-constraint}
 \end{align}
 Simple algebraic manipulations show that, for  every $\delta\in(\exp(-m),1)$, $\epsilon$ defined by 
 \[\epsilon= L \left[ \left(\frac{256\sigma^2\myexp \log(\frac{1}{\delta})}{m}\right)^{\frac{1}{2}}  + \frac{512\sigma}{\sqrt{m}}\right]^{\kappa}\] satisfies the constraint in \eqref{eq:eps-constraint}. 
% This constraint, when we translate to a high confidence form would be equivalent to 
% \begin{align}
% \delta \in [\exp(-n),1].\label{eq:deltaconstraint}
% \end{align}
This allows us to rewrite the tail bound for a \ref{type:one} risk measure from Theorem \ref{thm:t1-conc-bd-subgauss-const} in the high confidence form as 
\begin{align}		
\Prob{|\rho_{m}-\rho(X)|\le L \left[ \left(\frac{256\sigma^2\myexp \log(\frac{1}{\delta})}{m}\right)^{\frac{1}{2}}  + \frac{512\sigma}{\sqrt{m}}\right]^{\kappa} }
\geq 1-\delta, \mbox{ for every } \delta\in(\myexp^{-m},1).~ \label{eq:t1-hpb-form}
\end{align}

For the classic bandit setup with a expected value objective, the finite sample analysis provided by \cite{auer2002finite} chose to set $\delta = \frac1{t^4}$ for the $t$th round in a tail bound based on Hoeffding's inequality. This choice of $\delta$ was shown to be good enough to guarantee a sub-linear regret in \citep{auer2002finite}. However, unlike \eqref{eq:eps-constraint}, Hoeffding's inequality, which forms the basis for the UCB definition in a risk-neutral bandit setting, does not have a constraint on $\epsilon$. 
In our setting, for an arm $i$ that is pulled $T_i(t-1)$ times up to round $t$, we require $ \delta\in(\myexp^{-T_i(t-1)},1)$ for using the tail bound \eqref{eq:t1-tailbd}. Setting $\delta=\frac{8}{t^4}$ and using the fact that $T_i(t-1)\le t$, it is easy to see that this choice of $\delta$ satisfies the necessary constraint coming from \eqref{eq:t1-hpb-form}. Under this choice of $\delta$ for the risk-sensitive bandit problem that we consider, in any round $t$ of Risk-LCB, we have the following high-confidence guarantee for any arm $k \in \{1,\ldots,K\}$:
\begin{align}  
&	 \Prob{ \rho(i) \in [ \rho_{i,T_{i}(t-1)} - w_{i,T_{i}(t-1)}, \rho_{i,T_{i}(t-1)} + w_{i,T_{i}(t-1)}] } \ge 1-\frac{8}{t^4}, \textrm{ where} \nonumber\\
	& 		w_{i,T_{i}(t-1)} = L\sigma \left[ \frac{ 32\sqrt{\myexp\log(t)} + 512}{\sqrt{T_{i}(t-1)}}\right]^{\kappa} .\label{eq:lcb-w}
\end{align}	 		
where,
$\rho_{i,T_{i}(t-1)}$ is the estimate of the risk measure for arm $i$ computed using \eqref{eq:rho-est} from $T_i(t-1)$ samples, and $w_{i,T_{i}(t-1)}$ is the confidence width. In arriving at the form for the confidence width, we have ignored a subtractive factor involving $\log 8$ that arises from taking the logarithm of  $\delta=\frac{8}{t^4}$. Ignoring such a factor does not affect the high-confidence guarantee since we have increased the  size of the confidence interval through this simplification. 

\begin{center}
	\fbox{
		\begin{minipage}{0.95\textwidth}
			\textbf{Risk-LCB algorithm}\\[0.8ex]
			Initialization:  Play each arm once. \\
			{\bfseries For  $t=K+1,\dots,n$, repeat}
			\begin{enumerate}
				\item For each arm $i=1,\ldots,K$, define 
				\[\lcb_t(i) = \rho_{i,T_{i}(t-1)} - w_{i,T_{i}(t-1)},\]   
				where $\rho_{i,T_{i}(t-1)}$ is the estimate of the risk measure for arm $i$ computed using \eqref{eq:rho-est} from $T_i(t-1)$ samples, and $w_{i,T_{i}(t-1)}$ is defined in \eqref{eq:lcb-w}.
				\item Play  arm $I_t = \argmin\limits_{i=1,\ldots,K} \lcb_t(i)$.
				\item Observe sample $X_t$ from the distribution $P_{I_t}$ corresponding  to the arm $I_t$.
			\end{enumerate}
		\end{minipage}
	}\\[0.5ex]
\end{center}
The result below bounds the regret of Risk-LCB algorithm, and the proof is a straightforward adaptation of that used to establish the regret bound of the regular UCB algorithm in \cite{auer2002finite}. 
\begin{theorem}
	\label{thm:regret-subgauss}
	Consider a $K$-armed stochastic bandit problem with a risk measure $\rho$ that is  \ref{type:one} with parameters $L$ and $\kappa$. Assume that the arms' distributions are sub-Gaussian with a common parameter $\sigma$.
	Then the expected regret of Risk-LCB at the end of $n\geq 1$ rounds satisfies
	\[ \E (R_n) \le \sum\limits_{\{i:\Delta_i>0\}}\dfrac{\sigma^2(32\sqrt{\myexp\log n}+512)^2 \ (2L)^{\frac{2}{\kappa}}}{\Delta_i^{\frac{2}{\kappa}-1}}  + K\left(1 + \dfrac{8\pi^2}{3} \right)\Delta_i.\]
	Further, $R_n$ also satisfies the following bound that does not scale inversely with the gaps:
	\[ \E (R_n)  \le \left(K\sigma^2 (32\sqrt{\myexp\log n}+512)^2 (2L)^{\frac{2}{\kappa}}
		+ K \Delta_i^{2/\kappa} \left(1+\dfrac{8\pi^2}{3}  \right) \right)^{\frac{\kappa}{2}} n^{\frac{2-\kappa}{2}}.\]
\end{theorem} 
\begin{proof}
	See Section \ref{sec:appendix-lcb-proof}.
\end{proof}

For the case of OCE, SRM and UBSR, we have $\kappa=1$. Thus, the regret bound obtained above matches, up to log factors, that of a classic $K$-armed bandit problem w.r.t. the dependence on the underlying gaps, and the horizon $n$. 

A UCB-type algorithm for optimization of risk measures has been proposed earlier in \cite{cassel2018general}. 
In comparison to the bound in the result above, the regret bound in \cite{cassel2018general} exhibits a sub-optimal dependence on the underlying gaps. In particular, the bound scales inversely with $\min(1,\Delta)$, where $\Delta$ is the smallest gap. In contrast, the bound we derive for a UCB-type algorithm scales inversely with the smallest gap. Our bound is thus  better for problems with gaps bounded below by one.

CVaR optimization has been considered in a bandit setting in the literature. For instance, in \cite{galichet2015thesis}, the authors assume that the underlying arms' distributions have bounded support, and propose a UCB-type algorithm. We relax this assumption, and consider the case of sub-Gaussian distributions for the $K$ arms. 
The tail bounds in \cite{Kolla} and \cite{prashanth2019concentration} do not allow a bandit application, because forming the confidence term (required for UCB-type algorithms) using their bound would require knowledge of the density in a neighborhood of the true VaR. In contrast, the constants in our bounds depend only on the sub-Gaussian parameter $\sigma$, and several classic MAB algorithms (including UCB) assume this information.

\subsection{CPT risk measure with sub-Gaussian arms}
We now consider the case of CPT, which is a prominent \ref{type:two} risk measure, in conjunction with sub-Gaussian arms. 
CPT-based bandits have been considered in \cite{aditya2016weighted} for the case of arms' distributions with bounded support, and an UCB-based algorithm has been proposed therein. For handling the case of CPT-value in a bandit context, we propose a straightforward variant of the Risk-LCB algorithm presented earlier. Our algorithm can be seen as a generalization of the scheme in \cite{aditya2016weighted}, as it can handle arms' distributions that are sub-Gaussian. 

 The overall algorithm follows the template in Risk-LCB with  the only modifications being  to the confidence widths used in defining the LCBs for each arm. 
For defining the confidence widths, we start with the concentration bound in Proposition \ref{cor:cpt-subgauss-const} after recalling that the proposition applies to sub-Gaussian r.v.s with $\beta=2$. To write this bound in the high confidence form, note that
for every $m\geq 1$ and  $\delta\in(\exp(-m),1)$, $\epsilon$ defined by 
\[\epsilon=\left[L(K^++K^-)\tau_{m}^{1-\alpha}\right] \left[ \left(\frac{ 256\sigma^2\myexp \log(\frac{1}{\delta})}{m}\right)^{\frac{1}{2}}  + \frac{512\sigma}{\sqrt{m}}\right]^{{\alpha}}+ \frac{(K^{+}+K^{-})L}{\alpha(1-\alpha) n^{\alpha(1-\alpha)}\sqrt{\log m} }\]
satisfies the constraint $\frac{512\sigma}{\sqrt{m}}<\left(\frac{\epsilon - c_3(m)}{c_4(m)}\right)^{\frac{1}{\alpha}}< \frac{512\sigma}{\sqrt{m}}+16\sigma\sqrt{\myexp}$ appearing in Proposition \ref{cor:cpt-subgauss-const}, where $\tau_{m}, L, K^+,K^-,\alpha,c_{3}(m)$ and $c_{4}(m)$ are as defined in that proposition and $\sigma$ is the sub-Gaussianity parameter. This observation along with the fact that $\frac{1}{\sqrt{\log m}} < 2$  for all $m\ge 2$ allows us to rewrite the bound in Proposition \ref{cor:cpt-subgauss-const} in the high confidence form as follows: for every $\delta\in(\myexp^{-m},1)$
\begin{align}	
&	\mathbb{P}\left(|C_{m}-C(X)|\le \left[L(K^++K^-)\tau_{m}^{1-\alpha}\right] \times\right.\\
&\qquad\qquad\left.\left[ \left(\frac{ 256\sigma^2\myexp \log(\frac{1}{\delta})}{m}\right)^{\frac{1}{2}}  + \frac{512\sigma}{\sqrt{m}}\right]^{{\alpha}}+ \frac{2(K^{+}+K^{-})L}{\alpha(1-\alpha) n^{\alpha(1-\alpha)} }\right) \geq 1-\delta,\label{eq:cpt-highconf}
\end{align}
where $C_n$ is the $m$-sample estimate of the CPT-value $C(X)$ formed using \eqref{eq:cpt-est}. 

Let $X_1,\ldots, X_K$ denote the r.v.s corresponding to arms $1,\ldots,K$, respectively. Assume that $X_i, i=1,\ldots,K$ are sub-Gaussian with parameter $\sigma$.
Then, as in Section \ref{sec:t1-bandits-subgauss}, the choice $\delta=\frac{8}{t^4}$ is contained in $(\exp(-t),1)$, and for this choice, we have the following high-confidence guarantee for any arm $k \in \{1,\ldots,K\}$ at any round $t$ of CPT-LCB:
\begin{align}  
	&	 \Prob{ C(X_i) \in [ C_{i,T_{i}(t-1)} - w_{i,T_{i}(t-1)}, C_{i,T_{i}(t-1)} + w_{i,T_{i}(t-1)}] } \ge 1-\frac{8}{t^4}, \textrm{ where} \nonumber\\
	& 		w_{i,T_{i}(t-1)} =L(K^++K^-)\left[\max\left\{\frac{K^{+}}{k^{+}},\frac{K^{-}}{k^{-}}\right\}\left(\sqrt{\log T_i(t-1)}+1\right)\right]^{1-\alpha}\\
	&\qquad\times\left[ \frac{\sigma(32\sqrt{\myexp\log t} + 512)}{\sqrt{T_i(t-1)}}\right]^{{\alpha}}+ \frac{2(K^{+}+K^{-})L}{\alpha(1-\alpha) T_i(t-1)^{\alpha(1-\alpha)}},~\label{eq:lcb-cpt}
\end{align}	 		
where 
$C_{i,T_{i}(t-1)}$ is the estimate of the CPT-value for arm $i$ computed using \eqref{eq:cpt-est} from $T_i(t-1)$ samples. This high-confidence guarantee can be used to prove the following theorem. 

\begin{theorem}
	\label{thm:cptregret-subgauss}
	Consider a $K$-armed stochastic bandit problem with CPT as the risk measure. Assume that the arms' distributions are sub-Gaussian with common parameter $\sigma$.
	Then the expected regret of Risk-LCB at the end of $n\geq 1$ rounds with the  LCB given by \eqref{eq:lcb-cpt} satisfies 
	\begin{align} 
	\E (R_n) &\le
	\sum\limits_{\{i:\Delta_i>0\}} \dfrac{\bigg[ \tilde c_4 [ (32\sqrt{\myexp\log n}+512) \sigma]^{\alpha} + \tilde c_3\bigg]^{\frac{1}{\alpha \min\left\{\frac{1}{2},1-\alpha\right\}}}}{\Delta_i^{\frac{1}{\alpha \min\left\{\frac{1}{2},1-\alpha\right\}}-1}} + K\left(1 + \dfrac{8\pi^2}{3} \right)\Delta_i,\label{eq:cptregretbound}\\ 
	&\textrm{ and }\\
 \E (R_n)  &\le \left(K \bigg[ \tilde c_4 [ (32\sqrt{\myexp\log n}+512) \sigma]^{\alpha} + \tilde c_3\bigg]^{\frac{1}{\alpha \min\left\{\frac{1}{2},1-\alpha\right\}}}\right.\\
 & \left.+ K \Delta_i^{\frac{1}{\alpha \min\left\{\frac{1}{2},1-\alpha\right\}}} \left(1 + \dfrac{8\pi^2}{3} \right) \right)^{\alpha \min\left\{\frac{1}{2},1-\alpha\right\}} n^{1 - \alpha \min\left\{\frac{1}{2},1-\alpha\right\}},\label{eq:cptregretbound2}
	\end{align}
where $\tilde c_3 = \frac{2L(K^++ K^-)}{\alpha(1-\alpha)  }$ and $\tilde c_4=L(K^++K^-)\left(\max\left\{\frac{K^{+}}{k^{+}},\frac{K^{-}}{k^{-}}\right\}(\sqrt{\log n}+1)\right)^{1-\alpha}$.	
\end{theorem} 
\begin{proof}
    See Section \ref{sec:appendix-cpt-regret-proof}.
\end{proof}

Ignoring log factors, the regret bound above is $\tilde O\left(\sum\limits_i\frac{1}{\Delta_i^{\frac{1}{\alpha\min\left\{\frac{1}{2},1-\alpha\right\}}-1}}\right)$. The parameter $\alpha$ is the \holder exponent of the weight function, and is less than $1$ for the weight function shown in Figure \ref{fig:w}. Thus, the regret bound for CPT is weaker than the one for classical UCB that finds the arm with the best mean. 

In \cite{aditya2016weighted}, the authors provide a regret upper bound for a UCB-type algorithm with CPT-value as the risk measure under the assumption that the underlying arms' distributions have bounded support.
In particular, the authors in \cite{aditya2016weighted} establish a regret upper bound of the order $\tilde O\left(\sum_i\frac{1}{\Delta_i^{\frac{2}{\alpha}-1}}\right)$, and also show that this upper bound cannot be improved as far as the dependence on gaps and the horizon are concerned  through a minimax regret lower bound. We relax this assumption to consider the more general class of sub-Gaussian arms' distributions. For $\alpha < \frac{1}{2}$, our regret bound in Theorem \ref{thm:cptregret-subgauss} matches the bound in \cite{aditya2016weighted}, while for $\alpha > \frac{1}{2}$, our bound is weaker. It would be interesting future research to check if one can obtain an improved regret bound with sub-Gaussian arms for the latter case, or if the lower bound in \cite{aditya2016weighted} is sub-optimal in this case.

	%%%%%%%%%%%%%%%%%%%%%%%%%%%%%%%%%%%%%%%%%%%%%%%%%%%%%%%%%%%%%%%%%%%%% 
	%%%%%%%%%%%%%%%%%%%%%%%%%%%%%%%%%%%%%%%%%%%%%%%%%%%%%%%%%%%%%%%%%%%%% 
	%%%%%%%%%%%%%%%%%%%%%%%%%%%%%%%%%%%%%%%%%%%%%%%%%%%%%%%%%%%%%%%%%%%%% 
	\section{Proofs}
	\label{sec:proofs}
	%%%%%%%%%%%%%%%%%%%%%%%%%%%%%%%%%%%%%%%%%%%%%%%%%%%%%%%%%%%%%%%%%%%%%%%%%%%%%%%%%%%%%%%%%%%%%%%%%%%%%%%%%%%%%%%%%%%%%%%%%%%%%%%%%%
\subsection{Proofs of the claims in Section \ref{sec:prelims}}
\label{sec:prelims-proofs}

\subsubsection{Proof of Lemma \ref{lemma:lipschitz-wasserstein}}
\label{sec:wass-equiv-proof}
\begin{proof}
	The first equality in \eqref{lipwass2} is given by the Kantorovich-Rubinstein theorem  \citep{givensshortt,edwards}. The second equality is given in \cite{vallander}. 
	
	To prove the third inequality in \eqref{lipwass2}, we note that the integral on the left hand side 
	of the third inequality is unchanged if we replace $F_{1}$ and $F_{2}$ by the point-wise maximum and minimum, respectively, of $F_{1}$ and $F_{2}$. Hence, without loss of generality, we may assume that $F_{1}(s)\geq F_{2}(s)$ for all $s\in\R$. The integral in question then reduces to 
	\begin{equation}
		\int_{-\infty}^{\infty}|F_{1}(s)-F_{2}(s)|\mathrm{d}s=\int_{-\infty}^{\infty}(F_{1}(s)-F_{2}(s))\mathrm{d}s=\int_{-\infty}^{\infty}\int_{F_{2}(s)}^{F_{1}(s)}\mathrm{d}\beta\mathrm{d}s.\label{wass1}
	\end{equation}
	It can easily be shown from the definition of the generalized inverse that 
	\begin{align*}\{(\beta,s)\in\R^{2}:F_{2}(s)<\beta<F_{1}(s)\}&\subseteq &\{(\beta,s)\in\R^{2}:F_{1}^{-1}(\beta)\leq s\leq F_{2}^{-1}(\beta)\} \\ &\subseteq & \{(\beta,s)\in\R^{2}:F_{2}(s)\leq \beta \leq F_{1}(s)\}.
	\end{align*}
	This justifies interchanging the order of integration (see Theorem 14.14 of \cite{apostol}) in \eqref{wass1}, which yields 
	\begin{equation}
		\int_{-\infty}^{\infty}|F_{1}(s)-F_{2}(s)|\mathrm{d}s=\int_{0}^{1}\int_{F_{1}^{-1}(\beta)}^{F_{2}^{-1}(\beta)}\mathrm{d}s\mathrm{d}\beta=\int_{0}^{1}[F_{2}^{-1}(\beta)-F_{1}^{-1}(\beta)]\mathrm{d}\beta.
	\end{equation}
	The third inequality in \eqref{lipwass2} now follows by noting that, under our assumption that $F_{1}(s)\geq F_{2}(s)$ for all $s\in\R$, we have $F_{2}^{-1}(\beta)\geq F_{1}^{-1}(\beta)$ for all $\beta\in[0,1]$. 
\end{proof}

\subsubsection{Proof of Lemma \ref{lemma:wasserstein-dist-bound}}
\label{sec:wasserstein-dist-bound-c1-proof}
\begin{proof}
	The lemma follows directly by applying case (1) of Theorem 2 in \citep{fournier2015rate} to the r.v. $X$.
\end{proof}

\subsubsection{Proof of Lemma \ref{lemma:wasserstein-dist-bound-subgauss}}
\label{sec:wasserstein-dist-bound-subgauss}
We first state and prove a variation of McDiarmid's inequality, which will be used subsequently to prove Lemma \ref{lemma:wasserstein-dist-bound-subgauss}.
Let $\X=(X_1,\ldots,X_n)$ denote a vector of $n$ i.i.d. samples from a common distribution, say $F$. For each $i=1,\ldots, n$, let $\X'_{(i)}=(X_1,\ldots,X_{i-1},X'_i,X_{i+1},\ldots,X_n)$, where $X'_i$ denotes an independent copy of $X_{i}$. 
Let $f$ be a real-valued function on $\R^{n}$ satisfying $\E [|f(\X)|]<\infty$. For each $i=1,\ldots,n$,  define
\[D_i=f(\X)-f(\X'_{(i)}).\]
Finally, let $\F_{0}$ denote the trivial $\sigma$-field and,  for each $k=1,\ldots,n$, let $\F_k$ denote the $\sigma$-field generated by the samples  $\{X_i,i\le k\}$

\begin{lemma}\label{lemma:mcdiarmid} Let $n\geq 1$, $f$ and $D_{i}$, $i=1,\ldots,n$, be as above. 
     Suppose there exists $\tilde{\sigma} >0$ such that, for each $i=1,\ldots,n$,  $D_i$  satisfies 
    \begin{align}
            \E \left(|D_i|^k|\F_{i-1}\right) \le  4^{k} \tilde{\sigma}^k k^{k/2}, \forall k\ge 1.
        \label{eq:Di-subgauss}
    \end{align}
    Then, for every $\tilde{\epsilon} \in(0,16n\tilde{\sigma}\sqrt{\myexp})$, we have 
    \begin{align}
	\Prob{f(\X) - \E[ f(\X)] > \tilde{\epsilon}}   &\leq \exp\left(-\frac{\tilde{\epsilon}^2}{256 n \tilde{\sigma}^2 \myexp}\right), \label{eq:hpb1}
\end{align}
where $\myexp$ is  Euler's number.
\end{lemma}
\begin{proof}
To begin, choose $\tilde{\epsilon} \in(0,16n\tilde{\sigma}\sqrt{\myexp})$ and $i\in\{1,\ldots,n\}$. Since $X_{i}$ and $X_{i}^{\prime}$ have the same conditional distribution given $\F_{i-1}$, it  is easy to see that $\E(D_{i}|\F_{i-1})=0.$
%=\E_{X_{i+1},\ldots,X_{n}}[\E_{X_{i},X_{i}^{\prime}}(f(\X)-f(\X^{\prime}_{(i)}))]=0$.  
Let  $c=\tilde{\epsilon}/(128n\tilde{\sigma}^{2}\myexp)$, and note that $32c^{2}\tilde{\sigma}^{2}\myexp<1/2$. Using $\myexp^x \le x + \myexp^{x^2}$ and $\E(D_{i}|\F_{i-1})=0$, we obtain
\begin{align}
    \E [\exp(c D_i)|\F_{i-1}] &\le \E\left( c D_i + \exp(c^2 D_i^2)|\F_{i-1} \right)\nonumber \\
    & = \E\left( \exp(c^2 D_i^2)|\F_{i-1} \right) = 1 + \sum_{k \ge 1} \frac{c^{2k} \E (D_i^{2k}|\F_{i-1})}{k!} \\
    &\le 1 + \sum_{k \ge 1} \frac{c^{2k} 4^{2k} \tilde{\sigma}^{2k} (2k)^{k}}{k!}  \\
    &\le 1 + \sum_{k \ge 1} \frac{ (4 c \tilde{\sigma})^{2k} (2k)^{k} }{(k/\myexp)^k} \quad (\textrm{ Using Stirling's approximation } k! \ge (k/\myexp)^k) \\
    & =  \sum_{k\ge 0} (32 c^2 \tilde{\sigma}^2 \myexp)^k = \frac{1}{1-32 c^2 \tilde{\sigma}^2 \myexp} \quad (\textrm{ since } 32 c^2 \tilde{\sigma}^2 \myexp < 1)  \\
    & \le \exp\left(64 c^2 \tilde{\sigma}^2 \myexp\right) \quad (\textrm{ since  } 32 c^2 \tilde{\sigma}^2 \myexp < 1/2), \label{lastineq2}
\end{align}
where we have used the inequality $1/(1-x) \le \myexp^{2x}$ for $0<x<1/2$.

Next, define 
$\Delta_i = \E [ f(\X) \left| \F_k \right.] - \E [ f(\X) \left| \F_{k-1} \right.]$ for each $i=1,\ldots,n$. Fix $i\in\{1,\ldots,n\}$. Since $X_{i}$ and $X_{i}^{\prime}$ are identically distributed and all the samples are independent, we have $\E(f(\X)|\F_{i-1})=\E(f(\X^{\prime}_{(i)})|\F_{i-1})=\E(f(\X^{\prime}_{(i)})|\F_{i})$. As a result, we can write  $\Delta_{i}=\E(D_{i}|\F_{i}).$ Applying Jensen's inequality for conditional expectations now yields 
\[ \E (\exp(c\Delta_i) | \F_{i-1}) =\E[\exp(c\E(D_{i}|\F_{i}))|\F_{i-1}]\le \E[\E(\exp(cD_{i})|\F_{i})|\F_{i-1}]=\E (\exp(c D_i) | \F_{i-1}).\] Combining the above inequality with \eqref{lastineq2} gives
\begin{align}
\E (\exp(c\Delta_i) | \F_{i-1}) &\leq \exp\left(64 c^2 \tilde{\sigma}^2 \myexp\right)
\label{lastineq4}
\end{align}

On noting  that $f(\X) - \E (f(\X)) 	= \sum\limits_{i=1}^{n} \Delta_i$
and using \eqref{lastineq4},
we have
\begin{align}
	\Prob{f(\X) - \E (f(\X)) > \tilde{\epsilon}}  &= \Prob{\sum\limits_{k=1}^n \Delta_k > \tilde{\epsilon}}\\
	&\leq \exp(-c \tilde{\epsilon})\E\left[\exp\left(c\sum\limits_{k=1}^n \Delta_k \right)\right]\\
	&\leq \exp(-c \tilde{\epsilon})\E\left[\exp\left(c\sum\limits_{k=1}^{n-1} \Delta_k \right)\E\left[\exp(c \Delta_n )\middle|\mathcal{F}_{n-1}\right]\right]\\
	& \le \exp(-c \tilde{\epsilon})\E\left[\exp\left(c\sum\limits_{k=1}^{n-1} \Delta_k \right)\exp\left(64 c^2 \tilde{\sigma}^2 \myexp\right)\right]\\
	& \vdots \\
	& \le  \exp(-c \tilde{\epsilon})\exp\left(64n c^2 \tilde{\sigma}^2 \myexp\right)\\
	& = \exp\left(-\frac{\tilde{\epsilon}^2}{256 n \tilde\sigma^2 \myexp}\right),
\end{align}
where the suppressed steps involve successively conditioning over $\mathcal{F}_{n-2}, \mathcal{F}_{n-3}, \ldots,\mathcal{F}_{0}$ and using \eqref{lastineq2} at each step. The final equality comes from substituting the value of $c$. 
\end{proof}
Before proving Lemma \ref{lemma:wasserstein-dist-bound-subgauss}, we provide a standard result on sub-Gaussian r.v.s that establishes bounds on its moments. For the sake of completeness, we prove this result so that the constants in the bound can be inferred easily. 
\begin{lemma}
	\label{lemma:subgaussmoments}
	Suppose a r.v. $X$ is sub-Gaussian with parameter $\sigma$. Then $X$ satisfies 
	\[ \left(\E\left|X\right|^k\right)^{\frac{1}{k}} \le 2 \sigma \sqrt{k}, \ \forall k\ge 1.\]
\end{lemma}
\begin{proof}
	Notice that
	\begin{align*}
		\E\left(\left|X\right|^k\right) &= \int_{0}^{\infty} \Prob{\left|X\right|^k\ge u} du \\
		& = \int_{0}^{\infty} \Prob{\left|X\right|\ge \epsilon} k \epsilon^{k-1} d\epsilon\\
		& \le \int_{0}^{\infty} 2 \exp\left(-\frac{\epsilon^2}{2\sigma^2}\right) k \epsilon^{k-1} d\epsilon\\
		& \le 2^{\frac{k}{2}} \sigma^k k \int_{0}^{\infty} \exp(-s) \  s^{\frac{k}{2}-1} ds\\
		& = 2^{\frac{k}{2}} \sigma^k k\Gamma\left(\frac{k}{2}\right) \le 2^{\frac{k}{2}} \sigma^k k \left(\frac{k}{2}\right)^{\frac{k}{2}} \quad \textrm{(Since } \Gamma(x) \le x^x)\\
		&\le \sigma^k k\left(k\right)^{\frac{k}{2}}.
	\end{align*}
	Hence,
	\[ \left(\E\left|X\right|^k\right)^{\frac{1}{k}} 
	\le \sigma \sqrt{k} \left(k\right)^{\frac{1}{k}}\le 2 \sigma \sqrt{k}.\]
	Hence proved.
	\end{proof}

\begin{proof}[Lemma \ref{lemma:wasserstein-dist-bound-subgauss}]
Choose $n$ and $\epsilon$ as in the lemma, let $f(\X)=W_1(F_n,F)$, and fix $i\in\{1,\ldots,n\}$. On using the triangle inequality for the Wasserstein distance, we obtain
\[ |f(\X) - f(\X'_{(i)})|\leq W_{1}(F_{n},F^{\prime}_{n}) \le \frac{1}{n} | X_i - X'_i|,  \]
where $F^{\prime}_{n}$ denotes the EDF obtained from the sample $\mathcal{X}^{\prime}_{(i)}$, and the last inequality follows from the definition of an EDF. 
Setting $D_i=f(\X) - f(\X'_{(i)})$, using the independence of the samples and  the inequalities $|x-y|^{k}\leq (|x|+|y|)^k\leq 2^{k-1}(|x|^{k}+|y|^{k})$ for $k\geq 1$ (see Fact 2.2.59 in \citet{dsb}), and then applying Lemma \ref{lemma:subgaussmoments} to the  sub-Gaussian r.v.s $X_i,X'_i$, we obtain
\begin{align} \E ( |D_i|^k|\F_{i-1})=\frac{1}{n^{k}}\E(|X_{i}-X_{i}^{\prime}|^{k}) \le \frac{2^k}{n^k} \E (|X_i|^k) \le 4^k  \left(\frac{ \sigma}{n}\right)^k k^{k/2} 
\label{intereqn}
\end{align}
for every $k\geq 1$, where we have also used $\E(|X_{i}|^{k})=\E(|X_{i}^{\prime}|^{k})$. 

The inequalities above show that the assumptions of Lemma \ref{lemma:mcdiarmid} hold with $\tilde{\sigma}=\sigma/n$. Letting $\tilde{\epsilon}=\epsilon-\frac{512\sigma}{\sqrt{n}}$, we see that $\tilde{\epsilon}\in(0,16n\tilde{\sigma}\sqrt{\myexp})$. 
Applying Lemma \ref{lemma:mcdiarmid} now yields
\begin{align}
	\Prob{f(\X) - \E (f(\X)) > \tilde{\epsilon}}  &\le \exp\left(-\frac{\tilde{\epsilon}^2}{256 n (\sigma/n)^2 \myexp}\right) = \exp\left(-\frac{n\tilde{\epsilon}^2}{256  \sigma^2 \myexp}\right).\label{eq:subgauss-centerederror}
\end{align}

To infer the final claim in the lemma statement, we need to bound $\E( f(\X))$. For this purpose, we first specify the bound from Theorem 3.1 of \cite{lei2020convergence} and later specialize to our setting.
	For a r.v. $X$ satisfying $\E\left(|X|^q\right) < \top^q < \infty$ and some $q>p\ge 1$, 
	\begin{align} 
	\E(W_p(F_n,F)) \le c_{p,q} \top n^{-\min\left\{\frac{1}{\max\{2p,1\}},\frac{1}{p}-\frac{1}{q}\right\}} (\log n)^\frac{\zeta}{p}, 
	\label{eq:leibd}
	\end{align}
	where  
	\[\zeta = \begin{cases} 
	2 & \textrm{if }1=q=2p\\
	1 & \textrm{if }``1\ne 2p \textrm{ and } q= \min(\frac{p}{1-p},2p)" \mbox{ or } ``q > 1=2p \textrm{''}\\
	0 & \textrm{else}.
	\end{cases}, \] and $c_{p,q}$ is a constant that depends on $p$ and $q$. 
	Tracing through the proof of the aforementioned theorem, we found that $c_{p,q} = (1+3^p)2^{2p-2}2^{q+1}$.

In our setting, $p=1$. Choosing $q=4$, we obtain $c_{p,q}=128$, $\top=4\sigma$ using Lemma \ref{lemma:subgaussmoments}. Further, $\zeta=0$ for this choice of $p,q$. Thus,	
applying the bound in \eqref{eq:leibd} leads to
\begin{align}
	\E( f(\X)) \le \frac{512\sigma}{\sqrt{n}}.
	\label{eq:wass-expec-bd-subgauss}
\end{align}
The main claim now follows by combining \eqref{eq:subgauss-centerederror} and \eqref{eq:wass-expec-bd-subgauss}, and substituting for $\tilde{\epsilon}$. 
\end{proof}
\subsubsection{Proof of Lemma \ref{lemma:wasserstein-dist-bound-subexp}}
\label{sec:wasserstein-dist-bound-subexp-proof}
For establishing the bound in Lemma \ref{lemma:wasserstein-dist-bound-subexp}, we need to invoke a Wasserstein concentration result for r.v.s satisfying the `Bernstein's condition'. We specify this condition and show that a sub-exponential r.v. satisfies the same.

A r.v. $X$ satisfies the Bernstein's condition if 
	there exist $\sigma, b >0$ such that
	\begin{align}
		\E\left(\left|X\right|^k\right) \le \frac1{2}\sigma^2 (k!) b^{k-2} \textrm{ for\ all\ } k\ge 2.
		\label{eq:bern-cond}
	\end{align}
Note that \eqref{eq:bern-cond} does not require the r.v. to necessarily have mean zero. 
The result below shows that a sub-exponential r.v. satisfies \eqref{eq:bern-cond}
\begin{lemma}
	\label{lemma:subexpequiv}
	Suppose a r.v. $X$  satisfies \ref{ass:subexp} with parameter $c>0$. Then $X$ satisfies the Bernstein's condition \eqref{eq:bern-cond} with parameters $\sigma=\frac{2}{c}$ and $b=\frac{1}{c}$. 
\end{lemma}
\begin{proof}
	Suppose $X$ is sub-exponential. Then, we have
	\begin{align*}
		\E\left(\left|X\right|^k\right) &= \int_{0}^{\infty} \Prob{\left|X\right|^k\ge u} du \\
		& = \int_{0}^{\infty} \Prob{\left|X\right|\ge \epsilon} k \epsilon^{k-1} d\epsilon\\
		& \le \int_{0}^{\infty} 2 \exp(-c\epsilon) k \epsilon^{k-1} d\epsilon\\
		& \le \int_{0}^{\infty} \frac{2k }{c^{k}}\exp(-s)  s^{k-1} ds\\
		& = \frac{2k }{c^{k}}\Gamma(k) = \frac{(2/c)^2}{2} k! \left(\frac{1}{c}\right)^{k-2}.
	\end{align*}
	Thus, $X$ satisfies the Bernstein's condition \eqref{eq:bern-cond}.
\end{proof}

Next, we state a variant of Lemma \ref{lemma:mcdiarmid} for the sub-exponential case, and subsequently prove Lemma \ref{lemma:wasserstein-dist-bound-subexp} by applying this result. The lemma appears as Theorem 5.1 in \cite{lei2020convergence}, but we provide a proof here for the sake of completeness. 

\begin{lemma}\label{lemma:mcdiarmid-subexp} Let $n\geq 1$, $f$ and $D_{i}$, $i=1,\ldots,n$, be as defined above Lemma \ref{lemma:mcdiarmid}. 
	Suppose there exist $\tilde{\sigma}, \tilde{b}>0$ such that, for each $i=1,\dots,n$,  $D_i$  satisfies 
	\begin{align}
	 \E\left(\left.|D_{i}|^k\right|\F_{i-1}\right) \le \frac1{2}\tilde\sigma^2 (k!) \tilde{b}^{k-2} \textrm{ for\ all\ } k\ge 2,
	 \label{bernsteinforD}
	\end{align}
where the $\sigma$-fields $\F_{0},\ldots,\F_{n-1}$ are as defined in subsection \ref{sec:wasserstein-dist-bound-subgauss}.
	Then, for every $\tilde\epsilon>0$,
	we have
	\begin{align}
		\Prob{f(\X) - \E[ f(\X)] > \tilde{\epsilon}}   &\leq \exp\left(-\frac{\tilde{\epsilon}^2}{2 n\tilde \sigma^2 + 2 \tilde\epsilon \tilde b}\right). \label{eq:hpb-subexp}
	\end{align}
\end{lemma}
\begin{proof}
To begin, choose $\tilde{\epsilon}>0$ and $i\in\{1,\ldots,n\}$. Since $X_{i}$ and $X_{i}^{\prime}$ have the same conditional distribution given $\F_{i-1}$, it  is easy to see that $\E(D_{i}|\F_{i-1})=0.$
Let  $c=\frac{\tilde{\epsilon}}{n\tilde{\sigma}^{2}+\tilde{b}\tilde{\epsilon}}$, and note that $\tilde{b}c<1$.  By expanding the exponential and using  $\E(D_{i}|\F_{i-1})=0$, we may write 
\begin{align}
    \E [\exp(c D_i)|\F_{i-1}] &= \E\left(1+ c D_i + \left.\sum_{k=2}^{\infty}\frac{c^{k}D_{i}^{k}}{k!}\right|\F_{i-1} \right)\nonumber \\
    & \le 1 + \sum_{k \ge 2} \frac{c^{k} \E (|D_i|^{k}|\F_{i-1})}{k!} 
    \le 1 +\frac{1}{2} \sum_{k \ge 2} c^{k}\tilde{\sigma}^{2}\tilde{b}^{k-2}  \quad (\textrm{ Using \eqref{bernsteinforD} }) \\
    &= 1 +\frac{1}{2} c^{2}\tilde{\sigma}^{2}\sum_{k \ge 0} (c\tilde{b})^{k} \\
    &\le \exp\left(\frac{1}{2} c^{2}\tilde{\sigma}^{2}\sum_{k \ge 0} (c\tilde{b})^{k} \right) \quad (\textrm{ Using the inequality $1+ax\leq \myexp^{ax}$ for $a,x>0$}) \\
    & =  \exp\left(\frac{1}{2}\frac{c^{2}\tilde{\sigma}^{2}}{(1-c\tilde{b})}\right) \quad (\textrm{ since } c\tilde{b}<1).  \label{lastineqb2}
\end{align}

Next, define 
$\Delta_i = \E [ f(\X) \left| \F_i \right.] - \E [ f(\X) \left| \F_{i-1} \right.]$ for each $i=1,\ldots,n$. Fix $i\in\{1,\ldots,n\}$. Since $X_{i}$ and $X_{i}^{\prime}$ are identically distributed and all the samples are independent, we have
\[\E(f(\X)|\F_{i-1})=\E(f(\X^{\prime}_{(i)})|\F_{i-1})=\E(f(\X^{\prime}_{(i)})|\F_{i}).\] As a result, we can write  $\Delta_{i}=\E(D_{i}|\F_{i}).$ Applying Jensen's inequality for conditional expectations now yields 
\[ \E (\exp(c\Delta_i) | \F_{i-1}) =\E[\exp(c\E(D_{i}|\F_{i}))|\F_{i-1}]\le \E[\E(\exp(cD_{i})|\F_{i})|\F_{i-1}]=\E (\exp(c D_i) | \F_{i-1}).\] Combining the above inequality with \eqref{lastineqb2} gives
\begin{align}
\E (\exp(c\Delta_i) | \F_{i-1}) &\leq \exp\left(\frac{1}{2}\frac{c^{2}\tilde{\sigma}^{2}}{(1-c\tilde{b})}\right)
\label{lastineqb4}
\end{align}

On noting  that $f(\X) - \E (f(\X)) 	= \sum\limits_{i=1}^{n} \Delta_i$
and using \eqref{lastineqb4},
we have
\begin{align}
	\Prob{f(\X) - \E (f(\X)) > \tilde{\epsilon}}  &= \Prob{\sum\limits_{i=1}^n \Delta_i > \tilde{\epsilon}}\\
	&\leq \exp(-c \tilde{\epsilon})\E\left[\exp\left(c\sum\limits_{i=1}^n \Delta_i \right)\right]\\
	&\leq \exp(-c \tilde{\epsilon})\E\left[\exp\left(c\sum\limits_{i=1}^{n-1} \Delta_i \right)\E\left[\exp(c \Delta_n )\middle|\mathcal{F}_{n-1}\right]\right]\\
	& \le \exp(-c \tilde{\epsilon})\E\left[\exp\left(c\sum\limits_{i=1}^{n-1} \Delta_i \right)\exp\left(\frac{1}{2}\frac{c^{2}\tilde{\sigma}^{2}}{(1-c\tilde{b})}\right)\right]\\
	& \vdots \\
	& \le  \exp(-c \tilde{\epsilon})\exp\left(\frac{1}{2}\frac{nc^{2}\tilde{\sigma}^{2}}{(1-c\tilde{b})}\right) = \exp\left(-\frac{\tilde{\epsilon}^2}{2n\tilde{\sigma}^{2}+2\tilde{b}\tilde{\epsilon}}\right),
\end{align}
where the suppressed steps involve successively conditioning over $\mathcal{F}_{n-2}, \mathcal{F}_{n-3}, \ldots,\mathcal{F}_{0}$ and using \eqref{lastineqb4} at each step. The final equality comes from substituting the value of $c$. 
\end{proof}

\begin{proof}[Lemma \ref{lemma:wasserstein-dist-bound-subexp}] Define $n$ and $\epsilon$ as in the lemma, and let $\tilde{\epsilon}=\epsilon- \frac{384}{c\sqrt{n}}$.  Recall the notation $\mathcal{X}$ and $\mathcal{X}_{(i)}^{\prime}$ introduced in subsection \ref{sec:wasserstein-dist-bound-subgauss}.
Define $f$ such that $f(\X)=W_{1}(F_{n},F)$ is the Wasserstein distance between the EDF $F_{n}$ formed from the samples $\mathcal{X}$ and the CDF $F$ of $X$. For each $i=1,\ldots,n$, let $D_{i}=f(\mathcal{X})-f(\mathcal{X}_{(i)}^{\prime})$.

By Lemma \ref{lemma:subexpequiv}, $X$ satisfies Bernstein's condition with $\sigma=2/c$ and $b=1/c$. Consider $i\in\{1,\ldots,n\}$.
Using arguments similar to those employed in deriving \eqref{intereqn} in  the proof of Lemma \ref{lemma:wasserstein-dist-bound-subgauss} along with the Bernstein's condition for  $X_{i}$ and $X_{i}^{\prime}$, we can show that 
\[
\E\left(\left.\left|D_i\right|^k\right|\F_{i-1}\right) \le \frac1{2}\left(\frac{2\sigma}{n}\right)^2 (k!) \left(\frac{2b}{n}\right)^{k-2}
\]
for every $i=1,\ldots,n$ and every $k\geq 2$. 
Invoking Lemma \ref{lemma:mcdiarmid-subexp} with $f(\X)=W_1(F_n,F)$, $\tilde \sigma=\frac{2\sigma}{n}$ and $\tilde b=\frac{2b}{n}$ now yields
	\begin{align} 
	\Prob{ W_{1}(F_n,F) - \E( W_1(F_n,F))> \tilde \epsilon} \le \exp\left(-  \frac{n \tilde \epsilon^2}{8\sigma^2 + 4  b \tilde \epsilon}\right).
	\label{eq:ab23}
	\end{align}
Substituting for $\sigma$ and $b$ gives  
	\begin{align} 
	\Prob{ W_{1}(F_n,F) - \E( W_1(F_n,F))> \tilde\epsilon} \le \exp\left(-  \frac{n \tilde\epsilon^2}{\frac{32}{c^2} + \frac{4}{c}\tilde\epsilon}\right).
	\label{eq:ab2}
	\end{align}
	
	Next, applying the bound in \eqref{eq:leibd} with $p=1, q=4$ leads to $\zeta=0$. From Lemma \ref{lemma:subexpequiv}, we have
	$\E\left(\left|X\right|^4\right) \le \frac{(2/c)^2}{2} 4! \left(\frac{1}{c}\right)^{4-2}\le  \left(\frac{3}{c}\right)^4$, which implies
    $\top=\frac{3}{c}$. Using these values in \eqref{eq:leibd},  we obtain 
	\begin{align} 
	\E( W_1(F_n,F)) \le \frac{384}{c\sqrt{n}}. 
	\label{eq:ab1}
	\end{align}
	The lemma now follows by using the last inequality in \eqref{eq:ab2} and then substituting for $\tilde{\epsilon}$.
\end{proof}

\subsubsection{Proof of Lemma \ref{lemma:wasserstein-dist-bound-heavy}}
\label{sec:wasserstein-dist-bound-heavy-proof}
\begin{proof}
	The lemma follows directly by applying case (3) of Theorem 2 in \cite{fournier2015rate} to the r.v. $X$.
\end{proof}

%%%%%%%%%%%%%%%%%%%%%%%%%%%%%%%%%%%%%%%%%%%%%%%%%%%%%%%%%%%%%%%%%%%%%%%%%%%%%%%%%%%%%%%%%%%%%%%%%%%%%%%%%%%%%%%%%%%%%%%%%%%%%%%%%%
%%%%%%%%%%%%%%%%%%%%%%%%%%%%%%%%%%%%%%%%%%%%%%%%%%%%%%%%%%%%%%%%%%%%%%%%%%%%%%%%%%%%%%%%%%%%%%%%%%%%%%%%%%%%%%%%%%%%%%%%%%%%%%%%%%

\subsection{Proofs of the claims in Section \ref{sec:risk-abstract}}
\label{sec:t1-proofs}

\subsubsection{Proof of Lemma \ref{lemma:oce-t1}}
\label{sec:oce-t1-proof}
\begin{proof}
Choose $\xi\in \R$ arbitrarily and let $g_\xi(x) = L^{-1}[\xi + \phi(x-\xi)]$.
Then,
\begin{align*}
\int_\R g_\xi(x) dF_{X}(x) &= L^{-1}[\xi + \E\{\phi(X-\xi)\}] \triangleq D_{X}(\xi), \textrm{ and }\\
\int_\R g_\xi(x) dF_{Y}(x) &= L^{-1}[\xi + \E\{\phi(Y-\xi)\}] \triangleq D_{Y}(\xi).
\end{align*}
Observing that $g_\xi$ is $1$-Lipschitz in $x$ for every $\xi \in \R$ and using \eqref{lipwass2}, we obtain
\begin{align*}
\left| D_{X}(\xi) - D_Y(\xi) \right| \le W_{1}(F_{X},F_{Y}) \textrm{ for every } \xi \in \R. 
\end{align*}
Choose $m>0$ arbitrarily, and let $\xi_1, \xi_2 \in \R$ be such that
\begin{align*}
D_{X}(\xi_1) \le \inf_\xi D_{X}(\xi) + \frac{1}{m}, \textrm{ and } D_Y(\xi_2) \le \inf_\xi D_Y(\xi) + \frac{1}{m}.
\end{align*} 
Then, we obtain
\begin{align*}
&-W_{1}(F_{X},F_{Y}) - \frac{1}{m} \le D_X(\xi_1) - D_Y(\xi_1) - \frac{1}{m} \le \inf_\xi D_X(\xi) - \inf_\xi D_Y(\xi)\\
&\le D_X(\xi_2) - D_Y(\xi_2) + \frac{1}{m} \le W_{1}(F_{X},F_{Y}) + \frac{1}{m}.
\end{align*}
Since the chain of inequalities above hold for every $m>0$, we conclude that
\begin{align} 
\left|\inf_\xi D_X(\xi) - \inf_\xi D_Y(\xi)\right| \le W_{1}(F_{X},F_{Y}).
\label{eq:t291}
\end{align}
By definition, $\inf_\xi D_X(\xi) = L^{-1}\oce(X)$ and $\inf_\xi D_Y(\xi) = L^{-1}\oce(Y)$. Hence \eqref{ocelip} follows immediately from \eqref{eq:t291}.
\end{proof}

%%%%%%%%%%%%%%%%%%%%%%%%%%%%%%%%%%%%%%%%%%%%%%%%%%%%%%%%%%%%%%%%%%%%%%%%%%%%%%%%%%%%%%%%%%%%%%%%%%%%%%%%%%%%%%%%%%%%%%%%%%%%%%%%%%
%%%%%%%%%%%%%%%%%%%%%%%%%%%%%%%%%%%%%%%%%%%%%%%%%%%%%%%%%%%%%%%%%%%%%%%%%%%%%%%%%%%%%%%%%%%%%%%%%%%%%%%%%%%%%%%%%%%%%%%%%%%%%%%%%%

\subsubsection{Proof of Lemma \ref{lemma:ubsr-t1}}
\label{sec:ubsr-t1-proof}
\begin{proof}
    For convenience, define $f_{X}:\R\rightarrow \R$ by $f_{X}(\xi)=\mathbb{E}[l(X-\xi))]$, and note that $S_{\alpha}(X)=\inf\{\xi\in\R: f_{X}(\xi)\leq \alpha\}.$ Also, since $l$ is nondecreasing, it follows that $f_{X}$ is nonincreasing. Define $f_{Y}$ in an identical fashion. 
    
    Given  $\xi\in\R$, note that 
$|f_{X}(\xi)-f_{Y}(\xi)|=K|\mathbb{E}[K^{-1}l(X-\xi)]-\mathbb{E}[K^{-1}l(Y-\xi)]|$. Since the function $x\mapsto K^{-1}l(x-\xi)$ is $1$-Lipschitz, it follows from Lemma \ref{lemma:lipschitz-wasserstein} and the last equality that 
\begin{equation}
    |f_{X}(\xi)-f_{Y}(\xi)|\leq KW_{1}(F_{X},F_{Y})
    \label{ubsrpf1}
\end{equation}
for every $\xi\in\R$. 

Next, let $\epsilon>0$, and choose $\xi_{1},\xi_{2}\in\R$ such that 
\begin{equation}
S_{\alpha}(X)-\epsilon\leq \xi_{1}<  S_{\alpha}(X) < \xi_{2}<S_{\alpha}(X)+\epsilon.
\label{ubsrpf6}
\end{equation}
It follows from the definition of $S_{\alpha}(X)$ and the nonincreasing nature of $f_{X}$  that $f_{X}(\xi_{1})>\alpha$ while $f_{X}(\xi_{2})\leq \alpha$. Next, define $\xi_{1}^{\prime}=\xi_{1}-\frac{K}{k}W_{1}(F_{X},F_{Y})$ and $\xi_{2}^{\prime}=\xi_{2}+\frac{K}{k}W_{1}(F_{X},F_{Y})$.  Since $\xi_{1}^{\prime}\leq \xi_{1}$, the third assumption on $l$ implies that $l(X-\xi_{1}^{\prime})-l(X-\xi_{1})\geq k(\xi_{1}-\xi_{1}^{\prime})=KW_{1}(F_{X},F_{Y})$ almost surely. Taking expectations gives 
\begin{equation}
f_{X}(\xi_{1}^{\prime})-f_{X}(\xi_{1})\geq KW_{1}(F_{X},F_{Y}).
    \label{ubsrpf2}
\end{equation}
Similarly, it can be shown that 
\begin{equation}
f_{X}(\xi_{2})-f_{X}(\xi_{2}^{\prime})\geq KW_{1}(F_{X},F_{Y}).
    \label{ubsrpf3}
\end{equation}

Applying \eqref{ubsrpf1} and \eqref{ubsrpf2} along with our choice of $\xi_{1}$, we get
\begin{equation}
f_{Y}(\xi_{1}^{\prime})\geq f_{X}(\xi_{1}^{\prime})-KW_{1}(F_{X},F_{Y}) \geq f_{X}(\xi_{1})>\alpha.
\label{ubsrpf4}
\end{equation}
Similarly, applying \eqref{ubsrpf1} and \eqref{ubsrpf3} along with our choice of $\xi_{2}$, we get
\begin{equation}
f_{Y}(\xi_{2}^{\prime})\leq f_{X}(\xi_{2}^{\prime})+KW_{1}(F_{X},F_{Y}) \leq f_{X}(\xi_{2})\leq \alpha.
\label{ubsrpf5}
\end{equation}
The inequalities \eqref{ubsrpf4} and \eqref{ubsrpf5} together imply that $\xi_{1}^{\prime}\leq S_{\alpha}(Y)\leq \xi_{2}^{\prime}$. Substituting for $\xi_{1}^{\prime}$ and $\xi_{2}^{\prime}$ and using the inequalities \eqref{ubsrpf6} gives
\[S_{\alpha}(X)-\frac{K}{k}W_{1}(F_{X},F_{Y})-\epsilon\leq S_{\alpha}(Y)\leq S_{\alpha}(X)+\frac{K}{k}W_{1}(F_{X},F_{Y})+\epsilon.\]
% On noting that $\epsilon>0$ was chosen arbitrarily, we conclude that (\ref{ubrNotice 
that $g_{\mu,\xi}(\cdot)$ is a $1$-Lipschitz function for each $\xi$ and $\mu$. 
Hence, by Lemma \ref{lipwass2},
\eqref{ubrlip} holds. This completes the proof.  
\end{proof}

%%%%%%%%%%%%%%%%%%%%%%%%%%%%%%%%%%%%%%%%%%%%%%%%%%%%%%%%%%%%%%%%%%%%%%%%%%%%%%%%%%%%%%%%%%%%%%%%%%%%%%%%%%%%%%%%%%%%%%%%%%%%%%%%%%
%%%%%%%%%%%%%%%%%%%%%%%%%%%%%%%%%%%%%%%%%%%%%%%%%%%%%%%%%%%%%%%%%%%%%%%%%%%%%%%%%%%%%%%%%%%%%%%%%%%%%%%%%%%%%%%%%%%%%%%%%%%%%%%%%%

\subsubsection{Proof of Lemma \ref{lemma:cpt-t2}}
\label{sec:cpt-t2-proof}

\begin{proof}
Choose  $\tau>0$. Let $X$ and $Y$  be r.v.s having CDFs $F$ and $G$,  respectively. Recall that the CDF of $\hat{Y}=Y\indic{-\tau\leq Y <\tau}$ is $G|_{\tau}$.	Then 
\begin{equation}
C(F)-C(G|_{\tau})=\Delta^{+}-\Delta^{-},\label{cptt2pf1}
\end{equation}
where 
\begin{align}
\Delta^{+} &= \int_{0}^{\infty} w^+\left(\Prob{u^+({X})>z}\right) \mathrm{d}z- \int_{0}^{\infty} w^+\left(\Prob{u^+(\hat{Y})>z}\right) \mathrm{d}z,\ \mbox{and}\label{cptt2pf2}\\
\Delta^{-} &= \int_{0}^{\infty} w^-\left(\Prob{u^-({X})>z}\right) \mathrm{d}z- \int_{0}^{\infty} w^-\left(\Prob{u^-(\hat{Y})>z}\right) \mathrm{d}z.\label{cptt2pf3}
\end{align}
On noting that $\Prob{u^+(\hat{Y})>z}=0$ for $z\geq u^{+}(\tau)$ and $\Prob{u^+(\hat{Y})>z}=\Prob{u^+(Y)>z}$ for $0\leq z<u^{+}(\tau)$, we get 
\begin{align}
\Delta^+ &= \int_0^{u^{+}(\tau)} w^+\left(\Prob{u^+(X)>z}\right) \mathrm{d}z- \int_0^{u^{+}(\tau)} w^+\left(\Prob{u^+(Y)>z}\right) \mathrm{d}z\nonumber \\
& {}+\int_{u^{+}(\tau)}^{\infty} w^+\left(\Prob{u^+(X)>z}\right) \mathrm{d}z.
\label{eq:Deltanplus}
\end{align}
Using the assumption of H\"{o}lder continuity on the weight function $w^{+}$ along with $w^{+}(0)=0$  gives
 \begin{align}
 &\left|\Delta^+\right|  \le L \int_0^{u^{+}(\tau)} |F^+(z)-G^+(z)|^\alpha \mathrm{d}z+ L \int_{u^{+}(\tau)}^{\infty} [1-F^+(z)]^\alpha \mathrm{d}z,\
 \label{eq:cpt-conc-11}
 \end{align}
where $F^+(\cdot)$ and $G^{+}(\cdot)$ are the CDFs of the r.v.s $u^+(X)$ and $u^{+}(Y)$, respectively.

Applying Jensen's inequality to the concave function $x\mapsto x^{\alpha}$ after normalizing the Lebesgue measure on the interval $[0,u^{+}(\tau)]$, we obtain
\begin{align*}
\frac{1}{u^{+}(\tau)}\int_{0}^{u^{+}(\tau)}|F^+(z)-G^+(z)|^{\alpha}\mathrm{d}z \leq & \left[\frac{1}{u^{+}(\tau)}\int_{0}^{u^{+}(\tau)}|F^+(z)-G^+(z)|\mathrm{d}z\right]^{\alpha} \\
\leq & \left[\frac{1}{u^{+}(\tau)}\int_{0}^{\tau}|F(v)-G(v)|(u^{+})^{\prime}(v)\mathrm{d}v\right]^{\alpha} \\
\leq & \left[\frac{K^{+}}{u^{+}(\tau)}\int_{-\infty}^{\infty}|F(v)-G(v)|\mathrm{d}v\right]^{\alpha}. 
\end{align*}
In the second inequality above, which follows by using the substitution $z=u^{+}(v)$, $(u^{+})^{\prime}(v)$ denotes the derivative of $u^{+}(\cdot)$ at $v$. 
Applying the second equality in Lemma \ref{lemma:lipschitz-wasserstein} to the CDFs $F$ and $G$ gives 
\[\int_{0}^{u^{+}(\tau)}|F^+(z)- G^+(z)|^{\alpha}\mathrm{d}z \leq [K^{+}W_{1}(F,G)]^{\alpha} [u^{+}(\tau)]^{1-\alpha}.\]
%Substituting in \eqref{eq:cpt-conc-1} gives $\left|\Delta^+\right|\leq L [K^{+}W_{1}(F,G)]^{\alpha}[u^{+}(\tau)]^{1-\alpha}$. 
Our assumption that the derivative of $u^{+}$ on $[0,\infty)$ takes values in $[k^{+},K^{+}]$ implies that  $k^{+}\tau\leq u^{+}(\tau)\leq K^{+}\tau$. Hence, we get 
\begin{equation}
\int_{0}^{u^{+}(\tau)}|F^+(z)-G^+(z)|^{\alpha}\mathrm{d}z \leq K^{+} [W_{1}(F,G)]^{\alpha}\tau^{1-\alpha}.
\label{cptt2pf4}
\end{equation}
We also see that $F^{+}(z)=\Prob{u^{+}(X)\leq z}\geq \Prob{K^{+}X\leq z}=F(z/K^{+})$. Using this along with $u^{+}(\tau)\geq k^{+}\tau$ in the second integral in 
\eqref{eq:cpt-conc-11} gives 
\begin{equation}
\int_{u^{+}(\tau)}^{\infty}[1-F^+(z)]^\alpha \mathrm{d}z\leq \int_{k^{+}\tau}^{\infty}[1-F(z/K^{+})]^{\alpha}\mathrm{d}z.
\label{cptt2pf6}
\end{equation}
Performing  a simple change of variables in \eqref{cptt2pf6}, and using the outcome along with \eqref{cptt2pf4} in \eqref{eq:cpt-conc-11} gives
\begin{equation}
\left|\Delta^+\right|  \leq LK^{+} [W_{1}(F,G)]^{\alpha}\tau^{1-\alpha}+LK^{+}\int_{\frac{k^{+}}{K^{+}}\tau}^{\infty}[1-F(z)]^{\alpha}\mathrm{d}z.
\label{cptt2pf7}
\end{equation}

It follows from almost identical arguments that 
\begin{equation}
\left|\Delta^{-}\right|\leq LK^{-} [W_{1}(F,G)]^{\alpha}\tau^{1-\alpha}+LK^{-}\int^{-\frac{k^{-}}{K^{-}}\tau}_{-\infty}[F(z)]^{\alpha}\mathrm{d}z.
\label{cptt2pf5}
\end{equation}
Using \eqref{cptt2pf7} and \eqref{cptt2pf5} in \eqref{cptt2pf1} completes the proof. 
\end{proof}
%%%%%%%%%%%%%%%%%%%%%%%%%%%%%%%%%%%%%%%%%%%%%%%%%%%%%%%%%%%%%%%%%%%%%%%%%%%%%%%%%%%%%%%%%%%%%%%%%%%%%%%%%%%%%%%%%%%%%%%%%%%%%%%%%%
%%%%%%%%%%%%%%%%%%%%%%%%%%%%%%%%%%%%%%%%%%%%%%%%%%%%%%%%%%%%%%%%%%%%%%%%%%%%%%%%%%%%%%%%%%%%%%%%%%%%%%%%%%%%%%%%%%%%%%%%%%%%%%%%%%

\subsubsection{Proof of Lemma \ref{lemma:rdeu-t2}}
\label{sec:rdeu-t2-proof}

\begin{proof}
\begin{align*}
V(F) &= \int_{-\infty}^{\infty} u(x) \mathrm{d}(w \circ F)(x)\\
           & = \int_{0}^{\infty} \left(\int_{0}^{u(x)} \mathrm{d}z\right) \mathrm{d}(w \circ F)(x) + \int_{-\infty}^{0} \left(\int_{0}^{u(x)} \mathrm{d}z\right) \mathrm{d}(w \circ F)(x)\\
          &= \int_{0}^{\infty} \left(\int_{0}^{u(x)} \mathrm{d}z\right) \mathrm{d}(w \circ F)(x) - \int_{-\infty}^{0} \left(\int_{0}^{u^{-}(x)} \mathrm{d}z\right) \mathrm{d}(w \circ F)(x)\\
           & = \int_{0}^{\infty} \left(\int_{u^{-1}(z)}^{\infty} \mathrm{d}(w \circ F)(x)\right) \mathrm{d}z - \int_{0}^{\infty} \left(\int_{-\infty}^{(u^{-})^{-1}(z)} \mathrm{d}(w \circ F)(x)\right) \mathrm{d}z\\
           &=\underbrace{\int_{0}^{\infty} \left(1- w\left(F(u^{-1}(z))\right)\right) \mathrm{d}z}_{(A)} 
                 -  \underbrace{\int_{0}^{\infty} w\left(F((u^{-})^{-1}(z))\right) \mathrm{d}z}_{(B)},\stepcounter{equation}\tag{\theequation}\label{eq:t122}
\end{align*}
The first term on the RHS above may be simplified as
\begin{align*}
(A) &= \int_{0}^{\infty} \left(1- w\left(\Prob{X\le u^{-1}(z)}\right)\right) \mathrm{d}z = \int_{0}^{\infty} \left(1- w\left(\Prob{u(X)\le z}\right)\right) \mathrm{d}z\\
        & = \int_{0}^{\infty} w^+\left(\Prob{u(X)> z}\right) \mathrm{d}z = \int_{0}^{\infty} w^+\left(\Prob{u^+(X)> z}\right) \mathrm{d}z,
\end{align*}
where the third equality follows by the definition of $w^+$, and the last equality by observing that $\{x:u(x)>z\}=\{x:u^+(x)>z\}$ for every $z>0$.
	
Along similar lines, we obtain
\begin{align*}
(B) &= \int_{0}^{\infty} w\left(F((u^{-})^{-1}(z))\right) \mathrm{d}z 
= \int_{0}^{\infty} w\left(\Prob{X\le (u^{-})^{-1}(z)}\right) \mathrm{d}z \\& = \int_{0}^{\infty} w^-\left(\Prob{u^-(X)\ge z}\right) \mathrm{d}z
%\\ &  = -\int_{\infty}^{0} w^-\left(\Prob{u^-(X)> y}\right) \mathrm{d}y 
=  \int_{0}^{\infty} w^-\left(\Prob{u^-(X)> z}\right) \mathrm{d}z,
\end{align*}
where we used the fact that $u^{-}$ is decreasing.
%$\{u(X)\le z \le 0\} = \{-u^-(X)\le z\} = \{u^-(X)>-z\}$.
The claim follows.	
	\end{proof}

%%%%%%%%%%%%%%%%%%%%%%%%%%%%%%%%%%%%%%%%%%%%%%%%%%%%%%%%%%%%%%%%%%%%%%%%%%%%%%%%%%%%%%%%%%%%%%%%%%%%%%%%%%%%%%%%%%%%%%%%%%%%%%%%%%
%%%%%%%%%%%%%%%%%%%%%%%%%%%%%%%%%%%%%%%%%%%%%%%%%%%%%%%%%%%%%%%%%%%%%%%%%%%%%%%%%%%%%%%%%%%%%%%%%%%%%%%%%%%%%%%%%%%%%%%%%%%%%%%%%%
\subsection{Proofs of the claims in Section \ref{sec:expec-bounds}}
\label{sec:proofs-expec}
\subsubsection{Proof of Theorem \ref{thm:t1-expec-bd}}
\label{sec:proof-expec-t1}
\begin{proof}
	Let $F$ denote the distribution of $X$. Then, we have
	\begin{align}
		\left| \rho_n - \rho(X)\right| &=  \left| \rho(F_n) - \rho(F)\right| \le L (W_1(F_n,F))^\kappa. \label{eq:e1}
	\end{align}
	Using Theorem 3.1 of \cite{lei2020convergence}, we have
	\begin{align*}
		\E[ W_1(F_n,F)] \le  \frac{2^{\beta+3} \top}{n^{\min(\frac{1}{2},1-\frac{1}{\beta})}}, \textrm{ implying } \\
		\E \left[W_1(F_n,F)^\kappa\right] \le  \left[\E (W_1(F_n,F))\right]^\kappa \le \left(\frac{2^{\beta+3} \top}{n^{\min(\frac{1}{2},1-\frac{1}{\beta})}}\right)^\kappa,
	\end{align*}
	where we used Jensen's inequality to infer $\E( W_1(F_n,F)^\kappa) \le \left[\E (W_1(F_n,F))\right]^\kappa$, since $\kappa \in (0,1]$.
	The main claim follows by substituting the bound obtained above in \eqref{eq:e1}.
\end{proof}

\subsubsection{Proof of Theorem \ref{thm:t2-expec-bd}}
\label{sec:proof-expec-t2}
\begin{proof}
	Since $\rho$ is \ref{type:two}\ measure, we have
	\begin{align}
		\left| \rho(X) - \rho_{n,\tau}\right| = \left| \rho(X) - \rho(F_n|_{\tau})\right| &\leq \underbrace{L_{1} \left(W_1(F,F_n)\right)^{\alpha_{1}} \tau^{\gamma}}_{I_1}+\underbrace{L_{2}\int_{K_{1}\tau}^{\infty}[1-F(z)]^{\alpha_{2}}\mathrm{d}z}_{I_2} \nonumber \\&  {}+ \underbrace{L_{3}\int^{-K_{2}\tau}_{-\infty}[F(z)]^{\alpha_{3}}\mathrm{d}z}_{I_3}.
		\label{eq:t2conc1-expec}
	\end{align} 
	The second term on the RHS above can be bounded as follows:
	\begin{align}
		I_{2}=	L_2\int_{K_{1}\tau}^{\infty}[1-F(z)]^{\alpha_{2}}\mathrm{d}z &\le L_2 \top^{\alpha_2}\int_{K_{1}\tau}^{\infty} \frac{1}{z^{\beta\alpha_2}}\mathrm{d}z 
		= \frac{L_2 \top^{\alpha_2}}{(\beta\alpha_2-1) (K_1\tau)^{\beta\alpha_2-1}},\label{eq:t2-i2-expec}
	\end{align}	
	where we used the fact that $1-F(z) = \Prob{X>z} \le \Prob{ |X|^\beta>z^\beta} \le \frac{\top}{ z^\beta}$, since the r.v. $X$ satisfies \ref{ass:heavyvariant}.

	Along similar lines, the third term on the RHS of \eqref{eq:t2conc1-expec} can be bounded as follows:
	\begin{align}
		I_{3}=	L_3\int_{-\infty}^{-K_2\tau}[F(-z)]^{\alpha_{3}}\mathrm{d}z &\le 
		\frac{L_3 \top^{\alpha_3}}{(\beta\alpha_3-1) (K_2\tau)^{\beta\alpha_3-1}}.\label{eq:t2-i3-expec}
	\end{align}	
	
	As in the proof of Theorem \ref{thm:t1-expec-bd}, we have
	\begin{align}
		\E \left(W_1(F_n,F)^{\alpha_1}\right) \le  \left[\E (W_1(F_n,F))\right]^{\alpha_1} \le 
		\left(\frac{2^{\beta+3} \top}{n^{\min(\frac{1}{2},1-\frac{1}{\beta})}}\right)^{\alpha_1}.\label{eq:t2-i1-expec}
	\end{align}
	The claim follows by taking expectations in \eqref{eq:t2conc1-expec} followed by substitution of the bounds given by \eqref{eq:t2-i2-expec}, \eqref{eq:t2-i3-expec} and \eqref{eq:t2-i1-expec}.
\end{proof}

\subsubsection{Proof of Corollary \ref{cor:cpt-expec-bd}}
\label{sec:proof-expec-cpt}
\begin{proof}
	From Lemma \ref{cpttt2lem}, we know that CPT is \ref{type:two}, with the following parameters:
	\[ L_1 = (K^+ + K^-)L, L_2 = LK^+,L_3 = LK^-, \gamma= 1-\alpha, \alpha_1=\alpha_2=\alpha_3=\alpha, K_1 = \frac{k^+}{K^+}, \textrm{ and }\]	
	$ K_2 = \frac{k^-}{K^-}$.
	Here, $\alpha, L$ are the exponent and constant of the  \holdernosp-continuous weight function in the definition of CPT.

The claim follows by substituting in \eqref{eq:t2-expec-bd} the values above, and the choice of $\tau$ specified in the corollary statement.
\end{proof}

%%%%%%%%%%%%%%%%%%%%%%%%%%%%%%%%%%%%%%%%%%%%%%%%%%%%%%%%%%%%%%%%%%%%%%%%%%%%%%%%%%%%%%%%%%%%%%%%%%%%%%%%%%%%%%%%%%%%%%%%%%%%%%%%%%
%%%%%%%%%%%%%%%%%%%%%%%%%%%%%%%%%%%%%%%%%%%%%%%%%%%%%%%%%%%%%%%%%%%%%%%%%%%%%%%%%%%%%%%%%%%%%%%%%%%%%%%%%%%%%%%%%%%%%%%%%%%%%%%%%%
\subsection{Proofs of the claims in Section \ref{sec:subgauss}}
\label{sec:subgauss-proofs}
\subsubsection{Proof of Theorem \ref{thm:t1-conc-bd}}
\label{sec:t1-conc-bd-proof}

\begin{proof}
Let $F$ denote the CDF of $X$, and $F_{n}$ denote the EDF formed from $n$ independent samples of $X$.  
Consider the event $A=\{W_{1}(F,F_{n})>\left(\dfrac{\epsilon}{L}\right)^{\frac{1}{\kappa}}\}$. Note that, by \eqref{tonedef}, the event $\{|\rho_{n}-\rho(X)|>\epsilon\}$ is contained in $A$. Equation \eqref{c1t1bd} now follows by applying Lemma \ref{lemma:wasserstein-dist-bound} to the event $A$. 
\end{proof}

\subsubsection{Proof of Theorem \ref{thm:t1-conc-bd-subgauss-const}}
\label{sec:t1-conc-bd-subgauss-const-proof}
\begin{proof}
    The proof is similar to that of Theorem \ref{thm:t1-conc-bd}, except that one invokes Lemma \ref{lemma:wasserstein-dist-bound-subgauss} instead of Lemma \ref{lemma:wasserstein-dist-bound}. 
\end{proof}

\subsubsection{Proof of Corollary \ref{cvarconcprop}}
\label{sec:cvarconcprop-proof}

\begin{proof}
	From Lemma \ref{lemma:oce-t1}, we have that OCE is a \ref{type:one}\ risk measure with parameters $L=\frac{1}{1-\alpha}$, and $\kappa=1$. 
The proof now follows by an application of Theorem \ref{thm:t1-conc-bd}.
\end{proof}

\subsubsection{Proof of Corollary \ref{cor:oce-subgauss}}
\label{sec:oce-subgauss-proof}

\begin{proof}
	The proof is similar to that of Corollary \ref{cvarconcprop} except that one invokes Theorem \ref{thm:t1-conc-bd-subgauss-const} instead of Theorem \ref{thm:t1-conc-bd}.
\end{proof}

\subsubsection{Proof of Corollary \ref{specconcprop}}
\label{sec:specconcprop-proof}
\begin{proof}
The result follows by applying Theorem \ref{thm:t1-conc-bd}, after observing from Lemma \ref{lemma:srm-t1} that a spectral risk measure is of type \ref{type:one}\ with parameters $L=K$ and $\kappa=1$.
\end{proof}
\subsubsection{Proof of Corollary \ref{cor:spec-subgauss}}
\label{sec:spec-subgauss-proof}

\begin{proof}
	The proof is similar to that of Corollary \ref{specconcprop} except that one invokes Theorem \ref{thm:t1-conc-bd-subgauss-const} instead of Theorem \ref{thm:t1-conc-bd}.
\end{proof}

\subsubsection{Proof of Corollary \ref{ubsrconcprop}}
\label{sec:ubsrconcprop-proof}

\begin{proof}
		The result follows by applying Theorem \ref{thm:t1-conc-bd}, after observing that UBSR is a \ref{type:one}\ risk measure with parameters $L=K/k$ and $\kappa=1$ (see Lemma \ref{ubsrt1lem}).
\end{proof}

\subsubsection{Proof of Corollary \ref{cor:ubsr-subgauss}}
\label{sec:ubsr-subgauss-proof}

\begin{proof}
	The proof is similar to that of Corollary \ref{ubsrconcprop} except that one invokes Theorem \ref{thm:t1-conc-bd-subgauss-const} instead of Theorem \ref{thm:t1-conc-bd}.
\end{proof}

\subsubsection{Proof of Theorem \ref{thm:t2conc-bdd}}
\label{sec:t2conc-bdd-proof}
\begin{proof}
Recall that the since the r.v. $X$ is bounded in $[-B_2, B_1]$, we have $F(z)=0$, for $z<-B_2$ and $F(z)=1$ for $z>B_1$. Thus, on choosing $\tau = \max\left(\frac{B_1}{K_1},\frac{B_2}{K_2}\right)$, the second and third integrals on the RHS of \eqref{ttwodef} vanish, and the risk measure $\rho$ satisfies
\begin{equation}
\label{eq:t2-bdd1}
\left| \rho_n - \rho(X)\right| \leq L_1 \left(W_1(F,F_n)\right)^{\alpha_1} \tau^{\gamma},
\end{equation} 
where $\alpha_1$, $\gamma$ and $L_1$ are parameters of the \ref{type:two}\ risk measure $\rho$ as in \eqref{ttwodef}.  	
	
Fix $\epsilon>0$ and consider the event $A=\{W_{1}(F,F_{n})> [\epsilon/\{L_1 \tau^\gamma\}]^{1/\alpha_1}\}$, where $\tau$ is specified above. Recall that  a bounded r.v. is sub-Gaussian with parameter $\sigma$ determined by its bounds, and hence satisfies \ref{ass:c1} with $\beta=2$. Hence applying Lemma \ref{lemma:wasserstein-dist-bound} with $\beta=2$, we obtain  
\[\Prob{A} \leq c_1\exp\left(-c_2 n\left(\frac{\epsilon}{L_1\tau^{\gamma}}\right)^{\frac{2}{\alpha_1}}\right),\]
where $c_1,$ and $c_2$ are constants that depend on $B_{1}$ and $B_{2}$. 
By \eqref{eq:t2-bdd1}, the event $A$ contains the event $\{\left|  \rho_n - \rho(X)\right| > \epsilon\}$, 	and the claim follows.
\end{proof}

\subsubsection{Proof of Theorem \ref{thm:t2conc-bdd-const}}
\label{sec:t2conc-bdd-const-proof}
\begin{proof}
	Follows in a similar manner as Theorem \ref{thm:t2conc-bdd}, except that we invoke Lemma \ref{lemma:wasserstein-dist-bound-subgauss} in place of Lemma \ref{lemma:wasserstein-dist-bound}. For this invocation, we have used the fact that the sub-Gaussianity paramter $\sigma=\frac{(B_1+B_2)}{2}$ for a r.v. bounded within $[-B_2, B_1]$.
	\end{proof}
\subsubsection{Proof of Corollary \ref{cor:cpt-bounded}}
\label{sec:cpt-bounded-proof}

\begin{proof}
The result follows in a straightforward fashion by applying Theorem \ref{thm:t2conc-bdd} to CPT-value, which is a \ref{type:two}\ risk measure (see Lemma \ref{cpttt2lem}).
\end{proof}

\subsubsection{Proof of Corollary \ref{cor:cpt-bounded-const}}
\label{sec:cpt-bounded-const-proof}

\begin{proof}
	The result follows in a straightforward fashion by applying Theorem \ref{thm:t2conc-bdd-const}.
\end{proof}

\subsubsection{Proof of Theorem \ref{thm:t2conc-subGauss}}
\label{sec:t2conc-subGauss-proof}
\begin{proof}
	Since $\rho$ is \ref{type:two}\ measure, we have
	\begin{align}
	\left| \rho(X) - \rho_{n,\tau}\right| = \left| \rho(X) - \rho(F_n|_{\tau})\right| &\leq \underbrace{L_{1} \left(W_1(F,F_n)\right)^{\alpha_{1}} \tau^{\gamma}}_{I_1}+\underbrace{L_{2}\int_{K_{1}\tau}^{\infty}[1-F(z)]^{\alpha_{2}}\mathrm{d}z}_{I_2} \nonumber \\&  {}+ \underbrace{L_{3}\int^{-K_{2}\tau}_{-\infty}[F(z)]^{\alpha_{3}}\mathrm{d}z}_{I_3}.
	\label{eq:t2conc1}
	\end{align} 
The second term on the RHS above can be bounded as follows:
\begin{align}
L_2\int_{K_{1}\tau}^{\infty}[1-F(z)]^{\alpha_{2}}\mathrm{d}z &\le L_2 \top^{\alpha_2}\int_{K_{1}\tau}^{\infty} \exp\left(- \alpha_2\gamma z^\beta\right)\mathrm{d}z \nonumber\\
&\le 
L_2\top^{\alpha_2}\int_{K_{1}\tau}^{\infty} \left(\frac{z}{K_1\tau}\right)^{\beta-1}   \exp\left(- \alpha_2\gamma z^\beta\right)\mathrm{d}z   \nonumber\\
&= \frac{L_2}{\left(K_1\tau\right)^{\beta-1} \alpha_2\gamma (\beta-1)} \exp\left(-\alpha_2 \gamma \left(K_1\tau\right)^{\beta}\right),\label{eq:t2-i2}
\end{align}	
where we used the fact that, for $z>0$, we have 
\[1-F(z) = \Prob{X>z} =\Prob{|X|>z}= \Prob{\exp(\gamma |X|^\beta)>\exp(\gamma z^\beta)} \le \top \exp\left(- \gamma z^\beta\right),\] since the r.v. $X$ satisfies \ref{ass:c1}.

Along similar lines, the third term on the RHS of \eqref{eq:t2conc1} can be bounded as follows:
\begin{align}
L_3\int_{-\infty}^{-K_2\tau}[F(z)]^{\alpha_{3}}\mathrm{d}z &\le L_3\top^{\alpha_3}\int_{-\infty}^{-K_2\tau} \exp\left(-\alpha_3\gamma (-z)^\beta\right)\mathrm{d}z \nonumber\\
&
=L_3\top^{\alpha_3}\int_{K_2\tau}^{\infty} \exp\left(-\alpha_3\gamma z^\beta\right)\mathrm{d}z\nonumber \\
&\le 
L_3 \top^{\alpha_3}\int_{K_2\tau}^{\infty} \left(\frac{z}{K_2\tau}\right)^{\beta-1}   \exp\left(-\alpha_3\gamma z^\beta\right)\mathrm{d}z  \nonumber\\
&= 
\frac{L_3}{\left(K_2\tau\right)^{\beta-1} \alpha_3\gamma (\beta-1)} \exp\left(-\alpha_3 \gamma \left(K_2\tau\right)^{\beta}\right),\label{eq:t2-i3}
\end{align}	
where the first inequality holds because, for $z<0$, we have 
\[F(z) = \Prob{X\le z} =\Prob{|X|\ge -z}= \Prob{\exp(\gamma |X|^\beta)\ge\exp(\gamma (-z)^\beta)} \le \top \exp\left(- \gamma (-z)^\beta\right).\] The last inequality above uses Markov's inequality and \ref{ass:c1}.

Using \eqref{eq:t2-i2} and \eqref{eq:t2-i3}, we have
\begin{align}\Prob{\left| \rho(X) - \rho_{n,\tau}\right|>\epsilon}=\Prob{I_{1}+I_{2} + I_3>\epsilon}= \Prob{I_{1}>\epsilon-I_{1} - I_2}\leq \Prob{I_{1}>\epsilon'},
\label{thm38ineq1}
\end{align}
where $\epsilon'$ is as defined in the theorem statement.

Applying Lemma \ref{lemma:wasserstein-dist-bound}, we obtain
\begin{align}
&\Prob{I_1 > \epsilon'} = \Prob{W_1(F,F_n) > \left(\frac{\epsilon'}{L_1 \tau^{\gamma}}\right)^{\frac{1}{\alpha_1}}} \le c_1\exp\left(-\frac{c_2 n(\epsilon')^{2/\alpha_1}}{(L_1\tau^\gamma)^{2/\alpha_1}}\right), % \\
\label{thm38ineq2}
\end{align} 
where $c_1$ and $c_2$ are $\sigma$-dependent constants. 	This completes the proof.
\end{proof}	

%%%%%%%%%%%%%%%%%%%%%%%%%%%%%%%%%%%%%%%%%%%%%%%%%%%%%%%%%%%%%%%%%%%%%%%%%%%%%%%%%%%%%%%%%%%%%%%%%%%%%%%%%%%%%%%%%%%%%%%%%%%%%%%%%%%%%%%%%%%%%%%%%%%%%%%%%%%%%%%%%%%%%%%%%%%%%%%%%%%%%%%%%%%%%%%%%%%%%%%%%%%%%%%%%%%%%%%%%%%%%%%%%%%%%%%%%%%%%%%%%%%%%%%%%%%%%%%%%%%%%%%%%%%%%%%%%%%%%%%%%%%%%%%%%%%%%%%%%%%%%%%%%%%%%%%%%%%%%%%%%%%%%%%%%%%%%%%%%%%%%%%%%%%%%%%%%%%%%%%%%%%%%%%%%%%%%%%%%%%%%%%%%%%%%%%%%%%%%%%%%%%%%%%%%%%%%%%%%%%%%%%%%%%%%%%%%%%%%%%%%%%%%%%%%%%%%%%%%%%%%%%%%%%%%%%%%%%%%%%%%%%%%%%%%%%%%%%%%%%%%%%%%%%%%%%%%%%%%%%%%%%%%%%%%%%%%%%%%%%%%%%%%%%%%%%%%%%%%%%%%%%%%%%%%%%%%%%%%%%%%%%%%%%%%%%%%%%%%%%%%%%%%%%%%%%%%%%%%%%%%%%%%%%%%%%%%%%%%%%%%%%%%%%%%%%%%%%%%%%%%%%%%%%%%%%%%%%%%%%%%%%%%%%%%%%%%%%%%%%%%%%%%%%%%%%%%%%%%%%%%%%%%%%%%%%%%%%%%%%%%%%%%%%%%%%%%%%%%%%%%%%%%%%%%%%%%%%%%%%%%%%%%%%%%%%%%%%%%%%%

\subsubsection{Proof of Corollary \ref{cor:cpt-subgauss}}
\label{sec:cpt-subgauss-proof}
\begin{proof}
Let $n\geq 1$, $\epsilon>0$ and $\tau_{n}$ be chosen as in the corollary. From Lemma \ref{cpttt2lem}, we know that CPT is \ref{type:two}, with the parameters
\[ L_1 = (K^+ + K^-)L, L_2 = LK^+,L_3 = LK^-, \gamma= 1-\alpha, \alpha_1=\alpha_2=\alpha_3=\alpha, K_1 = \frac{k^+}{K^+}, \textrm{ and }\]	$ K_2 = \frac{k^-}{K^-}$,
where $\alpha$ and $ L$ are the exponent and H\"{o}lder    constant, respectively,  of the  \holdernosp-continuous weight function in the definition of CPT. 

Let $\epsilon^{\prime}$ be as defined in \eqref{eprimedef} with $\tau=\tau_{n}$ and the other parameters as given above. 
Note that we may rewrite $\tau_{n}$ as $\tau_{n}=[1+(\log n)^{\frac{1}{\beta}}]\max\{K_{1}^{-1},K_{2}^{-1}\}$. As a result, we have   $K_{i}\tau_{n}\geq (\log n)^{\frac{1}{\beta}}$ for $i=1,2$. Using these inequalities and substituting for various parameters in  \eqref{eprimedef}as above,  we get
\begin{align*}
\epsilon'& \geq\epsilon - \frac{LK^+ \exp\left(-\alpha(1-\alpha)\log n\right)}{\left(K_{1}\tau_n\right)^{\beta-1} \alpha(1-\alpha)(\beta-1)}
-\frac{LK^- \exp\left(-\alpha(1-\alpha)\log n\right)}{\left(K_{2}\tau_n\right)^{\beta-1} \alpha(1-\alpha)(\beta-1)}\\
&\ge \epsilon - \frac{L(K^+ + K^-) }{\left(\log n\right)^{\frac{\beta-1}{\beta}} \alpha(1-\alpha)(\beta-1)n^{\alpha(1-\alpha)}}=\epsilon-c_{3}(n).
\end{align*}
Since  $\epsilon>c_{3}(n)$, we have  $\epsilon^{\prime}>0$, and Theorem \ref{thm:t2conc-subGauss} applies. Using the inequality $\epsilon^{\prime}\geq \epsilon-c_{3}(n)$ in \eqref{cptbd1} and substituting for $L_{1}$, $\gamma$, $\alpha_{1}$ and $\tau$ now  yields  the required bound  in a straightforward fashion.
\end{proof}

\subsubsection{Proof of Proposition \ref{cor:cpt-subgauss-const}}
\label{sec:cpt-subgauss-const-proof}
\begin{proof}
Let $n\geq 1$, $\epsilon>0$ and $\tau_{n}$ be chosen as in the proposition. From Lemma \ref{cpttt2lem}, we know that CPT is \ref{type:two}, with the parameters
\[ L_1 = (K^+ + K^-)L, L_2 = LK^+,L_3 = LK^-, \gamma= 1-\alpha, \alpha_1=\alpha_2=\alpha_3=\alpha, K_1 = \frac{k^+}{K^+}, \textrm{ and }\]	$ K_2 = \frac{k^-}{K^-}$,
where $\alpha$ and $ L$ are the exponent and H\"{o}lder    constant, respectively,  of the  \holdernosp-continuous weight function in the definition of CPT. 

Note that $X$ satisfies \ref{ass:c1} with $\beta=2$. Following the steps leading to \eqref{thm38ineq1} and the equality in \eqref{thm38ineq2} in the proof of Theorem \ref{thm:t2conc-subGauss}, we conclude that 
\begin{align}
\Prob{|C_{n}-C(X)|>\epsilon}\leq \Prob{W_1(F,F_n) > \left(\frac{\epsilon'}{L_1 \tau_{n}^{\gamma}}\right)^{\frac{1}{\alpha_1}}},
\label{prop40ineq1}
\end{align}
where $\epsilon^{\prime }$ is as defined in \eqref{intereqn} with $\tau=\tau_{n}$, $\beta=2$ and the rest of the parameters as defined above.

As in the proof of Corollary \ref{cor:cpt-subgauss}, we can show that $\epsilon^{\prime}\geq \epsilon-c_{3}(n)$. As a result,  \eqref{prop40ineq1} yields 
\begin{align}
\Prob{|C_{n}-C(X)|>\epsilon}\leq \Prob{W_1(F,F_n) > \left(\frac{\epsilon-c_{3}(n)}{L_1 \tau_{n}^{\gamma}}\right)^{\frac{1}{\alpha_1}}}.
\label{prop40ineq2}
\end{align}
On substituting for $L_{1}, \tau_{n}$ and $\gamma$, we also recognise that $L_{1}\tau_{n}^{\gamma}=c_{4}(n)$. The condition on $\epsilon$ given in the proposition implies that the assumptions of Lemma    \ref{lemma:wasserstein-dist-bound-subgauss} hold with $\epsilon$ in that lemma replaced by $[(\epsilon-c_{3}(n))/c_{4}(n)]^{\frac{1}{\alpha}}$. Applying Lemma \ref{lemma:wasserstein-dist-bound-subgauss}  to bound the right hand side in \eqref{prop40ineq2} yields the final result. 
\end{proof}
%%%%%%%%%%%%%%%%%%%%%%%%%%%%%%%%%%%%%%%%%%%%%%%%%%%%%%%%%%%%%%%%%%%%%%%%%%%%%%%%%%%%%%%%%%%%%%%%%%%%%%%%%%%%%%%%%%%%%%%%%%%%%%%%%%
%%%%%%%%%%%%%%%%%%%%%%%%%%%%%%%%%%%%%%%%%%%%%%%%%%%%%%%%%%%%%%%%%%%%%%%%%%%%%%%%%%%%%%%%%%%%%%%%%%%%%%%%%%%%%%%%%%%%%%%%%%%%%%%%%%

\subsection{Proofs of the claims in Section \ref{sec:subexp}}
\label{sec:subexp-proofs}

\subsubsection{Proof of Theorem \ref{thm:t1-conc-bd-subexp}}
\label{sec:t1-conc-bd-subexp-proof}

\begin{proof}
	The proof is identical to  the proof of Theorem \ref{thm:t1-conc-bd} except that  Lemma \ref{lemma:wasserstein-dist-bound-subexp} is used in place of  Lemma \ref{lemma:wasserstein-dist-bound} there. 
\end{proof}

\subsubsection{Proof of Theorem \ref{thm:t2conc-subExp}}
\label{sec:t2conc-subExp-proof}
\begin{proof}
	The proof follows the same development as in the proof of Theorem \ref{thm:t2conc-subGauss}, except that the sub-exponential tail bound \eqref{subexptail} is invoked instead of the sub-Gaussian tail bound \eqref{eq:subgauss-1} for bounding the integrals $I_{2}$ and $I_{3}$ in \eqref{eq:t2conc1}. More precisely, for $z>0$, the sub-exponential tail bound \eqref{subexptail} gives 
	\begin{align}
	    1-F(z)=\Prob{X>z}\leq \exp(-cz).\label{eq:subexptailbd}
	\end{align}
	Using the bound allows us to verify through direct integration that \[I_{2}\leq \frac{L_{2}}{c\alpha_{2}}\exp(-\alpha_{2}cK_{1}\tau).\] 
	Similarly, for $z<0$, the tail bound \eqref{subexptail} gives \[F(z)=\Prob{X<z}\leq \exp(cz),\] 
	which yields 
	\[I_{3}\leq \frac{L_{3}}{c\alpha_{3}}\exp(-\alpha_{3}cK_{2}\tau).\] 
	These bounds on $I_{2}$ and $I_{3}$ lead to \eqref{thm38ineq1} with 
 $\epsilon^{\prime}$ defined as in Theorem \ref{thm:t2conc-subExp}. The equality in \eqref{thm38ineq2} immediately follows. Applying Lemma \ref{lemma:wasserstein-dist-bound-subexp} to bound the probability in \eqref{thm38ineq2} completes the proof. 
\end{proof}

\subsubsection{Proof of Corollary \ref{cor:cpt-subexp}}
\label{sec:cpt-subExp-proof}
\begin{proof}
Let $n\geq 1$, $\epsilon>0$ and $\tau_{n}$ be chosen as in the corollary. From Lemma \ref{cpttt2lem}, we know that CPT is \ref{type:two}, with the parameters
$L_1 = (K^+ + K^-)L, L_2 = LK^+,L_3 = LK^-, \gamma= 1-\alpha, \alpha_1=\alpha_2=\alpha_3=\alpha, K_1 = \frac{k^+}{K^+}$ and 	$ K_2 = \frac{k^-}{K^-}$,
where $\alpha$ and $ L$ are the exponent and H\"{o}lder    constant, respectively,  of the  \holdernosp-continuous weight function in the definition of CPT. 

Let $\epsilon^{\prime}$ be as defined in Theorem \ref{thm:t2conc-subExp} with $\tau=\tau_{n}$ and the other parameters as given above. 
Note that we may rewrite $\tau_{n}$ as $\tau_{n}=[1+c^{-1}\log n]\max\{K_{1}^{-1},K_{2}^{-1}\}$. As a result, we have   $cK_{i}\tau_{n}\geq \log n$ for $i=1,2$. Using these inequalities and substituting for various parameters in the expression for $\epsilon^{\prime}$ in Theorem \ref{thm:t2conc-subExp},  we get
\begin{align*}
\epsilon'& \geq\epsilon - \frac{LK^+}{c\alpha} \exp\left(-\alpha\log n\right)-
\frac{LK^-}{c\alpha} \exp\left(-\alpha\log n\right)
\ge \epsilon - \frac{L(K^+ + K^-) }{c\alpha n^{\alpha}}.
\end{align*}
It follows from the above inequality and our assumptions on $\epsilon$ that $\epsilon^{\prime}$ satisfies the conditions in Theorem \ref{thm:t2conc-subExp}. Applying Theorem \ref{thm:t2conc-subExp} gives $\Prob{|C_{n}-C(X)|>\epsilon}\leq g(\epsilon^{\prime})$, where $g$ represents the right hand side in \eqref{thm45ineq}. It is a simple matter to verify that $g(\cdot)$ is decreasing in its argument. Consequently, $g(\epsilon^{\prime})\leq g\left(\epsilon-\frac{L(K^+ + K^-) }{c\alpha n^{\alpha}}\right)$.  Substituting $\tau=\tau_{n}$ and  the parameters $L_{1},\gamma$ and $\alpha_{1}$ in the expression for $g$ immediately yields  the required bound.
\end{proof}

%%%%%%%%%%%%%%%%%%%%%%%%%%%%%%%%%%%%%%%%%%%%%%%%%%%%%%%%%%%%%%%%%%%%%%%%%%%%%%%%%%%%%%%%%%%%%%%%%%%%%%%%%%%%%%%%%%%%%%%%%%%%%%%%%%
%%%%%%%%%%%%%%%%%%%%%%%%%%%%%%%%%%%%%%%%%%%%%%%%%%%%%%%%%%%%%%%%%%%%%%%%%%%%%%%%%%%%%%%%%%%%%%%%%%%%%%%%%%%%%%%%%%%%%%%%%%%%%%%%%%

\subsection{Proofs of the claims in Section \ref{sec:heavytailed}}
\label{sec:sec:heavytailed-proofs}

\subsubsection{Proof of Theorem \ref{thm:t1-conc-bd-heavy}}
\label{sec:t1-conc-bd-heavy-proof}
\begin{proof}
	The proof is identical to  the proof of Theorem \ref{thm:t1-conc-bd} except that  Lemma \ref{lemma:wasserstein-dist-bound-heavy} is used in place of  Lemma \ref{lemma:wasserstein-dist-bound} there. 
\end{proof}

%%%%%%%%%%%%%%%%%%%%%%%%%%%%%%%%%%%%%%%%%%%%%%%%%%%%%%%%%%%%%%%%%%%%%%%%%%%%%%%%%
%%%%%%%%%%%%%%%%%%%%%%%%%%%%%%%%%%%%%%%%%%%%%%%%%%%%%%%%%%%%%%%%%%%%%%%%%%%%%%%%%
%%%%%%%%%%%%%%%%%%%%%%%%%%%%%%%%%%%%%%%%%%%%%%%%%%%%%%%%%%%%%%%%%%%%%%%%%%%%%%%%%
%%%%%%%%%%%%%%%%%%%%%%%%%%%%%%%%%%%%%%%%%%%%%%%%%%%%%%%%%%%%%%%%%%%%%%%%%%%%%%%%%
%%%%%%%%%%%%%%%%%%%%%%%%%%%%%%%%%%%%%%%%%%%%%%%%%%%%%%%%%%%%%%%%%%%%%%%%%%%%%%%%%

%%%%%%%%%%%%%%%%%%%%%%%%%%%%%%%%%%%%%%%%%%%%%%%%%%%%%%%%%%%%%%%%%%%%%%%%%%%%%%%%
%%%%%%%%%%%%%%%%%%%%%%%%%%%%%%%%%%%%%%%%%%%%%%%%%%%%%%%%%%%%%%%%%%%%%%%%%%%%%%%%
%%%%%%%%%%%%%%%%%%%%%%%%%%%%%%%%%%%%%%%%%%%%%%%%%%%%%%%%%%%%%%%%%%%%%%%%%%%%%%%%
%%%%%%%%%%%%%%%%%%%%%%%%%%%%%%%%%%%%%%%%%%%%%%%%%%%%%%%%%%%%%%%%%%%%%%%%%%%%%%%%
%%%%%%%%%%%%%%%%%%%%%%%%%%%%%%%%%%%%%%%%%%%%%%%%%%%%%%%%%%%%%%%%%%%%%%%%%%%%%%%%

\section{Proof of Section \ref{sec:bandits}}
\label{sec:appendix-bandits}

\subsection{Proof of Theorem \ref{thm:regret-subgauss}}
\label{sec:appendix-lcb-proof}
\begin{proof}
	The proof follows by using arguments analogous to that in the proof of Theorem 1 in \citep{auer2002finite}. 
	
	Let $1$ denote the optimal arm, without loss of generality. Also, we abuse notation slightly by denoting $\rho(P_{i})$ by $\rho(i)$. 
	First, we bound the number of pulls $T_i(n)$ of any suboptimal arm $i\ne 1$.
	Fix a round $t\in \{1,\ldots,n\}$ and suppose that a sub-optimal arm $i$ is pulled in this round.
	Then, we have 
	\begin{align}
		\rho_{i,T_{i}(t-1)} - w_{i,T_{i}(t-1)} \le   \rho_{1,T_{1}(t-1)} - w_{1,T_{1}(t-1)}.
		\label{eq:badarmpull}
	\end{align}
	The LCB-value of arm $i$ can be larger than that of $1$ {\em only if} one of the following three conditions holds:\\
	(1) $\rho_{1,T_{1}(t-1)}$ is outside the confidence interval, that is,
	\begin{align}
		\rho_{1,T_{1}(t-1)}- w_{1,T_{1}(t-1)} & \ge \rho(1), \label{eq:optarmmeanoutside}
	\end{align}
	(2) $\rho_{i,T_{i}(t-1)}$ is outside the confidence interval, that is, 
	\begin{align}
		\rho_{i,T_{i}(t-1)}+ w_{i,T_{i}(t-1)} & \le \rho(i) , \label{eq:karmmeanoutside}
	\end{align}
	(3) Gap $\Delta_i$ is small:
	If we negate both the two conditions above and use \eqref{eq:badarmpull}, then we obtain
	\begin{align}
		&\rho(i) -  2w_{i,T_{i}(t-1)} \le  \rho_{i,T_{i}(t-1)} -  w_{i,T_{i}(t-1)} \le\rho_{1,T_{1}(t-1)} - w_{1,T_{1}(t-1)} \le \rho(1)\nonumber\\
		&\Rightarrow \quad
		\Delta_i < \, \,2w_{i,T_{i}(t-1)}.\label{eq:cond3}  
	\end{align}
	The last condition is equivalent to the following:
	\begin{align} 
	T_{i}(t-1) <  \dfrac{(32\sqrt{\sigma^2\myexp\log (t)}+512\sigma)^2 (2L)^{\frac{2}{\kappa}}}{\Delta_i^{\frac{2}{\kappa}}}.
	\end{align}
	Let $u = \dfrac{(32\sqrt{\sigma^2\myexp\log n}+512\sigma)^2 (2L)^{\frac{2}{\kappa}}}{\Delta_i^{\frac{2}{\kappa}}} + 1$. When $T_{i}(t-1) \ge  u$, i.e., when the condition in \eqref{eq:cond3} does not hold, then either (i) arm $i$ is not pulled at time $t$, or 
	(ii) \eqref{eq:optarmmeanoutside} or \eqref{eq:karmmeanoutside} occurs.
	Thus, we have  
	\begin{align*}
		T_i(n) &= 1 + \sum_{t=K+1}^n \indic{I_t=i} 
		\le u + \sum_{t=u+1}^n \indic{I_t=i; T_i(t-1) \ge u} \\
		&\le u + \sum_{t=u+1}^n \mathbb{I}\left\{\rho_{i,T_{i}(t-1)}-  w_{i,T_{i}(t-1)}  \le\rho_{1,T_{1}(t-1)}- w_{1,T_{1}(t-1)};\ T_i(t-1) \ge u\right\} \\
		&\le u + \sum_{t=1}^\infty \sum_{s=1}^{t-1}\sum_{s_i=u}^{t-1}\indic{ \rho_{i,s_i} -  w_{i,s_i} \le  \rho_{1,s} - w_{1,s}} \\
		&\le u + \sum_{t=1}^\infty \sum_{s=1}^{t-1}\sum_{s_i=u}^{t-1} 
		\mathbb{I}\left\{
		\left( \rho(1)  < \rho_{1,s}- w_{1,s}\right)  \text{ or }  
		\left( \rho(i)  > \rho_{i,s_i} + w_{i,s_i}\right) 
		\text{ occurs}\right\}.
	\end{align*}
	Using Theorem \ref{thm:t1-conc-bd-subgauss-const}, we can bound the probability of occurrence of each of the two events inside the indicator on the RHS of the final display above as follows:
	\begin{align*} 
		&\Prob{ \rho(1)  < \rho_{1,s}- w_{1,s}} \le \frac{8}{t^{4}}, \text{ and } \\
		&\Prob{ \rho(i)  > \rho_{i,s_i} + w_{i,s_i}} \le \frac8{t^4}.
	\end{align*}
	Plugging the bounds on the events above and taking expectations on the inequality for $T_i(n)$ derived above, we obtain
	\begin{align}
		\E[T_i(n)] &\le u + \sum_{t=1}^\infty \sum_{s=1}^{t-1}\sum_{s_i=u}^{t-1} \frac{16}{t^4} 
		\leq u + 16 \sum_{t=1}^\infty \frac{1}{t^{2}} \le u + \dfrac{8\pi^2}{3}. \label{eq:Tibd}
	\end{align}
	The preceding analysis together with the fact that $\E( R_n) = \sum_{i=1}^K \Delta_i \E[T_i(n)]$ leads to the first regret bound presented in the theorem.
	
	For inferring the second bound on the regret, i.e., the bound that does not scale inversely with the gaps, observe that
		\begin{align}
		\E (R_n) 
		&= \sum_{i} \Delta_i \; \E [T_i(n)] \\
		&= \sum_{i} \left( \Delta_i \; \E [T_i(n)]^{\frac{\kappa}{2}}  \right) \left( \E [T_i(n)]^{1-\frac{\kappa}{2}} \right) \\
		&\leq   \left( \sum_i \Delta_i^{2/\kappa} \; \E [T_i(n)] \right)^{\frac{\kappa}{2}} \left( \sum_i \E [T_i(n)] \right)^{1 - \frac{\kappa}{2}} \label{eq:holderapp} \\
		&\le  \left(
		\sum_{i} \Delta_i^{2/\kappa}  \left(\dfrac{(32\sqrt{\sigma^2\myexp\log n}+512\sigma)^2 (2L)^{\frac{2}{\kappa}}}{\Delta_i^{\frac{2}{\kappa}}} 
		 + 1 + \dfrac{8\pi^2}{3} \right) \right)^{\frac{\kappa}{2}} n^{\frac{2-\kappa}{2}} \label{eq:dxs}\\
		&\le \left(K (32\sqrt{\sigma^2\myexp\log n}+512\sigma)^2 (2L)^{\frac{2}{\kappa}}
		+ K \Delta_i^{2/\kappa} \left(1 + \dfrac{8\pi^2}{3} \right) \right)^{\frac{\kappa}{2}} n^{\frac{2-\kappa}{2}},
	\end{align}
	where the inequality in \eqref{eq:holderapp} follows by applying  H\"{o}lder's inequality with the conjugate exponents $2/\kappa$ and $2/(2-\kappa)$, and the inequality in \eqref{eq:dxs} follows from   \eqref{eq:Tibd} and the fact that  $\sum_{i} \E [T_i(n)] = n$.
\end{proof}

\subsection{Proof of Theorem \ref{thm:cptregret-subgauss}}
\label{sec:appendix-cpt-regret-proof}
\begin{proof}
    The initial passage of the proof of Theorem \ref{thm:regret-subgauss} upto \eqref{eq:cond3} holds even in the case of Risk-LCB with CPT as the risk measure. However, there are deviations in simplifying the condition $\Delta_i < \, \,2w_{i,T_{i}(t-1)}$ in \eqref{eq:cond3}, and we specify this condition for the case of CPT risk measure below. 
    \begin{align}  
 &\Delta_i < L(K^++K^-)\left[\max\left\{\frac{K^{+}}{k^{+}},\frac{K^{-}}{k^{-}}\right\}\left(\sqrt{\log T_i(t-1)}+1\right)\right]^{1-\alpha}\\
	&\qquad\times\left[ \frac{\sigma(32\sqrt{\myexp\log t} + 512)}{\sqrt{T_i(t-1)}}\right]^{{\alpha}}+ \frac{2(K^{+}+K^{-})L}{\alpha(1-\alpha) T_i(t-1)^{\alpha(1-\alpha)}}\\
	&\le \frac{L(K^++K^-)}{T_i(t-1)^{\alpha\min\left\{\frac{1}{2},1-\alpha\right\}}}\left[\max\left\{\frac{K^{+}}{k^{+}},\frac{K^{-}}{k^{-}}\right\}\left(\sqrt{\log T_i(t-1)}+1\right)\right]^{1-\alpha}\\
	&\qquad\times\left[ \sigma(32\sqrt{\myexp\log t} + 512)\right]^{{\alpha}}+ \frac{2(K^{+}+K^{-})L}{\alpha(1-\alpha) }.\label{eq:deltai-cpt-bound}
\end{align}	 
    With $\tilde c_3$ and $\tilde c_4$ as defined in the theorem statement, the inequality in \eqref{eq:deltai-cpt-bound} is equivalent to the following:
    \begin{align}
    %%%%%
    &\left(T_{i}(t-1)\right)^{\min\left\{\frac{1}{2},1-\alpha\right\}} \le\\  &\dfrac{\left( L(K^++K^-) \left[\max\left\{\frac{K^{+}}{k^{+}},\frac{K^{-}}{k^{-}}\right\}\left(\sqrt{\log T_i(t-1)}+1\right)\right]^{1-\alpha} [(32\sqrt{\myexp\log t}+512) \sigma]^{\alpha} + \tilde c_3\right)^{\frac{1}{\alpha}}}{\Delta_i^{\frac{1}{\alpha}}}.\label{eq:cond3again} 
	\end{align}
	Then, if a sub-optimal arm $i$ is pulled at time $t$ and \eqref{eq:cond3} is negated, then  either \eqref{eq:optarmmeanoutside} or \eqref{eq:karmmeanoutside} must hold.
	
	Setting $u=\dfrac{\bigg[ \tilde c_4 [ (32\sqrt{\myexp\log n}+512) \sigma]^{\alpha} + \tilde c_3\bigg]^{\frac{1}{\alpha \min\left\{\frac{1}{2},1-\alpha\right\}}}}{\Delta_i^{\frac{1}{\alpha\min\left\{\frac{1}{2},1-\alpha\right\}}}}+1$, and following the  steps leading to \eqref{eq:Tibd} in the proof of  Theorem \ref{thm:regret-subgauss}, we obtain
	\begin{align}
		\E[T_i(n)] &\le u + \dfrac{8\pi^2}{3}. \label{eq:Tibd-T2}
	\end{align}
	The regret bound in \eqref{eq:cptregretbound} follows by recalling that $\E( R_n) = \sum_{i=1}^K \Delta_i \E[T_i(n)]$. 
	
	Using a technique similar to that employed in the proof of the second bound in Theorem \ref{thm:regret-subgauss}, we now prove the regret bound in \eqref{eq:cptregretbound2} that does not scale inversely with the gaps.
		\begin{align}
		&\E (R_n) = \sum_{i} \Delta_i \; \E [T_i(n)] \\
		&= \sum_{i} \left( \Delta_i \; \E [T_i(n)]^{\alpha \min\left\{\frac{1}{2},1-\alpha\right\}}  \right) \left( \E [T_i(n)]^{1-\alpha \min\left\{\frac{1}{2},1-\alpha\right\}} \right) \\
		&\leq   \left( \sum_i \Delta_i^{\frac{1}{\alpha \min\left\{\frac{1}{2},1-\alpha\right\}}} \; \E [T_i(n)] \right)^{\alpha \min\left\{\frac{1}{2},1-\alpha\right\}} \left( \sum_i \E [T_i(n)] \right)^{1 - \alpha \min\left\{\frac{1}{2},1-\alpha\right\}} \label{eq:holderapp2} \\
		&\le \! \left[
		\sum_{i} \Delta_i^{\frac{1}{\alpha \min\left\{\frac{1}{2},1-\alpha\right\}}}  \!\left[\dfrac{\bigg[ \tilde c_4 [ (32\sqrt{\myexp\log n}+512) \sigma]^{\alpha} + \tilde c_3\bigg]^{\frac{1}{\alpha \min\left\{\frac{1}{2},1-\alpha\right\}}}}{\Delta_i^{\frac{1}{\alpha\min\left\{\frac{1}{2},1-\alpha\right\}}}} 
		 \!+ 1 \!+ \dfrac{8\pi^2}{3} \right] \right]^{\alpha \min\left\{\frac{1}{2},1-\alpha\right\}} \\
		 &\qquad\times n^{1 - \alpha \min\left\{\frac{1}{2},1-\alpha\right\}} \label{eq:dxs1}\\
		&\le \!\left[K \bigg[ \tilde c_4 [ (32\sqrt{\myexp\log n}+512) \sigma]^{\alpha} + \tilde c_3\bigg]^{\frac{1}{\alpha \min\left\{\frac{1}{2},1-\alpha\right\}}}
		+ K \Delta_i^{\frac{1}{\alpha \min\left\{\frac{1}{2},1-\alpha\right\}}} \!\left[1 \!+ \dfrac{8\pi^2}{3} \right] \right]^{\alpha \min\left\{\frac{1}{2},1-\alpha\right\}} \\
		&\qquad\times n^{1 - \alpha \min\left\{\frac{1}{2},1-\alpha\right\}},
	\end{align}
	where the inequality in \eqref{eq:holderapp2} follows by applying  H\"{o}lder's inequality with the conjugate exponents $\frac{1}{\alpha \min\left\{\frac{1}{2},1-\alpha\right\}}$ and $\frac{1}{1 - \alpha \min\left\{\frac{1}{2},1-\alpha\right\}}$, and the inequality in \eqref{eq:dxs1} follows from   \eqref{eq:Tibd-T2} and the fact that  $\sum_{i} \E [T_i(n)] = n$.
\end{proof}
	%%%%%%%%%%%%%%%%%%%%%%%%%%%%%%%%%%%%%%%%%%%%%%%%%%%%%%%%%%%%%%%%%%%%% 
	%%%%%%%%%%%%%%%%%%%%%%%%%%%%%%%%%%%%%%%%%%%%%%%%%%%%%%%%%%%%%%%%%%%%% 
	%%%%%%%%%%%%%%%%%%%%%%%%%%%%%%%%%%%%%%%%%%%%%%%%%%%%%%%%%%%%%%%%%%%%% 
	\section{Conclusions}
	\label{sec:conclusions}
	We presented a unified approach to derive concentration bounds for empirical estimates of risk measures. Our approach for deriving concentration bounds involves relating the estimation error to the Wasserstein distance between the true CDF and the  EDF formed from an i.i.d. sample, and then applying recent concentration bounds for the latter. This approach yields concentration bounds for two general categories of risk measures  and two estimation schemes, for a class of distributions which includes sub-Gaussian, sub-exponential, and heavy-tailed distributions. The two categories of risk measures covered by our results contain well known risk measures such as OCE (with CVaR as a special case), spectral risk measures, UBSR, CPT value and RDEU as special cases, while the estimators given in the literature for these risk measures form specific examples of the two estimation schemes covered by our results. Our bounds extend and, in some cases, improve existing bounds for specific risk measures. More importantly, our unified approach contrasts with the case-by-case approaches tried in the literature. We illustrate the usefulness of our bounds by providing an algorithm and the corresponding regret bounds for a stochastic bandit problem involving a risk measure of each category.

\bibliography{refs}
\bibliographystyle{plainnat}

\end{document}